\documentclass[11pt,a4paper,twoside]{article}
\usepackage[latin1]{inputenc}
\usepackage{latexsym}
\usepackage{a4wide}
\usepackage{mathrsfs}   
\usepackage{amsthm}
\usepackage{amssymb}
\usepackage{amsmath}
\usepackage{amsfonts}
\usepackage[textsize=small, disable]{todonotes}
\usepackage{enumerate}

\usepackage{hyperref}

\newcommand{\e}{{\mathrm e}}

\usepackage{fancyhdr}

\newcommand{\Cbold} {{\mathbb C}}

\newcommand{\Hcal}   {\mathcal{H}}

\newcommand{\nin}{\not\in}

\def\1{{\mathchoice {1\mskip-4mu\mathrm l}      
{1\mskip-4mu\mathrm l}
{1\mskip-4.5mu\mathrm l} {1\mskip-5mu\mathrm l}}}

\DeclareMathSymbol{\expect}        {\mathord}{AMSb}{"45}
\DeclareMathSymbol{\expec}        {\mathord}{AMSb}{"45}
\DeclareMathSymbol{\prob}        {\mathord}{AMSb}{"50}
\DeclareMathSymbol{\Ibold}        {\mathord}{AMSb}{"49}
\DeclareMathSymbol{\Nbold}        {\mathord}{AMSb}{"4E}
\DeclareMathSymbol{\Rbold}        {\mathord}{AMSb}{"52}
\DeclareMathSymbol{\Zbold}        {\mathord}{AMSb}{"5A}

\newcommand{\C}{\mathbb C}

\newcommand{\Zd}{\Zbold^d}

\newcommand{\sss}   { \scriptscriptstyle }

\newcommand {\vep}{\varepsilon}

\newcommand{\conn}{\longleftrightarrow}

\newcommand{\ct}[1]     { \stackrel{#1}{\conn} }

\newcommand{\eqarray}   {\begin{eqnarray}}
\newcommand{\enarray}   {\end{eqnarray}}
\newcommand{\lbeq}[1]  {\label{e:#1}}
\newcommand{\refeq}[1] {\eqref{e:#1}}
\newcommand{\eq}{\begin{equation}}
\newcommand{\en}{\end{equation}}
\newcommand{\ben}{\begin{enumerate}}
\newcommand{\een}{\end{enumerate}}
\newcommand{\eqn}[1]{\begin{equation} #1 \end{equation}}
\newcommand{\eqan}[1]{\begin{align} #1 \end{align}}

\newcommand{\nn}{\nonumber}
\newcommand{\nnb}{\nonumber\\}

\renewcommand{\to}{\rightarrow}

\def\Zd{\mathbb{Z}^d}

\def\mA[#1]{ {\bf A}(#1)}
\def\miA[#1]{ {\bf A}^{-1}(#1)}

\def\mD[#1]{{\bf \hat D}(#1)}
\def\mE[#1]{{\bf \hat E}(#1)}
\def\mM[#1]{{\bf \hat M}_1(#1)}
\def\mMa[#1]{{\bf \hat M}_2(#1)}
\def\mE[#1]{{\bf E}_{#1}}
\def\ve[#1]{ {e}_{#1}}
\def\hdx{ \hat D^{(x)}(k)}
\def\Isupx{\mathcal{J}}

\def\v1{{\vec 1}}
\def\mJ{{\bf J}}
\def\mI{{\bf I}}

\newcommand{\dminanimal}{18}
\newcommand{\dmintree}{16}

\newcommand{\dmin}{11}

\def\mPi[#1]{\hat {\bf \Pi}_z(#1)}
\def\mPiwoz[#1]{\hat {\bf \Pi}(#1)}
\def\mPiM[#1]{\hat {\bf \Pi}_{{\sss M}}(#1)}

\def\vXi[#1]{\vec {\hat \Xi}(#1)}
\def\vXiz[#1]{\vec {\hat \Xi}_z(#1)}
\def\vXiM[#1]{\vec {\hat \Xi}_{\sss M}(#1)}
\def\vRM[#1]{\vec {\hat R}_{\sss M}(#1)}
\def\vPsi[#1]{\vec {\hat \Psi}(#1)}
\def\vPsiT[#1]{\vec {\hat \Psi}^T(#1)}
\def\vPsiz[#1]{\vec {\hat \Psi}_z(#1)}
\def\vPsiMT[#1]{\vec {\hat \Psi}^T_{\sss M}(#1)}
\def\vPsiM[#1]{\vec {\hat \Psi}_{\sss M}(#1)}
\def\hatPhiM[#1]{\hat {\Phi}_{\sss M}(#1)}

\def\lowK{\underline {K}_{\sss \Delta F}}

\def\betaaa{\beta_{\sss \mu}}
\def\betaaalow{\underline \beta_{\sss \mu}}
\def\lowaf{\underline \beta_{\sss  \alpha,F}}

\def\upaf{\overline \beta_{\sss  \alpha,F}}
\def\upap{\overline \beta_{\sss  \alpha,\Phi}}
\def\betaap{\beta_{\sss  |\alpha,\Phi|}}
\def\betaRp{\beta_{\sss  |R,\Phi|}}
\def\lowcp{\underline \beta_{\sss  c,\Phi}}
\def\upcp{\overline \beta_{\sss  c,\Phi}}

\def\aaz{\mu_z}
\def\aabz{\bar {\mu}_z}
\def\aa{\mu}
\def\aab{\bar \mu}
\def\aap{\mu_p}
\def\aabp{\bar {\mu}_p}

\def\afz{\alpha_{{\sss F},z}}

\def\apz{\alpha_{{\sss \Phi},z}}

\def\Rfz{R_{{\sss F},z}}

\def\Rpz{R_{{\sss \Phi},z}}

\def\hRpz{\hat {R}_{{\sss \Phi},z}}
\def\hRfz{\hat {R}_{{\sss F},z}}
\def\deltaRfz{\hat {R}_{{\sss F},z}(0;k)}
\def\deltaRpz{\hat {R}_{{\sss \Phi},z}(0;k)}

\def\betadeltaRfzlow{ \underline{\beta}_{\sss \Delta R,F}}
\def\betadeltaRfz{ \beta_{\sss \Delta R,F}}
\def\betadeltaRpz{ \beta_{\sss \Delta R,\Phi}}

\def\cfz{c_{{\sss F},z}}

\def\cpz{c_{{\sss \Phi},z}}

\def\ssss[#1]{{\sss \text{\rm #1}}}
\def\ssc[#1]{{\sss( \text{\rm #1})}}
\def\genC[#1]{\hat {C}_{\mu_z}(#1)}
\def\genG[#1]{\hat {G}_{z}(#1)}
\def\bvtheta[#1]{\vec  {\boldsymbol \theta}_{#1}}

\def\diagRepulsiveLetter{\mathscr}

\def\projIndexsetdir[#1]{{\mathsf {IK} }(#1)}
\def\projIndexsetPoints[#1]{{\mathsf {AB} }(#1)}
\def\projIndexsetNumber[#1]{{\mathsf {NM} }(#1)}
\def\projIndexsetPointsTwo[#1]{{\mathsf {X} }(#1)}

\newcommand{\ii}{{\mathrm i}}

\newtheorem{theorem}{Theorem}[section]

\newtheorem{lemma}[theorem]{Lemma}
\newtheorem{prop}[theorem]{Proposition}

\newtheorem{remark}[theorem]{Remark}
\newtheorem{ass}[theorem]{Assumption}
\newtheorem{definition}[theorem]{Definition}
\numberwithin{equation}{section}
\numberwithin{theorem}{section}

\author{Robert Fitzner\thanks{Department of Mathematics and
        Computer Science, Eindhoven University of Technology,
        5600 MB Eindhoven, The Netherlands.
        {\tt math@fitzner.nl},{\tt rhofstad@win.tue.nl}}
        \and
        Remco van der Hofstad$^*$
    }
\title{Generalized approach to the non-backtracking lace expansion}
\begin{document}
\maketitle

\begin{abstract}
The lace expansion is a powerful perturbative technique to analyze the critical behavior of random spatial processes such as the self-avoiding walk, percolation and lattice trees and animals. The non-backtracking lace expansion (NoBLE) is a modification that allows us to improve its applicability in the nearest-neighbor setting on the $\Zd$-lattice for percolation, lattice trees and lattice animals.

The NoBLE gives rise to a recursive formula that we study in this paper at a general level. We state assumptions that guarantee that the solution of this recursive formula satisfies the infrared bound. In two related papers, we show that these conditions are satisfied for percolation in $d\geq\dmin$, for lattice trees in $d\geq\dmintree$ and for lattice animals in $d\geq\dminanimal$.
\end{abstract}

\tableofcontents

\section{Introduction and results}
\subsection{Motivation}
We use the \emph{non-backtracking lace expansion} (NoBLE) to prove the \emph{infrared bound} for several spatial models. The infrared bound implies \emph{mean-field behavior} for these models. The classical lace expansion is a perturbative technique that can be used to show that the two-point function of a model is a perturbation of the two-point function of simple random walk (SRW). This result was used to prove mean-field behavior for self-avoiding walk (SAW) \cite{HarSla91,HarSla92a}, percolation \cite{HarSla90a, HarSla94}, lattice trees and lattice animals \cite{HarSla90b}, oriented percolation \cite {HofSla02}, the contact process \cite{Saka01,HofSak04} and the Ising model \cite{Saka07} in high dimensions.

Being a \emph{perturbative method} in nature, applications of the lace expansion typically necessitate a small parameter. This small parameter is often the inverse of the degree of the underlying graph. There are two possible approaches to obtain a small parameter. The first is to work in a so-called \emph{spread-out} model, where long-, but finite-range, connections over a distance $L$ are possible, and we take $L$ large. This approach has the advantage that the results hold, for $L$ sufficiently large, all the way down to the critical dimension of the corresponding model. The second approach applies to the simplest and most often studied \emph{nearest-neighbor} version of the model. For the nearest-neighbor model, the degree of a vertex is $2d$, which has to be taken large in order to prove mean-field results.

For the self-avoiding walk (SAW) on the nearest-neighbor lattice, Hara and Slade (1991) \cite{HarSla91, HarSla92a} proved the seminal result that dimension $5$ is small enough for the lace expansion to be applied, thus that mean-field behavior holds in dimension $d\geq 5$. This result is optimal in the sense that we do not expect mean-field behavior of SAW in dimension $4$ and smaller. The dimension 4 thus acts as the upper critical dimension. Results in this direction, proving explicit logarithmic corrections, can be found in a series of papers by Brydges and Slade (some also with Bauerschmidt), see \cite{BrySla14a} and the references therein.

For percolation, we expect mean-field behavior for dimensions $d>6$. Hara and Slade also proved this result down to $d>6$ for the spread-out model with sufficiently large $L$ \cite{HarSla90a, HarSla94}. For the nearest-neighbor setting, Hara and Slade computed that dimension $19$ is large enough. These computations were never published. Through private communication with Takashi Hara, the authors learned that in a recent rework of the analysis and implementation the result was further improved to $d\geq 15$ for percolation.

To obtain the mean-field result also in smaller dimensions above the upper critical dimensions, we rely on the NoBLE. In the NoBLE, we explicitly take the interaction due to the last edge used into account. Doing this, we reduce the size of the involved perturbation drastically and are able to show the mean-field behavior for dimensions closer to the upper critical dimension.

In this paper, we formalize a number of assumptions on the general model and prove that under these assumptions the two-point function obeys the infrared bound. The derivation of the model-dependent NoBLE and the verification of the assumptions are not part of this article. We use the generalized analysis to obtain mean-field behavior for the following models: lattice trees in $d\geq\dmintree$ and lattice animals in $d\geq\dminanimal$ \cite{FitHof13g}, and percolation in $d\geq\dmin$ \cite{FitHof13d}.

A NoBLE analysis consists of four steps: Firstly, for a given model, we derive the perturbative lace expansion. Secondly, we prove diagrammatic bounds on the perturbation. Thirdly, we analyze the expansion to conclude the infrared bound given certain assumptions on the expansion. In our analysis, we derive diagrammatic bounds on the lace-expansion coefficients in terms of simple random walk integrals, which can be computed explicitly. This allows us to compute numerical bounds on the lace-expansion coefficients. The fourth step consists of the numerical computation of these SRW-integrals.

In the accompanying papers \cite{FitHof13g,FitHof13d}, we perform the first two steps for  percolation, as well as for lattice trees and lattice animals. In this paper and in a model-independent way, we perform the analysis in the third step and explain the numerical computations of the fourth step. The numerical computations and the explicit checks of the sufficient conditions for the NoBLE analysis to be successful are done in Mathematica notebooks that are available on the website of Robert Fitzner \cite{FitNoblePage}.

The analysis presented in this paper is an enhancement of the analyses performed by Hara and Slade in \cite{HarSla92b} (see also \cite{HarSla90a, HarSla90b, HarSla92a, HarSla94} for related work by Hara and Slade), and by Heydenreich, the second author and Sakai in \cite{HeyHofSak08}.
This paper is organized as follows: In Section \ref{sec-RW}, we first introduce simple random walk (SRW) and non-backtracking walk (NBW). In Section \ref{sec-results-spec-mod} we state the two basic NoBLE relations that are perturbed versions of relations describing the NBW. Then, we state the results for percolation, lattice trees and lattice animals proved in \cite{FitHof13g,FitHof13d}.

In Section \ref{subsecIntroAna}, we first explain the idea of the proof at a heuristic level. Then, we state all assumptions required to perform the analysis in the generalized setting and state the result we prove in this document, namely the infrared bound. We close in Section \ref{sec-disc} with a discussion of our approach.

In Section \ref{secVeriBootstrap}, we prove the technical cornerstone of the analysis, namely, that we can perform the so-called {\em bootstrap argument}. For the analysis in Sections \ref{subsecIntroAna}-\ref{secVeriBootstrap}, we use a simplified NoBLE form of the two-point function that allows us to present the analysis in a clearer way. Thus, we also state the assumptions
in Section \ref{subsecAss} in terms of this simplified characterization.

In Section \ref{secRewrite}, we explain how to derive the simplified NoBLE equation starting from the NoBLE equation for the two-point function, and reformulate the assumptions of Section \ref{subsecAss} into assumptions on the NoBLE coefficients that are derived and bounded in the accompanying papers \cite{FitHof13g,FitHof13d}.

Section \ref{secNumerics} is devoted to the numerical part of the computer-assisted proof. We explain how we compute bounds on the required SRW-integrals.  In Section \ref{secImpliedBounds}, we explain the ideas to bound the NoBLE coefficients that are used for all models that we consider. We end this paper with a general discussion.

\subsection{Random walks}
\label{sec-RW}
We begin by introducing the random walk models that we perturb around and fix our notation.

\subsubsection{Simple random walk}
\label{subsecintroSRW}
Simple random walk (SRW) is one of the simplest stochastic processes imaginable and has proven to be useful in countless applications. For a review of SRW and related models, we refer the reader to \cite{Hugh95, Lawl10, Spit76}.

An $n$-step \emph{nearest-neighbor simple random walk} on $\Zd$ is an ordered $(n+1)$-tuple $\omega=(\omega_0,\omega_1,\omega_2,\dots, \omega_n)$, with $\omega_i\in\Zd$ and $\|\omega_i-\omega_{i+1}\|_1=1$, where $\|x\|_1=\sum_{i=1}^d |x_i|$. Unless stated otherwise, we take $\omega_0=\vec 0=(0,0,\dots,0)$. The step distribution of SRW is given by
	\begin{align}
	\lbeq{DefDK}
	D(x)=\frac 1 {2d} \delta_{\|x\|_1,1},
	\end{align}
where $\delta$ is the Kronecker delta. For two functions $f,g\colon\Zd \mapsto \Rbold$ and $n\in\Nbold$, we define the convolution $f\star g$ and the $n$-fold convolution $f^{\star n}$ by
	\begin{align}
    \lbeq{definition-convolution}
	(f\star g)(x)&= \sum_{y\in\Zd}f(y)g(x-y),
	\end{align}
and
	\begin{align}
    \lbeq{definition-multi-convolution}
	f^{\star n}(x)&=(f^{\star (n-1)}\star f)(x)=(f\star f \star f \star \dots \star f) (x).
	\end{align}
We define $p_n(x)$ as the number of $n$-step SRWs with $\omega_n=x$, so that, for $n\geq 1$,
	\begin{eqnarray}
	\lbeq{SRWRecScheme}
	p_n(x) =\sum_{y\in\Zd} 2d D(y)p_{n-1}(x-y) =2d (D \star p_{n-1})(x)	= (2d)^{n} D^{\star n}(x).
	\end{eqnarray}
We analyze this function using Fourier theory. For an absolutely summable function $f$, we define the Fourier transform of $f$ by
	\begin{eqnarray}
	\lbeq{defFourier}
	\hat f (k) =\sum_{x\in\Zd} f(x) \e^{\ii k\cdot x}\qquad \text{for}\qquad k\in [-\pi,\pi]^d,
	\end{eqnarray}
where $k\cdot x=\sum^d_{i=1} k_ix_i$, with inverse
	\begin{eqnarray}
	\label{defFourierInverse}
	f (x) = \int_{(-\pi,\pi)^d} \hat f(k) \e^{-\ii k\cdot x} \frac {d^d k}{(2\pi)^d}.
	\end{eqnarray}
We use the letter $k$ exclusively to denote values in the Fourier dual space $(-\pi,\pi)^d$. We note that the Fourier transform of $f^{\star n}(x)$ is given by $\hat f(k)^n$ and conclude that
	\begin{eqnarray}
	\lbeq{SRWFourier}
	\hat p_n(k) =(2d)^{n} \hat D^{n}(k),\qquad \text{ with }\qquad
	\hat D(k)=\frac 1 d \sum_{\iota=1}^d \cos(k_\iota).
	\end{eqnarray}
The {\em SRW two-point function} is given by the generating function of $p_n$, i.e., for $z\in \C$,
	\begin{eqnarray}
	\lbeq{genSRW}
	C_z(x)&=&\sum_{n=0}^\infty p_n(x)z^n,\qquad\text{and}
	\qquad \hat C_z(k) =\frac {1}{1-2dz \hat D(k)}
	\end{eqnarray}
in $x$-space and $k$-space, respectively. We denote the SRW \emph{susceptibility} by
	\begin{eqnarray}
    \lbeq{susceptibility-SRW}
	\chi^{\sss\rm SRW}(z)= \hat C_z(0)&=& \frac {1} {1-2dz},
	\end{eqnarray}
with \emph{critical point} $z_c=1/(2d)$. By the form of $\hat C_z(k)$ in {\refeq{genSRW}} and using that $1-\cos(t)\approx t^2/2$ for small $t\in\Rbold$, we see that $\hat C_{z_c}(k)=[1-\hat D(k)]^{-1}\approx 2d / \|k\|^2_2$ for small $k$, where $\|\cdot\|_2$ denotes the Euclidean norm. Since small $k$ correspond to large wave lengths, the above asymptotics is sometimes called the {\it infrared} asymptotics. The main aim in this paper is to formulate general conditions under which the infrared bound is valid for general spatial models.

\subsubsection{Non-backtracking walk}
\label{subsecintroNRW}
If an $n$-step SRW $\omega$ satisfies $\omega_i\not=\omega_{i+2}$ for all $i=0,1,2,\dots,n-2$, then we call $\omega$ \emph{non-backtracking}. In order to analyze non-backtracking walk (NBW), we derive an equation similar to \refeq{SRWRecScheme}. The same equation does not hold for NBW as it neglects the condition that the walk does not revisit the origin after the second step. For this reason, we introduce the condition that a walk should not go in a certain direction $\iota$ in its first step.

We exclusively use the Greek letters $\iota$ and $\kappa$ for values in $\{-d,-d+1,\dots,-1,1,2,\dots,d\}$ and denote the unit vector in direction $\iota$ by $\ve[\iota]\in\Zd$, e.g.\ $(\ve[\iota])_i=\text{sign}(\iota) \delta_{|\iota|,i}$.

Let $b_n(x)$ be the number of $n$-step NBWs with $\omega_0=0,\omega_n=x$. Further, let $b^{\iota}_{n}(x)$ be the number of $n$-step NBWs $\omega$ with $\omega_n=x$ and $\omega_1 \not =\ve[\iota]$. Summing over the direction of the first step we obtain, for $n\geq 1$,
	\begin{align}
	\lbeq{NBWRecScheme1}
	b_{n}(x) &= \sum_{\iota\in\{\pm1,\dots,\pm d\}} b^{\iota}_{n-1}(x+\ve[\iota]).
	\end{align}
Further, we distinguish between the case that the walk visits $-\ve[\iota]$ in the first step or not to obtain, for $n\geq 1$,
	\begin{align}
	\lbeq{NBWRecScheme2}
	b_{n}(x)&=b^{-\iota}_{n}(x) + b^{\iota}_{n-1}(x+\ve[\iota]).
	\end{align}
The {\em NBW two-point functions} $B_z$ and $B^{\iota}_{z}$ are defined as the generating functions of $b_n$ and $b_n^{\iota}$, respectively, i.e.,
	\begin{align}
    \lbeq{NBWRecScheme3}
	B_{z}(x)=&\sum_{n=0}^{\infty} b_n(x)z^n,\qquad \qquad
	B^{\iota}_{z}(x)=\sum_{n=0}^{\infty} b^{\iota}_n(x)z^n.
	\end{align}
Using \refeq{NBWRecScheme1} and \refeq{NBWRecScheme2} for the two-point functions gives
	\begin{align}
	\lbeq{NBWScheme}
	B_{z}(x)=&\delta_{0,x}+z\sum_{\iota\in\{\pm1,\dots,\pm d\}} B^{\iota}_{z}(x+\ve[\iota]),
	\qquad 	
	B_{z}(x)=B^{-\iota}_{z}(x) + z B^{\iota}_{z}(x+\ve[\iota]).
	\end{align}
Taking the Fourier transform, we obtain
	\begin{align}
	\lbeq{NBWScheme-fou}
	\hat B_{z}(k)=&1+\sum_{\iota\in\{\pm1,\dots,\pm d\}} z\e^{-\ii k_\iota}\hat B^{\iota}_{z}(k),
		\qquad
	\hat B_{z}(k)=\hat B^{-\iota}_{z}(k) + z\e^{-\ii k_\iota}\hat B^{\iota}_{z}(k).
	\end{align}
In this paper, we use $\Cbold^{2d}$-valued and $\Cbold^{2d}\times\Cbold^{2d}$-valued functions.
For a clear distinction between scalar-, vector- and matrix-valued quantities, we always write $\Cbold^{2d}$-valued functions with a vector arrow (e.g.\ $\vec v$) and matrix-valued functions with bold capital letters (e.g.\ ${\bf M}$). We do not use $\{1,2,\dots,2d\}$ as index set for the elements of a vector or a matrix, but use $\{-d,-d+1,\dots,-1,1,2,\dots,d\}$ instead. Further, for a $k\in(-\pi,\pi)^d$ and negative index $\iota\in\{-d,-d+1,\dots,-1\}$, we write $k_\iota=-k_{|\iota|}$.

We denote the identity matrix by $\mI\in\Cbold^{2d\times 2d}$ and the all-one vector by $ \v1 =(1,1,\dots,1)^T\in\Cbold^{2d}$. Moreover, we define the matrices $\mJ,\mD[k]\in\Cbold^{2d\times 2d}$ by
	\begin{eqnarray}
    \lbeq{NBW-MatrixDefinition}
	(\mJ)_{\iota,\kappa}=\delta_{\iota,-\kappa}\qquad\quad\text{ and }\qquad\quad
	(\mD[k])_{\iota,\kappa}=\delta_{\iota,\kappa} \e^{\ii k_\iota}.
	\end{eqnarray}
We define the vector $\vec {\hat B}_{z}(k)$ with entries $(\vec {\hat B}_{z}(k))_\iota=\vec {\hat B}^{\iota}_{z}(k)$ and rewrite \refeq{NBWScheme-fou} as
	\begin{eqnarray}
	\lbeq{NBWScheme-k}
	\hat B_{z}(k)&=1+z \v1^T \mD[-k] \vec {\hat B}_{z}(k),
	\qquad	
	\hat B_{z}(k)\v1 =\mJ\vec {\hat B}_{z}(k) +  z\mD[-k] \vec {\hat B}_{z}(k).
	\end{eqnarray}
We use $\mJ\mJ=\mI$ and $\mD[k]\mD[-k]=\mI$ to modify the second equation as follows:
\eqn{
    \hat B_{z}(k)\v1 =\mJ\mD[k]\mD[-k]\vec {\hat B}_{z}(k) +  z\mJ\mJ\mD[-k] \vec {\hat B}_{z}(k)=\mJ\big(\mD[k]+z\mJ\big) \mD[-k]\vec {\hat B}_{z}(k)
	}
which implies that
	\eqn{
	\lbeq{vecB-B-rel}
	 \mD[-k]\vec {\hat B}_{z}(k) = \hat B_{z}(k)\left[\mD[k]+z\mJ\right]^{-1}\mJ\v1.
	}
We use $\mJ\v1=\v1$ and then combine \refeq{vecB-B-rel} with the first equation in \refeq{NBWScheme-k} to obtain
	\begin{eqnarray}
	\lbeq{NBWGen}
	\hat B_{z}(k)&=& \frac {1} {1 -z\v1^T\left[\mD[k]+z \mJ\right]^{-1}\v1}.
	\end{eqnarray}
Then, we use that
	\begin{eqnarray}
	\lbeq{invIDJ}
	\left[\mD[k]+z \mJ\right]^{-1}&=& \frac 1 {1-z^2} \left(\mD[-k]-z \mJ\right),
	\end{eqnarray}
and $\v1^T\mD[-k]\v1=2d\hat D(k)$ to conclude that
	\begin{eqnarray}
	\lbeq{NBWGenSolved}
	\hat B_{z}(k)= \frac {1}{1-2d z\frac {\hat D(k)-z}{1-z^2}}
	= \frac {1-z^2} {1+(2d-1)z^2-2dz\hat D(k)}.
	\end{eqnarray}
The NBW susceptibility is $\chi^{\rm\sss NBW}(z)=\hat B_{z}(0)$ with critical point $z_c=1/(2d-1)$.
The NBW and SRW two-point functions are related by
	\begin{eqnarray}
	\nn
	\hat B_{z}(k)&=& \frac {1-z^2}{1+(2d-1)z^2}\frac {1} {1
-\frac{2dz}{1+(2d-1)z^2}\hat D(k)}\\
	&=&\frac {1-z^2}{1+(2d-1)z^2} \hat C_\frac{z}{1+(2d-1)z^2}(k),
	\end{eqnarray}
so that
	\eqn{
	\lbeq{SRWtoNBWlink}
	\hat B_{1/(2d-1)}(k)=\frac {2d-2}{2d-1}\hat{C}_{1/2d}(k)
	=\frac {2d-2}{2d-1} \frac 1 {1-\hat D(k)}.
	}
This link allows us to compute values for the NBW two-point function in $x$- and $k$-space.
A detailed analysis of the NBW, based on such ideas, can be found in \cite{FitHof13a}.


\subsection{General setting}
\label{subsecBasisRec}
We consider general models defined on the $d$-dimensional hypercubic lattice $\Zd$. For these models, the \emph{two-point function} $G_z\colon \Zd\mapsto \Rbold$ is the central quantity. The two-point function is defined for parameters $z\in[0,z_c)$ where $z_c$ acts as the \emph{critical value}. As for the SRW and NBW, the susceptibility $\hat G_z(0)$ diverges as $z$ approaches $z_c$ from below. The behavior of $G_z$ and $\hat{G}_z$ as $z\nearrow z_c$ is of special interest. We use the NoBLE to prove that the two-point function of the general model is a small perturbation of the critical NBW two-point function \refeq{SRWtoNBWlink} and thereby obeys the \emph{infrared bound}.

To do this, we define $G^\iota_z$ as the two-point function of the model where $e_{\iota}$ is being avoided. The precise definition depends on the model. For NBW, $e_{\iota}$ is avoided in the first step, for percolation $G^\iota_z$ is the two-point function of the model defined on the graph $\Zd\setminus\{\ve[\iota]\}$. For the NBW, these two-point functions are linked by the relations \refeq{NBWScheme}. In the NoBLE, we adapt these two relations for the general model with model-dependent perturbations, and bound the arising perturbation coeffcients. For $d\geq 2$, the NoBLE gives rise to functions $\Xi_z,\Xi^{\iota}_z,\Psi^{\iota}_z$ and $\Pi^{\iota,\kappa}_z$ for  $\iota,\kappa\in\{\pm1,\dots,\pm d\}$, all mapping from $\Zd$ to $\Rbold$, and a function $\aaz\colon \Rbold_+ \to \Rbold_+$, such that, for all $x\in\Zd$ and $z\in[0,z_c)$,
	\begin{align}
	\lbeq{basicgeneral1}
	G_z(x)&=\delta_{0,x}+\Xi_z(x)
	+\aaz\sum_{y\in\Zd}\sum_{\iota\in\{\pm 1, \dots, \pm d\}} (\delta_{0,y}
	+\Psi^{\iota}_{z}(y)) G^{\iota}_{z}(x-y+\ve[\iota]),\\
	\lbeq{basicgeneral2}
	G_z(x)&=G^{\iota}_z(x)+ \aaz G^{-\iota}_{z}(x-\ve[\iota])
	+ \sum_{y\in\Zd} \sum_{\kappa\in \{\pm 1, \dots, \pm d\}} \Pi^{\iota,\kappa}_z(y)
	G^\kappa_{z}(x-y+\ve[\kappa])+\Xi^{\iota}(x).
	\end{align}
In our applications, the variable $\mu_z$ is closely related to the main parameter $z$, but it is not equal.
Therefore, in the analysis we use a second parameter $\aabz$ that allows us to control the critical value in the models under consideration.
For example, for percolation we use
	\eqn{
	\lbeq{parameters-perc}
	\aabp=p,\qquad \qquad \aap=p\prob_p(\ve[1]\text{ not connected to } 0 \mid \text{ the bond $(0,\ve[1])$ is vacant}).
	}
(See Section \ref{secPercolation}, where the percolation model is formally introduced.)

Our goal is to understand the behavior of $G_z$, where we consider the functions $\aaz,$ $\Xi_z,$ $\Xi^{\iota}_z,$ $\Psi^{\iota}_z,\Pi^{\iota,\kappa}_z$ as given. Applying the Fourier transformation on \refeq{basicgeneral1} and \refeq{basicgeneral2} gives
	\begin{align}
	\lbeq{basicgeneralFourier1}
	\hat G_z(k)&= 1+\hat \Xi_z(k) +
	\aaz \sum_{\iota\in\{\pm 1, \dots, \pm d\}} (1+\hat \Psi^{\iota}_z(k))
	\e^{-\ii k_\iota} \hat G^{\iota}_{z}(k),\\
	\lbeq{basicgeneralFourier2}
	\hat G_z(k)&=\hat G^{\iota}_z(k)+
	\aaz \e^{\ii k_\iota} \hat G^{-\iota}_{z}(k)
	+ \sum_{\kappa\in \{\pm 1, \dots, \pm d\}} \hat\Pi^{\iota,\kappa}_z(k)
	\e^{-\ii k_\kappa} \hat G^\kappa_{z}(k)+\hat \Xi^{\iota}(k).
	\end{align}
We define the vectors $\vec {\hat G}_z(k),\vXi[k]$ and $\vPsi[k]$ and the matrix $\mPi[k]$ by
	\begin{eqnarray}
	\big(\vec {\hat G}_z(k)\big)_\iota =\hat G^{\iota}_z(k), \quad
	\big(\vPsi[k]\big)_\iota=\hat \Psi^\iota_z(k),
	\quad \big(\vXi[k]\big)_\iota=\hat \Xi^\iota_z(k),
	\quad \big(\mPi[k]\big)_{\iota,\kappa}=\hat \Pi^{\iota,\kappa}_z(k).
	\end{eqnarray}
Then, we can rewrite \refeq{basicgeneralFourier2} using vectors and matrices as
	\begin{eqnarray}
	\hat G_z(k) \v1 &=&\vec {\hat G}_z(k)+\aaz \mD[k]\mJ \vec {\hat G}_z(k) +
	\mPi[k]\mD[-k] \vec {\hat G}_z(k)+\vXi[k].
	\end{eqnarray}
We obtain from this that
	\begin{align}
	\vec {\hat G}_z(k) =& \mD[k]\left[ \mD[k] + \aaz\mJ +\mPi[k]\right]^{-1} (\hat G_z(k) \v1 -\vXi[k]).
	\end{align}
That this matrix inverse is well defined will be shown in Section \ref{secRewriteProperties}. Next, we rewrite \refeq{basicgeneralFourier1} in vector-matrix notation and solve for $\hat G_z(k)$ as
	\begin{align}
	 \nonumber
	\hat G_z(k)=& 1+\hat \Xi_z(k) + \aaz(\v1+\vPsiz[k])^T \mD[-k] \vec {\hat G}_z(k)\\
	\nonumber
	=& 1+\hat \Xi_z(k) + \aaz(\v1+\vPsiz[k])^T\left[ \mD[k] + \aaz\mJ
	+\mPi[k]\right]^{-1} (\hat G_z(k) \v1 -\vXiz[k])\\
	\lbeq{generalForm}
	=& \frac {1+\hat \Xi_z(k)- \aaz(\v1+\vPsiz[k])^T\left[ \mD[k] + \aaz\mJ +\mPi[k]\right]^{-1} \vXiz[k]}
	{1-\aaz(\v1+\vPsiz[k])^T \left[ \mD[k] + \aaz\mJ +\mPi[k]\right]^{-1} \v1}=\frac {\hat \Phi_z(k)}{1-\hat F_z(k)},
	\end{align}
with
	\begin{align}
      \lbeq{DefPhi}
	\hat \Phi_z(k) :=& 1+\hat \Xi_z(k)- \aaz(\v1+\vPsiz[k])^T\left[ \mD[k] + \aaz\mJ +\mPi[k]\right]^{-1} \vXiz[k],\\
	\lbeq{DefFFunction}
	\hat F_z(k) :=&\aaz(\v1 + \vPsiz[k])^T\left[ \mD[k] + \aaz\mJ +\mPi[k]\right]^{-1} \v1.
	\end{align}
When comparing \refeq{generalForm}-\refeq{DefFFunction} to its equivalent for NBW in \refeq{NBWGen}, we see that \refeq{generalForm} reduces to \refeq{NBWGen} when taking $\aaz=z$, $\hat \Xi_z(k)=(\vPsiz[k])_\kappa=(\mPi[k])_{\iota,\kappa}=(\vXiz[k])_\kappa=0$ for all $z,k,\iota,\kappa$.
Our analysis is based on the intuition that the NoBLE coefficients are small in high dimensions. The majority of our work is to quantify this statement.

\paragraph{Rewrite of the two-point function.}
We use another characterization of the two-point function $\hat G_z(k)$ to perform the analysis. We extract all contributions involving constants and $\hat{D}(k)$, by defining $\cfz,\cpz,\apz,\afz,\hRpz(k),\hRfz(k)$ such that
	\begin{eqnarray}
	\lbeq{DefPhi-simple}
	\hat \Phi_z(k) &:=& \cpz+\apz\hat D(k) +\hRpz(k), \\
	\lbeq{DefFFunction-simple}
	\hat F_z(k) &:=& \cfz+\afz \hat D(k) +\hRfz(k).
	\end{eqnarray}
Then,
	\begin{eqnarray}
	\nonumber
	\hat G_z(k)&=& \frac{\hat \Phi_z(k)}{1-\hat F_z(0)+\hat F_z(0)-\hat F_z(k)}\\
	\lbeq{generalForm-simple}
	&=&\frac{\cpz+\apz \hat D(k) +\hRpz(k)}{\hat \Phi_z(0)/\hat G_z(0)+
	\afz[1-\hat D(k)]+\hRfz(0)-\hRfz(k)}.
	\end{eqnarray}
In Section \ref{secRewrite}, we show how we transform \refeq{DefPhi}-\refeq{DefFFunction} into \refeq{DefPhi-simple}-\refeq{DefFFunction-simple}. Up to that point, we only work with the representation \refeq{DefPhi-simple}-\refeq{DefFFunction-simple} as it simplifies and shortens the analysis.
The quantity in \refeq{generalForm-simple} reduces to the NBW-equivalent \refeq{NBWGenSolved}, when we set $\apz=\hRfz(k)=\hRpz(k)=0$ and $\cpz=1-z^2,\afz=2d z, \hat \Phi_z(0)/\hat G_z(0)=1+(2d-1)z^2-2d z$.

\subsection{Results for specific models}
\label{sec-results-spec-mod}
In this section, we describe the results that our method allows to prove. These results are proved in two accompanying papers \cite{FitHof13g,FitHof13d}.

\subsubsection{Percolation}
\label{secPercolation}
Percolation is a central model in statistical physics and is a very active field of research since its rigorous definition by Broadbent and Hammersley in 1957 \cite{BroHam57}, where this model was proposed to describe the spread of a fluid through a medium. General references for percolation are \cite{BolRio06,Grim99, Hugh96}. A review of recent results can be found in \cite{GriKes12,HeyHof15} and the references therein.
We consider Bernoulli percolation on the hypercubic lattice. We use the definition of \cite[Section 9]{Slad06}:
To each nearest-neighbor bond $\{x,y\}$ we associate an independent Bernoulli random variable $n_{\{x,y\}}$  which takes the value $1$ with probability $p$ and the value $0$ with probability $1-p$, where $p\in[0,1]$.  If $n_{\{x,y\}}=1$, then we say that the bond $\{x,y\}$ is \emph{open}, and otherwise we say that it is \emph{closed}. A configuration is a realization of the random variables of all bonds. The joint probability distribution is denoted by $\prob_p$ with corresponding expectation $\expec_p$.

We say that $x$ and $y$ are connected, denoted by $x\conn y$, when there exists a path consisting of open
bonds connecting $x$ and $y$, or when $x=y$. We denote by ${\mathscr{C}}(x)$ the random set of vertices connected to $x$ and denote its cardinality by $|{\mathscr{C}}(x)|$.
The \emph{two-point function} $\tau_p(x)$ is the probability that $0$ and $x$ are connected, i.e.,
	\begin{align}
	\tau_p(x)=\prob_p (0 \conn x).
	\end{align}
By translation invariance $\prob_p (x \conn y)=\tau_p(x-y)$ for all $x,y\in\Zd$.
We define the \emph{percolation susceptibility}, or expected cluster size, by
	\begin{align}
	\chi(p)=\sum_{x\in\Zd} \tau_p(x) =\expec_p\left[|{\mathscr{C}}(0)|\right].
	\end{align}
We say that the system \emph{percolates} when there exists a cluster ${\mathscr{C}}(x)$ such that $|{\mathscr{C}}(x)|=\infty$. We define $\theta(p)$ as the probability that the origin is part of an infinite cluster, i.e.,
	\begin{align}
	\theta(p)=\prob_p (|{\mathscr{C}}(0)|=\infty).
	\end{align}
For $d\geq 2$, there exists a \emph{critical value} $p_c\in(0,1)$ such that
	\begin{eqnarray}
	p_c(d)=\inf\{p\mid \theta(p)>0\}.
	\end{eqnarray}
Menshikov in 1986 \cite{Mens86}, as well as Aizenmann and Barsky in 1987 \cite{AizBar87}, have proven that the critical value can alternatively be characterized as
	\begin{eqnarray}
	p_c(d)=\sup\left\{p\mid \chi(p)<\infty\right\}.
	\end{eqnarray}
The \emph{percolation probability} $p\mapsto \theta(p)$ is clearly continuous on $[0,p_c)$, and it is also continuous (and even infinitely differentiable) on $(p_c,1]$ by the results of \cite{BerKea84}
(for infinite differentiability of $p\mapsto \theta(p)$ for $p\in (p_c,1]$, see \cite{Russ78}).
Thus, the continuity of $p\mapsto \theta(p)$ on $\Zd$ is equivalent to the statement that $\theta(p_c(d))=0$.

\paragraph{Critical exponents.}
We introduce three \emph{critical exponents} for percolation. It is widely believed that the following limits exist in all dimensions:
	\begin{align}
	\lbeq{percolation-conj-gamma}
	\gamma&=-\lim_{p\nearrow p_c} \frac {\log \chi(p)}{\log (|p-p_c|)},\\
	\lbeq{percolation-conj-beta}
	\beta&=-\lim_{p\searrow p_c} \frac {\log \theta(p)}{\log (|p-p_c|)},\\
	\lbeq{percolation-conj-delta}
	1/\delta&=-\lim_{n\rightarrow \infty} \frac {\log \prob_{p_c} (|{\mathscr{C}}(0)|\geq n)}{\log{n}}.
	\end{align}
A strong form of \refeq{percolation-conj-gamma}, \refeq{percolation-conj-beta} and \refeq{percolation-conj-delta} is that there exist constants $c_\chi, c_\theta,c_\delta\in(0,\infty)$ such that
	\begin{align}
	&\chi(p)=(1+o(1)) c_\chi (p_c-p)^{-\gamma}\qquad\qquad~ \text{ as }p\nearrow p_c,\\
	&\theta(p)=(1+o(1)) c_\theta (p-p_c)^\beta\qquad\qquad\quad~ \text{ as }p\searrow p_c,\\
	&\prob_{p_c} (|{\mathscr{C}}(0)|\geq n)=(1+o(1)) c_\delta n^{-1/\delta} \qquad \text{ as }n\rightarrow \infty,
	\end{align}
and is expected to hold in all dimensions, expect for the critical dimension $d=d_c$, where logarithmic corrections are predicted.
The constants $c_\chi, c_\theta$ and $c_\delta$ depend on the dimension. We say that these exponents exist {\em in the bounded-ratio sense} when the asymptotics is replaced with upper and lower bounds with different positive constants. Further, it is believed that there exist $\eta$ and $c_1, c_2$ such that
	\begin{align}
	\lbeq{Gx-perc}
	\tau_{p_c}(x)&=(1+o(1))\frac {c_1} {|x|^{d-2+\eta}}, &
	\hat \tau_{p_c}(k)=(1+o(1)) \frac {c_2} {|k|^{2-\eta}},
	\end{align}
where  $c_1$ and $c_2$ depend on the dimension only. For percolation, the existence of many more exponents is conjectured and partially also proven. See \cite[Section 2.2]{Grim99} for more details. Our main result for percolation is formulated in the following theorem:

\begin{theorem}[Infrared bound for percolation]
\label{thm-perc}
The infrared bound $\hat \tau_{p_c}(k)\leq A_2(d)/[1-\hat{D}(k)]$ for some constant $A_2(d)$ holds for nearest-neighbor percolation in dimension $d$ satisfying $d\geq \dmin$.
As a result, the critical exponents $\gamma,\beta,\delta$ and $\eta$ exist in the bounded-ratio sense and take their mean-field values $\gamma=\beta=1$, $\delta=2$ and $\eta=0$.
\end{theorem}

Theorem \ref{thm-perc} is proved by combining the model-independent results proved in this paper, with the model-dependent results as proved in \cite{FitHof13d}. There, we also state and prove related results on percolation, such as the existence of the so-called incipient infinite cluster and the existence of one-arm critical exponents. Further, we derive numerical upper bounds on the critical percolation probability $p_c(d)$.

The critical exponents for percolation have received considerable attention in the literature. For the critical exponents, it is known that $\beta\leq 1$ and $\gamma\geq 1$ for all $d\geq 2$, see Chayes and Chayes \cite{ChaCha87b} and Aizenman and Newman \cite{AizNew84}. Further, we know that if $\beta$ and $\gamma$ exists, then $\beta\in(0,1]$ and $\gamma\in [1,\infty)$, see \cite[Section 10.2 and 10.4] {Grim99}. In high dimensions, we expect \emph{mean-field behavior} for percolation. Namely, we expect that
for all dimensions $d>6$, the critical exponents correspond to the exponents of the regular tree given by $\gamma=\beta=1$, and $\eta=0$.
(For the definition of $\eta$ for percolation on trees, see Grimmett \cite[Section 10.1]{Grim99}.) Alternatively, we can interpret $\gamma=\beta=1$, and $\eta=0$ as the critical exponents for branching random walk, see the discussion in \cite{HeyHof15}. An important step to prove mean-field behavior for percolation is the result of Aizenman and Newman \cite{AizNew84} that the finiteness of the triangle diagram, defined by
	\begin{align}
	\Delta(p_c)=(\tau_{p_c} \star \tau_{p_c} \star \tau_{p_c})(0),
 	\end{align}
implies that $\gamma= 1$. This triangle condition also implies that $\beta=1$, see \cite{BarAiz91}. In particular, this implies that $p\mapsto \theta(p)$ is continuous.

Hara and Slade \cite{HarSla89a} use the lace expansion to prove that $\eta=0$ in Fourier space as well as the finiteness of triangle diagram for $d\geq 7$ in the spread-out setting with a sufficiently large parameter $L$. In the spread-out setting all bonds $\{x,y\}$ with $|x-y|\leq L$ are independently open or closed. This is an optimal result in the sense that mean-field behavior is not expected in $d\leq 6$, see \cite{Tou74}, where Toulouse argues that the upper critical dimensions $d_c$, above which we can expect mean-field behavior, equals $6$.
For mathematical arguments why $d_c=6$, see  Chayes and Chayes \cite{ChaCha87b}, Tasaki \cite{Tasa87} or \cite[Section 11.3.3]{HeyHof15}.

For the nearest-neighbor setting, Hara and Slade proved mean-field behavior in sufficiently high dimensions \cite{HarSla89a}. Later, they numerically verified that $d= 19$ is sufficiently high by adapting the proof of the seminal result that self-avoiding walk in dimensions $d\geq 5$ satisfies the infrared bound. In private communication with Takashi Hara, the authors have learned that in a recent improvement of their numerical methods, the mean-field result was established for $d\geq 15$, and thus this implies Theorem \ref{thm-perc}. The proof for both these results ($d\geq 15$ and $d\geq 19$) were never published.

Let us briefly explain how percolation fits into our general framework.  For percolation, in \cite{FitHof13d}, we perform the non-bracktracking lace expansion (NoBLE) for the two-point function $\tau_p(x)$ and $\tau_p^\iota(x)=\prob_p(0\conn x \text{ without using }e_\iota)$. Further, we bound the coefficients arising in this expansion and check that all general assumptions used in the present paper are satisfied.
In the NoBLE, we further identify that $\aap=p\prob_p(e\in {\mathscr{C}}(0)\mid \{0,e\}\text{ is vacant})$. For our analysis we also require a bound on $\aabp=p$ (recall \refeq{parameters-perc}).

\subsubsection{Lattice trees and animals}
\label{secIntroLT}
A nearest-neighbor \emph{lattice tree} (LT) on $\Zd$ is a finite, connected set of nearest-neighbor bonds containing no cycles (closed loops).
A nearest-neighbor \emph{lattice animal} (LA) on $\Zd$ is a finite, connected set of nearest-neighbor bonds, which may or may not contain cycles.
Although a tree/animal $A$ is defined as a set of bonds, we write $x\in A$, for $x\in\Zd$, to denote that $x$ is an element of a bond of $A$. The number of bonds in $A$ is denoted by $|A|$.
We define $t^{\sss (a)}_n(x)$ and $t^{\sss (t)}_n(x)$ to be the number of LAs and LTs, respectively, that consist of exactly $n$ bonds and contains the origin and $x\in\Zd$. We study LA and LT using the \emph{one-point function} $g_z$ and the \emph{two-point function} $\bar G_z$ defined as
	\begin{align}
	g^{\sss (a)}_z=\bar G^{\sss (a)}_z(0)
	=\sum_{A\colon A\ni 0}z^{|A|},& &g^{\sss (t)}_z={\bar G}^{\sss (t)}_z(0)=\sum_{T\colon T\ni 0}z^{|T|},\\
	\lbeq{defLTLATwoPoint}
	\bar G^{\sss (a)}_z(x)=\sum_{n=0}^\infty t^{\sss (a)}_n(x) z^n=\sum_{A\colon A\ni 0,x}z^{|A|},&
	&\bar G^{\sss (t)}_z(x)=\sum_{n=0}^\infty t^{\sss (t)}_n(x) z^n=\sum_{T\colon T\ni 0,x}z^{|T|},
	\end{align}
where we sum over lattice animals $A$ and trees $T$, respectively. For technical reasons, we perform the analysis for the normalised two-point function $G_z(x)=\bar G_z(x)/g_z$. This is not necessary, but simplifies our analysis in the general framework and improves the numerical performance of our method.

We define the \emph{LA and LT susceptibilies} by
	\begin{align}
	\chi^{\sss (a)}(z)=\hat {\bar G}^{\sss (a)}_z(0),&
	&\chi^{\sss (t)}(z)=\hat {\bar G}^{\sss (t)}_z(0),
	\end{align}
and denote the radii of convergence of these sums by $z^{\sss (a)}_c$ and $z^{\sss (t)}_c$, respectively. As for SRW and NBW, $1/z_c$ describes the exponential growth of the number of LTs/LAs as the number of bonds $n$ grows. When we drop the superscript $(a)$ or $(t)$, we speak about LTs and LAs simultaneously. The typical length scale of a lattice tree/animal of size $n$ is characterized by the average radius of gyration $R_n$ given by
	\begin{align}
	\lbeq{genTreestatement}
	R_{n}&= \frac {1} {2\hat t_n(0)} \sum_{x\in\Zd} \|x\|^2_2 t_n(x).
	\end{align}

\paragraph{Critical exponents.}
The asymptotic behavior of $t_n$ and $G_z$ can be described using critical exponents.
We define three of these critical exponents for LA and LT.
In doing so we drop the superscripts $(a)$ and $(t)$ as the following holds for LA and LT.
It is believed that there exist $\gamma,\delta,\nu,\eta$ and $A_1,A_2,A_3,A_4>0$ such that
	\begin{align}
	\chi(z)&=(1+o(1))\frac{A_1}{(1-z/z_c)^\gamma},\qquad
	R_n=(1+o(1)) A_2 \cdot n^{\nu},
	\lbeq{genTreestatement2}
	\\
	\bar{G}_{z_c}(x)&=(1+o(1))\frac {A_3} {\|x\|_2^{d-2+\eta}},\qquad\text{ and }\qquad
	\hat{\bar{G}}_{z_c}(k)=(1+o(1)) A_4 \|k\|_2^{\eta-2},
	\lbeq{Gx-LTLA}
	\end{align}
as $z\nearrow z_c,~n\rightarrow \infty, \|x\|_2\to\infty$ and $k\to 0$, respectively. The exponents are believed to be universal, in the sense that they do not depend on the detailed lattice structure (as long as the lattice is non-degenerate and symmetric). In particular, it is believed that the values of $\gamma,\delta,\eta$ and $\nu$ are the same in the nearest-neighbor setting that we consider here, and in the spread-out setting. The constants $A_i$ do depend on the lattice structure.

For LT/LA, it is believed that the critical exponents take their mean-field values above their upper critical dimension $d_c$, which are $\gamma=1/2, \nu=1/4, \eta=0$. These values correspond to the mean-field model of LT/LA, studied in \cite{BorChaHofSla99}. It is conjectured in \cite{LubIsa79} that the upper critical dimension of LT and LA is $d_c=8$. This conjecture is supported by \cite{HarTas87}, where it is shown that if the ``square diagram'' is finite at the critical point, as is believed for $d>8$, then the critical exponent $\gamma$ satisfies $\gamma\leq1/2$. In \cite{BovFroGla86b}, it has been proven that $\gamma\geq 1/2$ in all dimensions.

For the nearest-neighbor setting that we consider, Hara and Slade give a rigorous proof of mean-field behavior for LT and LA in \emph{sufficiently high dimensions}, see \cite{HarSla90b}. What sufficiently high dimensions means was not made precise. The authors learned through private communication with Takashi Hara that it was not investigated in which dimension the classical lace expansion starts to works. In particular, Hara and Slade expected it to only be successful in dimensions much larger than $d_c=8$. In the \emph{spread-out setting} with $L$ large enough, Hara and Slade in \cite{HarSla90b} proved  the mean-field behavior for LT and LA in all dimension $d>8$.  Our main result for LTs and LAs is formulated in the following theorem:

\begin{theorem}[Infrared bound for LTs and LAs]
\label{thm-LT-LA}
The infrared bound $\hat{\bar{G}}_{z_c}(k)\leq A(d)/[1-\hat{D}(k)]$ holds for some $A(d)$ for nearest-neighbor lattice trees in dimension $d$ satisfying $d\geq \dmintree$,
and for nearest-neighbor lattice animals in dimension $d$ satisfying $d\geq \dminanimal$. As a result, $\gamma$ takes its mean-field value $\gamma=1/2$. The critical exponent $\eta$ exists in the bounded-ratio sense and takes its  mean-field value $\eta=0$.
\end{theorem}

From our analysis we show directly that $\eta=0$ in Fourier space, see the right side of \refeq{Gx-LTLA}, and can easily deduce that $\gamma=1/2$. The proof that 
that $\eta=0$ in $x$-space is non-trivial and is explained in \cite{FitHof13g} using the results by Hara \cite{Hara08}.

For LTs and LAs, in \cite{FitHof13g}, we perform the non-bracktracking lace expansion (NoBLE) on the two-point function $\bar G_z(x)$ and $\bar G_z^\iota(x),$ which is the two-point function in which the LT and LA partially avoid $\ve[\iota]$. In the expansion, we identify $\aaz=zg^\iota_z$ and $\aabz=zg_z$, where $g_z=\bar G_z(0)$ and $g^\iota_z=\bar G^\iota_z(0)$ are the one-point functions.

The expansion proves two relations for $\bar G_z(x)$ and $\bar G_z^\iota(x)$ that are perturbations of \refeq{NBWScheme-fou}.
The NoBLE can further be used to obtain bounds on the NoBLE coefficients and to verifies that the assumptions formulated in this paper indeed hold in the dimensions in Theorem \ref{thm-LT-LA}.

\def\CycleBootstrapInitial[#1]{
\begin{tikzpicture}[line width=1pt,auto,scale=#1]
 \node   at ({2*cos(20)},{2*sin(20)})      {Diagrams};
 \node   at ({2*cos(140)},{2*sin(140)})     {$G_{z_I}$};
 \node   at ({2*cos(260)},{2*sin(260)})      {$\Xi_{z_I}$};
 \draw [->,very thick] ({2*cos(245)},{2*sin(245)})  arc  (245:150:2);
 \draw [->,very thick] ({2*cos(10)},{2*sin(10)})  arc  (10:-90:2);
 \draw [->,very thick] (4,1.2) to ({2.3*cos(20)+0.8},{2.3*sin(20)+0.2}) ;
 \color{white}
 \node[left]   at(-4,1.2)      {Assume a bound};
  \color{black}
 \node[right]   at(4,1.2)      {Comparison to NBW};
\end{tikzpicture}}

\def\CycleBootstrap[#1]{
\begin{tikzpicture}[line width=1pt,auto,scale=#1]

\node   at ({3*cos(90)-0.1},{3*sin(90)})     {$f_i(z)\leq \Gamma_i$};
 \node   at ({3*cos(150)},{3*sin(150)})     {$f_i(z)\leq \gamma_i<\Gamma_i$};
 \node   at ({3*cos(0)},{3*sin(0)})      {Diagrams};
 \node   at ({3*cos(230)},{3*sin(230)})      {Perturbation};
 \draw [->,very thick] ({3*cos(215)-0.2},{3*sin(215)})  arc  (215:165:3);
 \draw [->,very thick] ({3*cos(70)},{3*sin(70)})  arc  (70:10:3);
 \draw [dotted,very thick] ({3*cos(140)},{3*sin(140)})  arc  (140:115:3);

\draw [->,very thick] ({3*cos(350)},{3*sin(350)})  arc  (350:240:3);
 \draw [->,very thick] (-4,3) to ({3*cos(90)-1.4},{3*sin(90)}) ;
 \color{white}
 \node[right]   at(4,1.2)      {Comparison to NBW};

  \color{black}
 \node[left]   at(-4,3)      {Assume a bound};
 \node[left]   at(-3.1,0)   {Conclude a bound};
\end{tikzpicture}}

\section{Main results}
\label{subsecIntroAna}
The main result of this paper is the infrared bound in a general setting. In this section, we first explain the idea of the proof.  Then, we state the assumptions on the general model and prove the infrared bound under these assumptions. We close this section with a discussion of our results.

\subsection{Overview}
\label{subsecStructureAna}
The NoBLE writes $\hat G_z(k)$ as a perturbation of the NBW two-point function, see \refeq{generalForm}, where the perturbation is described by certain NoBLE-coefficients that we denote by $\Xi_z,\Xi_z^\iota, \Psi^\kappa_z$ and $\Pi^{\iota,\kappa}_z$.  In the accompanying papers, we derive the NoBLE and its coefficients and prove that they can be bounded by a combination of simple diagrams. When we can bound these simple diagrams, then we are able to bound the perturbation, and thereby also to derive asymptotics for the two-point function.

Thus, we would like to bound simple diagrams for all $z\leq z_c$. It turn out there exists a $z_I\in(0,z_c)$ such that we can bound the simple diagrams in terms of NBW  diagrams for all $z\leq z_I$. For example, for percolation, $z_I=1/(2d-1)$. To obtain bounds also for $z\in(z_I,z_c)$, we use a {\em bootstrap argument}, which is a common tool in lace-expansion proofs, see e.g.\ \cite{BrySpe85,HarSla90a,HarSla90b}. We next explain how such a bootstrap argument works.

We use the following minor modification of the classical bootstrap argument:
\begin{lemma}[Bootstrap argument]
\label{bootstraplemma}
For $i=1,2,3$, let $z\mapsto f_i(z)$ be continuous functions on the interval $[z_I,z_c)$. Further, let $\gamma_i,\Gamma_i\in\Rbold$ be such that $1\leq\gamma_i<\Gamma_i$ and $f_i(z_I)\leq \gamma_i$. If for $z\in(z_I,z_c)$ the condition $f_i(z)\leq \Gamma_i$ for all $i\in\{1,2,3\}$ implies that $f_i(z)\leq \gamma_i$ for all $i\in\{1,2,3\}$, then in fact $f_i(z)\leq \gamma_i$ for all $z\in[z_I,z_c)$ and $i\in\{1,2,3\}$.
\end{lemma}

\begin{proof}
We consider the continuous function $z\mapsto\max_{i=1,2,3} f_i(z)/\Gamma_i$ and see that the lemma follows directly from the intermediate value theorem for continuous functions.
\end{proof}
\noindent

To apply the bootstrap argument, we use three different functions that we need in order to bound the lace-expansion coefficients:
\begin{enumerate}[(a)]
\item $f_1$ to bound \emph{the critical values} of $\aabz$, $\aaz$ and $z\in(z_I,z_c)$;
\item $f_2$ to bound the \emph{two-point function} $G_z(x)$ in Fourier space;
\item $f_3$ to bound so-called \emph{weighted diagrams}, such as weighted bubbles or triangles.
\end{enumerate}

We define and explain these functions at the end of this section. We first continue our explanation at a more heuristic level.

For $z=z_I$, we bound the simple diagrams using the NBW diagrams and use these bounds to prove that $f_i(z_I)\leq \gamma_i$. This `initializes' the bootstrap argument. For $z\in(z_I,z_c)$, we use the idea depicted in Figure \ref{fig-Heuristic-bootstrap} . Firstly, we assume that $f_i(z)\leq \Gamma_i$ and use this assumption to conclude bounds on various diagrams consisting of combinations of two-point functions. Then, we use these bounds to bound the lace-expansion coefficients. This, in turn, enables us to conclude bounds on the bootstrap functions $f_i(z)$. {\em If} these bounds turn out to be such that $f_i(z)\leq \gamma_i<\Gamma_i$, {\em then} we can use Lemma \ref{bootstraplemma} to conclude that $f_i(z)\leq \gamma_i$ for {\em all} $z<z_c$. These bounds then imply the infrared bound. We extend the result to $z_c$ using a left-continuity property of the two-point function and the NoBLE coefficients, which we prove in the accompanying papers as these arguments are model-dependent.

 \begin{figure}[h]
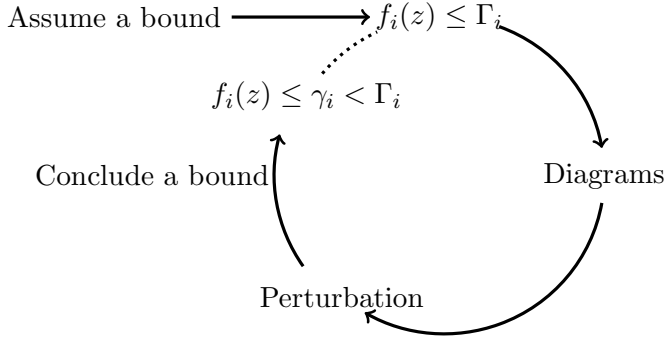

 \begin{center}
\CycleBootstrap[0.7]
\end{center}
 \caption{\label{fig-Heuristic-bootstrap} Structure of the bootstrap argument: Improvement of bounds}
\end{figure}

The structure of the proof, shown in Figure \ref{fig-Heuristic-bootstrap}, is the reason that we are not able to prove the infrared bound for all dimensions above the upper critical dimensions $d_c$. The perturbation, identified by the NoBLE, is for dimensions close to $d_c$ quite large and reduces as we increase the dimension. Thus, for high dimensions such as $d\geq 100$, it is relatively simple to show that the statement that $f_i(z)\leq \Gamma_i$ implies that $f_i(z)\leq \gamma_i<\Gamma_i$. However, it is very difficult to prove such a statement for dimension closer to the upper critical dimension $d_c$. The dimensions stated in Theorem \ref{thm-perc} and \ref{thm-LT-LA} do not show anything specific about the model, but only the limitation of our technique. In fact, for example for percolation, $d>d_c=6$ is the {\em proper} condition.

Using exhaustive bounds on the model-dependent NoBLE coefficients and a more tedious computer-assisted proof might allow to prove the infrared bound in dimensions above, yet closer to, $d_c$. Thereby, it is not clear whether it can be used to obtain the infrared bound for \emph{all} dimension above $d_c$ for percolation, LT and LA.

We next explain the idea of our proof in more detail, so as to further highlight the ideas in this paper. We continue by defining and discussing the bootstrap functions.

\paragraph{Bootstrap functions.}
For the bootstrap, we use the following functions:
	\begin{eqnarray}
	\lbeq{defFunc1}
	f_1(z)&:=&\max\left\{(2d-1)\aabz,c_\mu (2d-1)\aaz\right\},\\
	\lbeq{defFunc2}
	f_2(z)&:=& \sup_{k\in(-\pi,\pi)^d}
	\frac{|\genG[k]|}{\hat{B}_{\mu_c}(k)} = \frac{2d-1}{2d-2}\sup_{k\in(-\pi,\pi)^d} [1-\hat D(k)]\ |\genG[k]|,\\
	\lbeq{defFunc3}
	f_3(z)&:=& \max_{\{n,l,S\}\in \mathcal{S}} \frac{\sup_{x\in S}
	\sum_{y}\|y\|_2^2G_z(y)(G_z^{\star n}\star D^{\star l})(x-y)}{c_{n,l,S}},
	\end{eqnarray}
where $c_\mu>1$ and $c_{n,l,S}>0$ are some well-chosen constants and $\mathcal{S}$ is some finite set of indices. Let us now start to discuss the choice of these functions.

The functions $f_1$ and $f_3$ can been seen as the combinations of multiple functions.  We group these functions together as they play a similar role and are analyzed in the same way.
We do not expect that the values of the bounds on the individual functions constituting $f_1$ and $f_3$ are comparable. This is the reason that we introduce the constants $c_\mu$ and $c_{n,l,S}$.

The value of $n$ is model-dependent.  For SAW, we would use only $n=0$. For percolation we use $n=0,1$ and $n=0,1,2$ for LT and LA. This can intuitively be understood as follows. By the $x$-space asymptotics in \refeq{Gx-perc} and \refeq{Gx-LTLA}, and the fact that $(f\star f)(x)\sim \|x\|_2^{4-d}$ when $d>4$ and $f(x)\sim \|x\|_2^{2-d}$, we have that $\|y\|_2^2G_z(y)\sim (G_z\star G_z)(y)$. As a result, this suggests that
	\eqan{
	\sum_{y}\|y\|_2^2G_z(y)(G_z^{\star n}\star D^{\star l})(x-y)
	&\sim \sum_{y}(G_z\star G_z)(y)(G_z^{\star n}\star D^{\star l})(x-y)\\
	&=\big(G^{\star (n+2)}_z\star D^{\star l}\big)(x),\nn
	}
so that finiteness of $\sum_{y}\|y\|_2^2G_z(y)(G_z^{\star n}\star D^{\star l})(x-y)$ is related to
finiteness of the bubble when $n=0$, of the triangle when $n=1$ and of the square when $n=2$.

The choices of point-sets $S\in \mathcal{S}$ improve the numerical accuracy of the method. For example, we obtain much better estimates in the case when $x=0$, since this leads to closed diagrams, than for $x\neq 0$. For $x$ being a neighbor of the origin, we can use symmetry to improve our bounds significantly.
To obtain the infrared bound for percolation in $d\geq \dmin$ we use
	\begin{align*}
  	\mathcal{S}=\big\{ \{0,0,\mathcal{X}\},\{1,0,\mathcal{X}\},\{1,1,\mathcal{X}\},\{1,2,\mathcal{X}\},
	\{1,3,\mathcal{X}\},\{1,4,\{0\}\} \big\},
	\end{align*}
with $\mathcal{X}=\{x\in\Zd\colon \|x\|_2>1\}$. This turns out to be sufficient for our main results.

\subsection{Assumptions}
\label{subsecAss}
In this section, we state the assumptions that we need to perform the general NoBLE analysis.
The assumptions are in terms of the simplified form of the NoBLE in \refeq{generalForm-simple}. In Section \ref{secRewrite}, we derive this simplified NoBLE form of the NoBLE and translate the following assumption
in terms of the NoBLE-coefficients. We begin with two assumptions on the two-point function that are completely independent of the expansion:

\begin{ass} [Bound for the initial value]
\label{assXBoundInitialCondition}
There exists a $z_I\in[0,z_c)$ such that
	\begin{align}
	\lbeq{assInitialXBound}
	G_{z}(x)\leq B_{1/(2d-1)}(x)=\frac {2d-2}{2d-1} C_{1/2d}(x)
	\end{align}
for all $x\in \Zd$ and $z\in[0,z_I]$.
\end{ass}

\noindent
To control the growth of the two-point function as we approach the critical value $z_c$, we use the following two assumptions:

\begin{ass} [Growth of the two-point function]
\label{assGzBehavoir}
For every $x\in \Zd$, the two-point functions $z\mapsto G_z(x)$ and $z\mapsto G^\iota_z(x)$ are non-decreasing and  differentiable in $z\in(0,z_c)$. For all $\varepsilon>0$ there exists a constant $c_{\varepsilon}\geq 0$ such that for all $z\in(0,z_c-\varepsilon)$ and $x\in\Zd\setminus\{0\}$,
	\begin{eqnarray}
	\lbeq{assGzDiffBound}
	\frac d {dz} G_z(x)\leq c_{\varepsilon} (G_z\star D\star G_z)(x)
	\quad \text{ and therefore }\quad \frac d {dz} \hat G_z(0)\leq c_{\varepsilon} \hat G_z(0)^2.
	\end{eqnarray}
For all $z\in(0,z_c)$, there exists a constant $K(z)<\infty$ such that $\sum_{x\in\Zd} \|x\|_2^2 G_{z}(x)<K(z)$.
\end{ass}

\begin{ass} [Continuity]
\label{assContinuous}
For $z\in[0,z_c)$, $z\mapsto \aabz$ and $z\mapsto \aaz$ are continuous.
\end{ass}

We only consider models where the two-point function has the following set of symmetries:

\begin{definition}[Total rotational symmetry]
\label{defSignedPermuations}
We denote by $\mathcal{P}_d$ the set of all permutations of $\{1,2,\dots, d\}$. For $\nu\in \mathcal{P}_d$, $\delta\in\{-1,1\}^d$ and $x\in\Zd$, we define $p(x;\nu,\delta)\in\Zd$ to be the vector with entries $(p(x;\nu,\delta))_j=\delta_j x_{\nu_j}$. We say that a function $f:\Zd\mapsto \Rbold$ is {\em totally rotationally symmetric} when $f(x)=f(p(x;\nu,\delta))$ for all $\nu\in \mathcal{P}_d$ and $\delta\in\{-1,1\}^d$.
\end{definition}

\begin{ass}[Symmetry]
\label{assSym}
We assume that $x\mapsto G_z(x),x\mapsto \Rfz(x)$ and $x\mapsto \Rpz(x)$ are totally rotationally symmetric. Further, we assume that the lace-expansion coefficients satisfy
	\begin{eqnarray}
  	\lbeq{assSym-interchange}
	\hat \Psi^{\iota}_z(0)&=&\hat \Psi^{\kappa}_z(0), \qquad \qquad\
	\sum_{\iota'}\hat \Pi^{\iota',\kappa}_z(0)=\sum_{\kappa'}\hat \Pi^{\iota,\kappa'}_z(0)
	\end{eqnarray}
for all $\iota,\kappa\in\{\pm 1,\pm 2,\dots,\pm d\}$ and $z\leq z_c$.
\end{ass}

The following is the central assumption to perform the bootstrap.  We assume that \emph{if} $f_1(z),$ $f_2(z),$ $f_3(z)$ are bounded for a given $z\in[0,z_c)$, \emph{then} the functions $\afz,\apz,\Rfz,\Rpz$ obey certain diagrammatic bounds.  The form of these bounds is delicate and depends sensitively on the precise model under consideration.  

\begin{ass}[Diagrammatic bounds]
\label{assDiagBounds}
Let $\Gamma_1,\Gamma_2,\Gamma_3\geq 0$.
Assume that $z\in (z_I,z_c)$ is such that $f_i(z)\leq \Gamma_i$ holds for all $i\in\{1,2,3\}$.
Then, $\hat G_z(k)\geq 0$ for all $k\in(-\pi,\pi)^d$, and the following bounds hold with $\beta_{\bullet}$ depending only on $\Gamma_1,\Gamma_2,\Gamma_3,d$ and the model:\\
(a) There exist $\betaaa>1, \lowaf,\upaf,\betaap,\lowcp,\upcp>0 $, such that
	\begin{align}
	\lbeq{analys-assumed-rho-Bound}
	\frac {\bar \mu_z}{\mu_z}\leq \betaaa,\qquad \qquad \lowcp \leq \cpz \leq \upcp,\\
	\lowaf \leq \afz \leq \upaf, \qquad\qquad \quad |\apz| \leq \betaap.
	\end{align}
(b) There exist $\overline{\beta}_{\sss \Pi^{\iota}}, \underline {\beta}_{\sss \Psi^{\kappa}}>0$, such that
	\begin{align}
	\lbeq{analys-assumed-Bound-PiPsi}
	\sum_{x,\kappa}\Pi^{\iota,\kappa}_z(x)\leq \overline{\beta}_{\sss \Pi^{\iota}},\qquad\qquad
	\sum_{x}\Psi^{\kappa}_z(x)\geq  -\underline {\beta}_{\sss \Psi^{\kappa}}.
	\end{align}
(c) There exist $\beta_{{\sss R,F}},\beta_{{\sss R,\Phi}},\betadeltaRpz,\betadeltaRfz,\betadeltaRfzlow>0$ such that
	\begin{align}
	\lbeq{analys-assumed-Remainder}
	\sum_{x} |\Rfz(x)|\leq& \beta_{{\sss R,F}}, \qquad\quad\quad  \sum_{x}|\Rpz(x)|\leq \beta_{{\sss R,\Phi}},\\
	\lbeq{analys-assumed-displacement-1}
	\sum_{x}\|x\|_2^2 |\Rpz(x)|\leq&\betadeltaRpz,\qquad\qquad 
	\sum_{x}\|x\|_2^2 |\Rfz(x)|\leq \betadeltaRfz,\\
	\lbeq{analys-assumed-displacement-3}
	\hRfz(0)-\hRfz(k)\geq& -\betadeltaRfzlow[1-\hat D(k)],
	\end{align}
for all $k\in(-\pi,\pi)^d$. Further, we assume that $\lowaf-\betadeltaRfzlow>0$ and $\lowcp-\betaap-\beta_{{\sss R,\Phi}}>0$.
If Assumption \ref{assXBoundInitialCondition} holds, then the bounds stated above also hold for $z=z_I$, where in this case the constants $\beta_{\bullet}$ only depend on the dimension $d$ and the model.
\end{ass}
To extend the infrared bound to $z=z_c$, we use the following assumption:

\begin{ass} [Growth at the critical point]
\label{assGzBehavoirCritical}
We assume that, if the bounds stated in Assumption \ref{assDiagBounds} hold uniformly for $z\in[z_I,z_c)$, then $z\mapsto \hat G_z(k)$ is left-continuous at $z=z_c$ for any $k\neq 0$, and that the bounds stated in Assumption \ref{assDiagBounds} also hold for $z=z_c$.
\end{ass}

Assumptions \ref{assSym}-\ref{assGzBehavoirCritical} depend on the NoBLE  and are stated in terms of its simplified form \refeq{generalForm-simple}. In Section \ref{secRewrite}, we replace these assumptions by assumptions on the NoBLE-coefficients. We have chosen to use the form \refeq{generalForm-simple} for the analysis, as it simplifies the presentation of the analysis considerably.

\subsection{Main result: Infrared bound}
To successfully apply the bootstrap argument, we require that we can improve the bound on the bootstrap
functions. This is the content of the following condition:

\begin{definition}[Sufficient condition for the improvement of bounds]
\label{finalCondition}
For $\gamma,\Gamma\in\Rbold^3$ and $z\in[z_I,z_c)$, we say that $P(\gamma,\Gamma,z)$ holds when $f_i(z)\leq \Gamma_i$ for $i\in\{1,2,3\}$ and the following conditions hold:
	\begin{eqnarray}
	\lbeq{conditiongamma}
	0&\leq& \gamma_i < \Gamma_i\qquad\qquad \text{ for }i=1,2,3,\\
	\lbeq{conditionf1}
	\gamma_1 &\geq&\max\Big\{f_1(z_I), \max\{\betaaa,c_\mu\}
	\frac {1+\overline{\beta}_{\sss \Pi^{\iota}}}
	{ 1 - \frac {2d}{2d-1}\underline {\beta}_{\sss \Psi^{\kappa}}}\Big\},\\
	\lbeq{conditionf2}
	\gamma_2  &\geq&  \frac{2d-1}{2d-2}
	\frac {\upcp+\betaap+\betaRp}{\lowaf-\betadeltaRfzlow},
	\end{eqnarray}
and $\gamma_3$ is larger than the maximum of the right-hand sides of \refeq{conditionf3-two-initial-statement} and \refeq{bound-on-f3}.
\end{definition}

The last condition states that the initial condition $f_3(z_I)\leq\gamma_3$ holds and that the improvement of bounds succeeds for $f_3$. We do not give a formal statement at this point, as it is involved and would require notation that has not yet been introduced. If the bootstrap succeeds, then we are able to prove our main result:

\begin{theorem}[Infrared bound]
\label{centralresult}
Let $k\in[-\pi,\pi]^d,z_I\in[0,z_c)$ and $\gamma,\Gamma\in\Rbold^3$. If Assumptions \ref{assXBoundInitialCondition}-\ref{assGzBehavoirCritical} and $P(\gamma,\Gamma,z)$ hold for all $z\in[z_I,z_c)$, then, for all $z\in [z_I, z_c]$,
	\begin{eqnarray}
	\lbeq{infrared1}
	\genG[k] [1-\hat D(k)]&\leq& \frac {2d-2}{2d-1}\gamma_2,\\
	\lbeq{infrared2}
	\genG[k] &\leq& \frac {A(d)}{1/\hat G_z(0) +[1-\hat D(k)]},
	\end{eqnarray}
with
	\begin{eqnarray}
	\lbeq{infrared2-ad}
	A(d)=\frac {\upcp+\betaap+ \beta_{{\sss R,\Phi}}} {
	\min\big\{\lowcp-\betaap- \beta_{{\sss R,\Phi}},\lowaf-\betadeltaRfzlow \big\}}.
	\end{eqnarray}
\end{theorem}

We postpone the discussion of Theorem \ref{centralresult} to Section \ref{sec-disc}. We continue to discuss the strategy of proof.

\subsection{Proof subject to a successful bootstrap}
The central statement, that we prove in Section \ref{secVeriBootstrap}, is that we can apply the bootstrap argument when $P(\gamma,\Gamma,z)$ holds. We now formalize this statement in Proposition \ref{succesfulbootsprop}:

\begin{prop}[A successful bootstrap]
\label{succesfulbootsprop}
Let $\gamma,\Gamma\in\Rbold^3$. If Assumptions \ref{assXBoundInitialCondition}-\ref{assGzBehavoirCritical} and $P(\gamma,\Gamma,z)$ hold for all $z\in[z_I,z_c)$, then the functions $f_1,f_2,f_3$ defined in \refeq{defFunc1}-\refeq{defFunc3} are
continuous, $f_i(z_I)<\gamma_i$ for $i\in\{1,2,3\}$ hold, and $f_i(z)\leq \Gamma_i$ for all $i\in\{1,2,3\}$ implies that $f_i(z)\leq \gamma_i$ for all $i\in\{1,2,3\}$.
\end{prop}

We prove Proposition \ref{succesfulbootsprop} in Section \ref{secVeriBootstrap} one function $f_i$ at a time.
Now we prove our main result, Theorem \ref{centralresult}, assuming that Proposition \ref{succesfulbootsprop} holds:

\begin{proof}[Proof of Theorem \ref{centralresult} subject to Proposition \ref{succesfulbootsprop}] By Lemma \ref{bootstraplemma},
	\begin{eqnarray}
	\lbeq{infrared1-ministep}
	\frac {2d-1}{2d-2}\hat G_z(k)[1-\hat D(k)]\leq f_2(z)\leq \gamma_2 \qquad \text{ for all }z\in(z_I,z_c).
	\end{eqnarray}
By Assumption \ref{assGzBehavoirCritical}, $z\mapsto \hat G_{z}(k)$ is left-continuous at $z=z_c$ for $k\neq 0$. From this, we conclude that \refeq{infrared1-ministep} also holds for $z=z_c$ and $k\neq 0$, which proves  \refeq{infrared1}.
To prove \refeq{infrared2}, we use the bounds of Assumption \ref{assDiagBounds} on $\hat G_z(k)$ as given in \refeq{generalForm-simple}
\begin{eqnarray}
\hat G_z(k) &=&\frac{\cpz+\apz \hat D(k) +\hRpz(k)}{ \frac {\cpz+\apz \hat D(0) +\hRpz(0)}{\hat G_z(0)}+
	\afz[1-\hat D(k)]+\hRfz(0)-\hRfz(k)}\nnb
&\leq &\frac {\upcp+\betaap+ \beta_{{\sss R,\Phi}}}
{
\frac {\lowcp-\betaap- \beta_{{\sss R,\Phi}}} {\hat G_z(0)}
+
(\lowaf-\betadeltaRfzlow)[1-\hat D(k)]}\nnb
&\leq &
\frac {\upcp+\betaap+ \beta_{{\sss R,\Phi}}}
{\min\{\lowcp-\betaap- \beta_{{\sss R,\Phi}},\lowaf-\betadeltaRfzlow\} }
\ \frac 1 {1/G_z(0)+[1-\hat D(k)]},
\lbeq{mainTheoremtmp1}
\end{eqnarray}
which implies \refeq{infrared2} and derives the expression for $A(d)$ in \refeq{infrared2-ad}.
\end{proof}


\subsection{Discussion}
\label{sec-disc}
In general, a proof using the NoBLE consists of four parts, see also Figure \ref{Struct-NoBLE}: (a) the derivation of the non-backtracking lace expansion or NoBLE; (b) diagrammatic bounds on the NoBLE coefficients; (c) the analysis of the NoBLE equation; and (d) a numerical verification of the conditions in $P(\gamma,\Gamma,z)$, using a computer-assisted proof.

\begin{center}
\begin{figure}[h]
\begin{tikzpicture} [scale=0.85]

  \draw [->,line width=2] (3,0) to (6.75,0);
  \draw [->,line width=2] (3.5,-5) to (5.25,-5);
  \draw [<-,line width=2] (10,-1.25) to (10,-3.25);
  \draw [->,line width=2] (9.5,-1.25) to (9.5,-3.25);
  \draw [->,line width=2] (0,-1.25) to (0,-3.25);

  \draw [<-,line width=2] (9.75,-5.75) to (9.75,-6.75);

  \node[above]   at(4.75,0)     {General relation};
  \node[below]   at(4.35,-5)   {Bounds in};
  \node[below]   at(4.35,-5.5)   { form of diagrams};
  \node[left]   at(0,-1.5)        {Coefficients to};
  \node[left]   at(0,-2)        {describe the};
  \node[left]   at(0,-2.5)        {perturbation};
  \node[right]   at(10,-1.75)      {Bound on the};
  \node[right]   at(10,-2.25)      {perturbation};

  \node[left]   at(9.5,-1.5)      {Prior bounds};
  \node[left]   at(9.5,-2)       {on the two-point };
  \node[left]   at(9.5,-2.5)      {function};
  \node[right]   at(10,-6.2)      {Numerical values};

\node [draw] {
\begin{tabular}{c}
{\bf \large Expansion} \\[1mm]
\hline
\vspace{2mm}
{\small Model dependent} \\[-2mm] {\small in accompanying papers}
\end{tabular}
};

\node [draw]  at (10,0){
\begin{tabular}{c}
{\bf \large Analysis/ Bootstrap } \\[1mm]
\hline
\vspace{2mm}
{\small Model independent }\\[-2mm]
{\small in this paper}
\end{tabular}
};

\node [draw]  at (10,-4.5){
\begin{tabular}{c}
{\bf \large Numerical bounds } \\[1mm]
\hline
\vspace{2mm}
{\small Model dependent} \\[-2mm]
{\small in accompanying Mathematica notebooks}
\end{tabular}};

\node [draw]  at (10,-8){
\begin{tabular}{c}
{\bf \large Numerical computation} \\[1mm]
\hline
{\small Model dependent }\\
{\small in accompanying Mathematica notebooks
}
\end{tabular}};

\node [draw]  at (0,-4.5){
\begin{tabular}{c}
{\bf \large Diagrammatic bounds } \\[1mm]
\hline
\vspace{2mm}
{\small Model dependent }\\[-2mm]
{\small in accompanying papers}
\end{tabular}};
\end{tikzpicture}
\label{Struct-NoBLE}
\caption{Structure of the non-backtracking lace expansion.}
\end{figure}
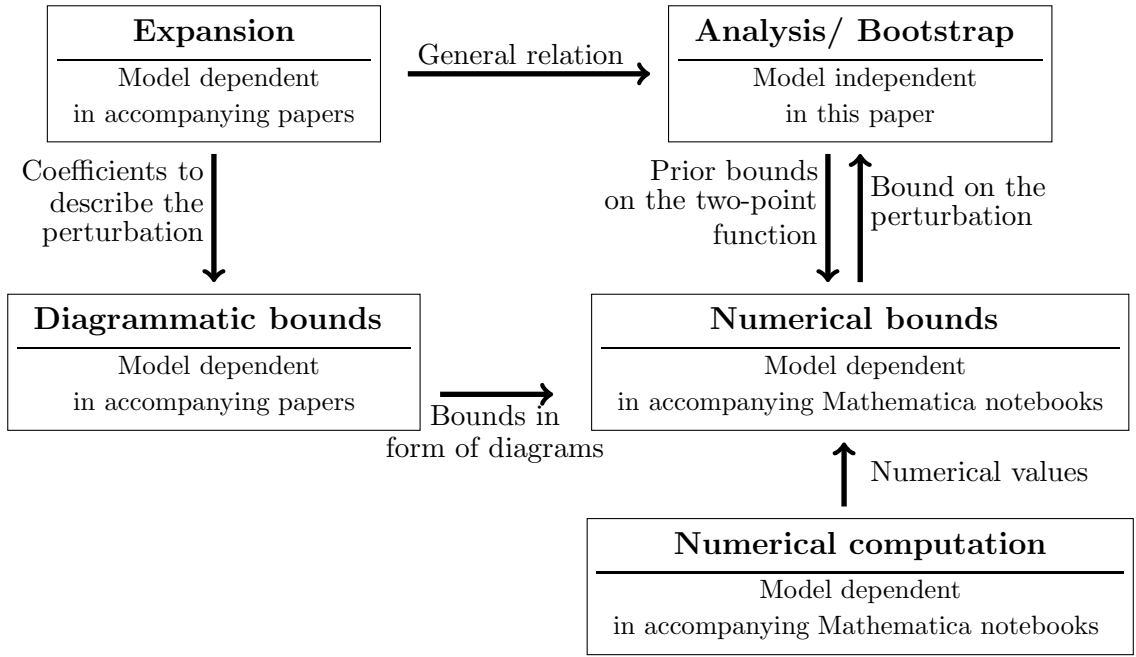
\end{center}

Parts (a) and (b) are performed in the model-dependent papers \cite{FitHof13d, FitHof13g}.
Part (c) is performed here in a generalized setting. Part (d) is explained in Section \ref{secNumerics}, and is numerically performed in three Mathematica  notebooks. In the first notebook, we compute SRW-integrals for a given dimension, see Sections \ref{secNumericsSRW}-\ref{secNumericsAdvanced}. In the second notebook, we implement bounds on the simplified rewrite \refeq{generalForm-simple} and the bound necessary for the improvement of $f_3$, see Appendix \ref{secRewriteBounds} and Section \ref{subsecImproF3bar}, respectively. These two parts are completely model independent. In the third notebook, we use the values of the SRW-integrals and the bootstrap assumptions to compute numerical bounds on the diagrammatic bounds on the NoBLE coefficients. These bounds are then used to verify the conditions $P(\gamma,\Gamma,z)$, which, when successful, imply that the analysis here yields the infrared bounds in the specific dimension under consideration. Since the bounds are {\em monotone} in the dimension, the bounds then also follow for all dimensions larger than that specific dimension.

In the thesis of the first author \cite{Fit13}, the analysis was performed in two ways.
The first was based on the $x$-space approach, as originally worked out by Hara and Slade in \cite{HarSla92b}. This approach was used by Hara and Slade in \cite{HarSla92b, HarSla92a} to obtain that mean-field behavior holds for self-avoiding walk (SAW) in all dimensions $d\geq 5$. This is optimal in the sense that mean-field behavior for SAW in $d\leq 4$ should not be true. See \cite{BrySla14a} and references therein for results in this direction. Further, Hara and Slade adapted their method to percolation \cite{HarSla90a}, which led to the famous result that mean-field behavior for percolation holds for $d\geq 19$.

The second analysis was based on the trigonometric approach, first used for finite tori in \cite{BorChaHofSlaSpe05b} and worked out for $\Zd$ in \cite{HeyHofSak08, Slad06}. However, it was never verified above which dimension this technique can be applied. Thus, it was initially not obvious to us which method would be numerically optimal. It was only by implementing both methods that we discovered that the $x$-space approach, combined with the NoBLE analysis, is numerically superior. This is the reason that we only describe this method here. In conclusion, our method is, after the derivation of the NoBLE, heavily inspired by that of Hara and Slade in \cite{HarSla92b}. We have benefitted tremendously from their work, as well as from the many discussion that we have had with Takashi Hara and Gordon Slade over the past years. From private communication, we have learned that Takashi Hara also managed to prove that percolation in dimension $d\geq 15$ obeys the infrared bound, although this result has not appeared in print. We hope that our method, as well as the accompanying Mathematica notebooks that are publicly available \cite{FitNoblePage} to anyone who is interested, increase the transparency of the proof of the infrared bound for all the models involved.

The main difference between our work and the work by Hara and Slade in \cite{HarSla92b} is that in our method, the loops creating the perturbative terms are made to consist of at least 4 bonds, while in the classical lace expansion they could consist of immediate reversals (2 bonds). This makes the perturbation considerably smaller, and allows for an analysis that is to a much larger extent model-independent as the analysis in \cite{HarSla92b} and its adaptation to percolation. It also explains why our method gives reasonable results for lattice trees and lattice animals, models that previously had not been attempted by Hara and Slade. In discussions with Takashi Hara, we have found that our bounds on e.g.\ the triangle diagram are slightly better in dimension $15$, whereas he has a much more sophisticated and model-dependent analysis of the lace-expansion coefficients.

For the SAW, we also derived a NoBLE and implemented the bootstrap. In this way, we can show that mean-field behavior holds for SAW in $d\geq 7$, see \cite{Fit13} and \cite{FitNoblePage}. While the proof for $d\geq 7$ is relatively simple, we expect that an extension of the technique to $d=5,6$ will not produce a substantially simpler proof than that of Hara and Slade, that is already optimal in the sense that it proves the result in all dimension above the SAW-upper critical dimension $4$. Thus, we have not attempted to improve upon our result.

Let us dwell a bit on the distinction between the $x$-space approach and the $k$-space or trigonometric approach. We require bounds on {\it weighted diagrams}, alike
\begin{align}
\hRfz(0)-\hRfz(k)=\sum_x \Rfz(x)[1-\cos(k\cdot x)] \leq [1-\hat D(k)]  \beta.
\end{align}
We can either bound the underlying diagram directly in Fourier space or use the following lemma to translate it into the $x$-space approach:
\begin{lemma}[Fourier transforms and step distributions]
\label{lemmaFdifferenceToX}
For a summable, non-negative function $g$ that is totally rotationally symmetric, as defined in Definition \ref{defSignedPermuations}, the following bound holds:
	\begin{eqnarray}
	\lbeq{DifferenceTosecondMoment}
	\sum_{x}g(x)[1-\cos(k\cdot x)]&\leq& [1-\hat D(k)] \sum_{x}g(x) \|x\|_2^2.
	\end{eqnarray}
\end{lemma}
To distribute the {\it weight} $\|x\|_2^2$ over a large diagram into the weights of parts of the diagram, we use the relation that, for $x_i\in\Zd$:
	\begin{align}
	    \lbeq{Split-weight-ineq}
	   \big\|\sum_{i=1}^J x_i\big\|_2^2=
	&\sum_{i=1}^J \|x_i\|_2^2+ \sum_{i=2}^J x_i^T \Big(\sum_{j=1}^{i-1} x_j\Big),
	\qquad
	  &
	  \big\|\sum_{i=1}^J x_i\big\|_2^2\leq& J\sum_{i=1}^J\|x_i\|_2^2.
	\end{align}
The diagrammatic bounds for the $k$-space approach are very similar, due the following analogous result, which is of independent interest:
\begin{lemma}[Split of cosines]
\label{lemmacosdecomp}
Let $t\in\Rbold$ and $t_i\in\Rbold$ for $i=1,\dots,J$ such that\\
$t=\sum_{i=1}^Jt_i$. Then,
	\begin{align}
	\lbeq{lemmacosdecompFirstStep}
	1-\cos(t)&\leq  \sum_{i=1}^J [1-\cos(t_i)]+ \sum_{i=2}^J \sin(t_i)\sin\Big(\sum_{j=1}^{i-1} t_j\Big),\\
	\lbeq{lemmacosdecomp-line}
	1-\cos(t)&\leq J\sum_{i=1}^J[1-\cos(t_i)].
	\end{align}
\end{lemma}
The inequality \refeq{lemmacosdecomp-line} with a factor $2J+1$ is commonly used in the lace-expansion literature. While reviewing the proof, the authors found that a minor adaptation improves the leading factor to be $J$. The proof of Lemmas \ref{lemmaFdifferenceToX} and \ref{lemmacosdecomp} can be found in Appendix \ref{sec-proof-lemmas}.\\

Let us close this section by proposing some extensions of our work. We do not manage to prove the infrared bound all the way down to the upper critical dimension for percolation. For this, we need even better arguments. One might hope that this can be done by a more careful analysis that compares the interacting models with memory-$m$ walk for large values of $m$. Here, a SRW is called a memory-$m$ walk when it has no loops of length at most $m$. Thus, NBW is memory-2. We can easily derive such a memory-$m$ expansion for SAW, for percolation, LT and LA this is already more involved. However, the analysis required for this expansion is much more involved, and we have not tried this more general approach. One particular problem is that we do not know what the memory-$m$ Green's function is, so that it is harder to explicitly expand around this. We think that also for this approach a numerical (and thus computer-assisted) proof is necessary.
\section{Verification of the bootstrap conditions}

\label{secVeriBootstrap}
In this section, we prove Proposition \ref{succesfulbootsprop} one function $f_i$ at a time.

\subsection{Conditions for $f_1$}
\label{secf1}
In this section, we prove that the properties of $f_1$ in Proposition \ref{succesfulbootsprop} hold.

By Assumption \ref{assContinuous}, $z\mapsto \aabz$ and $z\mapsto \aaz$ are continuous, so that $z\mapsto f_1(z)$ is also continuous. From \refeq{conditionf1}, we conclude that $f_1(z_I)\leq \gamma_1$.
To show that for all $z\in(z_I,z_c)$, $f_i(z)\leq\Gamma_i$ for all $i\in\{1,2,3\}$ implies that $f_1(z)\leq\gamma_1$, we prove a relation between $\hat B_{\mu}(0)$ and $\hat G_z(0)$. We use the abbreviations $\psi_z= \hat \Psi^{\iota}(0)$ and
$\pi^\iota_z=\sum_\kappa \hat \Pi^{\iota,\kappa}(0)$, where the choice of $\iota\in\{\pm1,\dots, \pm d\}$ is by Assumption \ref{assSym} not relevant.

\begin{lemma}[Link between NBW and general susceptibility]
\label{lemmaChoiceMu}
Let Assumption \ref{assSym} hold and define
	\begin{eqnarray}
	\lbeq{choiceMu}
	\lambda_z&=& \frac {(1+\psi_z)  \aaz } {1+\pi^\iota_z-\aaz\psi_z},\quad\qquad \text{ so that }\quad
	\aaz=\frac {1+\pi_z^\iota}{(1+\psi_z)/\lambda_z+\psi_z}.
	\end{eqnarray}
Then, $\hat B_{\lambda_z}(0)\hat \Phi_z(0)=\genG[0]$ for all $z<z_c$.
\end{lemma}

\begin{proof} Since $\mD[0]=\mI$ and $\vPsi[0]=\psi_z\v1$, the two-point function $\hat G_{z}(k)$ in the form of \refeq{generalForm} simplifies for $k=0$ to
	\begin{eqnarray}
	\hat G_{z}(0) &=& \frac {\hat \Phi_z(0)}{1- \aaz(\v1+\vPsi[0])^T\left[ \mD[0]+\aaz \mJ
	+ \mPi[0]\right]^{-1}\v1}\\
	&=& \frac {\hat \Phi_z(0)}{1- \aaz(1+\psi_z)\v1^T\left[ \mI+\aaz \mJ + \mPi[0]\right]^{-1}\v1}.\nn
	\end{eqnarray}
Due to the simple form of $\mI$ and $\mJ$ and the symmetry of $\Pi^{\iota,\kappa}_z$ stated in Assumption \ref{assSym}, the sum of each column and the sum of each row of $\mI+\aaz \mJ + \mPi[0]$ equals $1+\aaz+\pi^\iota$. Thus, the one vector $\v1$ is an eigenvector of $\mI+\aaz \mJ + \mPi[0]$
corresponding to the eigenvalue $1+\aaz+\pi^\iota$ and we can compute that
	\begin{eqnarray*}
	\hat G_{z}(0) &=& \frac{\hat \Phi_z(0)}{1-  \aaz(1+\psi_z)\frac{2d}{1+\aaz+\pi^\iota_z}}\qquad
	\text{and}\qquad \hat B_{\lambda}(0) = \frac {1}{1-   \frac {2d\lambda} {1 +\lambda} }.
	\end{eqnarray*}
Solving $\hat B_{\lambda_z}(0)\hat \Phi_z(0)=\hat G_{z}(0)$ for $\lambda_z$ for the first equality, and for $\aaz$ for the second, gives the desired result.
\end{proof}

The above identification allows us to improve the bound on $f_1$ as required in the bootstrap analysis:

\begin{lemma}[Improvement of $f_1$]
\label{improvementBoundsf1}
Let $z\in(z_I,z_c)$ and $\gamma,\Gamma\in\Rbold^3$. If Assumptions \ref{assSym}-\ref{assDiagBounds} and condition $P(\gamma,\Gamma,z)$ hold, then $f_1(z)\leq \gamma_1$.
\end{lemma}

\begin{proof} Recall that $f_1(z)=\max\left\{(2d-1)\aabz,c_\mu (2d-1)\aaz\right\}$.
We select $\lambda_z$ as in Lemma $\ref{lemmaChoiceMu}$ and note that $\lambda_z<(2d-1)^{-1}$ if $z<z_c$ as $\hat G_{z}(0)=\hat B_{\lambda_z}(0)\hat \Phi_z(0)<\infty$.  Hence, we can compute that
	\eqan{
	\aaz(2d-1) &\stackrel{\text{Lemma }\ref{lemmaChoiceMu}}=
	 \frac {(1+\pi^{\iota}_z)(2d-1)\lambda_z}{1+\psi_z(1+\lambda_z)}
	\stackrel{\refeq{analys-assumed-Bound-PiPsi}} \leq
	 \frac {(1+\overline{\beta}_{\sss \Pi^{\iota}} )(2d-1)\lambda_z}{1-\underline {\beta}_{\sss \Psi^{\kappa}} (1+\lambda_z)}\nn\\
	&\stackrel{(2d-1)\lambda_z\leq 1} \leq \frac {1+\overline{\beta}_{\sss \Pi^{\iota}}} { 1 - \frac {2d}{2d-1}\underline {\beta}_{\sss \Psi^{\kappa}}}.
	}
Using $\frac {\aabz}{\aaz}\leq \betaaa$ from Assumption \ref{assDiagBounds}, we obtain
	\begin{eqnarray}
	\aabz(2d-1)  &\leq& \betaaa \frac {1+\overline{\beta}_{\sss \Pi^{\iota}}}
	{ 1 - \frac {2d}{2d-1}\underline {\beta}_{\sss \Psi^{\kappa}}}.
	\end{eqnarray}
Thus,
	\eqn{
	f_1(z)\leq \max\Big\{
	\betaaa \frac {1+\overline{\beta}_{\sss \Pi^{\iota}}}{ 1 - \frac {2d}{2d-1}\underline {\beta}_{\sss \Psi^{\kappa}}},
	c_\mu \frac {1+\overline{\beta}_{\sss \Pi^{\iota}}}  { 1 - \frac {2d}{2d-1}\underline {\beta}_{\sss \Psi^{\kappa}}}\Big\},
	}
which is by \refeq{conditionf1} smaller than $\gamma_1$ when condition $P(\gamma,\Gamma,z)$ holds.
\end{proof}

\subsection{Conditions for $f_2$}
\label{secf2}
In this section, we prove that the properties of $f_2$ in Proposition \ref{succesfulbootsprop} hold.
We start with the required continuity.

\begin{lemma}[Continuity of $f_2$]
\label{lemmaf2cont}
The function $z\mapsto f_2(z)$ defined in \refeq{defFunc2} is continuous for $z$ in $[0,z_c)$.
\end{lemma}

\begin{proof} We follow the proof of \cite[Lemma 5.3]{HeyHofSak08}.
To show that $f_2$ is continuous on $[0,z_c)$, we prove that it is continuous on the closed interval $[0,z_c-\varepsilon]$ for any $\varepsilon>0$. Using Assumption \ref{assGzBehavoir}, we know that for any $k$ and $z\in[0,z_c-\varepsilon]$,
	\begin{eqnarray}
	\left|\frac {d}{dz}\hat G_z(k)\right|&=&\Big|\sum_{x}\e^{\ii k\cdot x}\frac {d}{dz} G_z(x)\Big|
	\leq 	\sum_{x}\frac {d}{dz} G_z(x)=\frac {d}{dz} \hat G_z(0)\nnb
	&\leq& c_\varepsilon (\hat G_z(0))^2\leq c_\varepsilon (\hat G_{z_c-\varepsilon}(0))^2,
	\end{eqnarray}
where we can interchange differentiation and summation as the sum is bounded in absolute value, as just shown. From this, we conclude that the derivative of $f_2(z)$ is uniformly bounded on $[0,z_c-\varepsilon]$, which implies the continuity of $f_2$ on $[0,z_c-\varepsilon]$.
\end{proof}

We continue to prove the bootstrap for $f_2$:
\begin{lemma}[Improvement of $f_2$]
\label{improvementBoundsf2}
Let $z\in[z_I,z_c)$ be such that Assumptions \ref{assSym}-\ref{assDiagBounds} and $P(\gamma,\Gamma,z)$ hold. Then $f_2(z)\leq \gamma_2$.
\end{lemma}

\begin{proof}
Recall that $f_2(z)=\frac{2d-1}{2d-2}\sup_{k\in(-\pi,\pi)^d} [1-\hat D(k)]\genG[k].$
As already used in the proof of Theorem \ref{centralresult}, we know that Assumption \ref{assDiagBounds} implies that
	\begin{align}
	|\hat \Phi_z(k)| &\leq \upcp+\betaap+\betaRp,
	\end{align}
and
	\begin{align}	
	1-\hat F_z(k) &= \hat \Phi_z(0)\hat G_z(0)^{-1}+
	\afz[1-\hat D(k)]+ \hRfz(0)- \hRfz(k)\nnb
	&\geq (\lowaf-\betadeltaRfzlow)[1-\hat D(k)],
	\end{align}
where we use in the last step that $1-\hat F_z(0)=(\hat \Phi_z(0)\hat G^{-1}_{z}(0))=\hat B^{-1}_{\lambda_z}(0)\geq 0$ by Lemma \ref{lemmaChoiceMu}.
We conclude from this that
	\begin{align}
	\lbeq{f2-improve-lastline}
	|\hat G_z(k)| [1-\hat D(k)]&=\frac {|\hat \Phi_z(k)| [1-\hat D(k))]} {1-\hat F_z(k)}
	\leq \frac {\upcp+\betaap+\betaRp}{\lowaf-\betadeltaRfzlow}.
	\end{align}
If condition $P(\gamma,\Gamma,z)$ holds, then this is smaller than $\gamma_2$, see \refeq{conditionf2}, which completes the proof.
\end{proof}


\subsection{Conditions for $f_3$}
\label{secf3}
In this section, we show that the function $f_{3}$, defined in \refeq{defFunc3}, satisfies the conditions of the Bootstrap Lemma (Lemma \ref{bootstraplemma}). Namely, we prove that $z\mapsto f_{3}(z)$ is continuous, that $f_3(z_I)\leq \gamma_3$, and that for all $z\in(z_I,z_c)$, $f_i(z)\leq \Gamma_i$ for all $i\in\{1,2,3\}$ implies that $f_3(z)\leq \gamma_3$.
As this is more elaborate for $f_3$ than for $f_1$ and $f_2$, we divide the proof into multiple steps.

The techniques of this section are an adaptation of those used by Hara and Slade to prove the mean-field behavior for SAW in $d\geq 5$, see \cite{HarSla92a}. The central idea needed for the adaptation was developed in discussions with Takashi Hara.

\subsubsection{Rewrite of $f_3$}
We analyze the function
	\begin{eqnarray}
	\lbeq{Xspace-Analysis-Object}
	\Hcal^{n,l}_z(x)=\sum_{y}\|y\|_2^2 G_z(y)(G_z^{\star n}\star D^{\star l})(x-y),
	\end{eqnarray}
and conclude the desired results on $f_3$ using that, by definition \refeq{defFunc3},
	\begin{eqnarray}
	f_3(z)=\max_{\{n,l,S\}\in \mathcal{S}} \frac{\sup_{x\in S} \Hcal^{n,l}_z(x)}{c_{n,l,S}}.
	\end{eqnarray}
We bound $\Hcal^{n,l}_z$ using the continuous Laplace operator $\bigtriangleup$. For a differentiable function $g$ and $s\in\{1,2,\dots,d\}$, let $\partial_s g(k)=\frac \partial {\partial k_s} g(k)$  and $\bigtriangleup g(k)=\sum_{s=1}^d\partial_s^2 g(k)$. Then,
	\begin{eqnarray}
	\lbeq{Xspace-Analysis-deltaArgument}
	\partial_s^2 \hat G_z(k)
	&=&\sum_{x\in\Zd} G_z(x) \partial_s^2 \e^{\ii k\cdot x}
	=-\sum_{x\in\Zd} x_s^2 \hat G_z(k)\e^{\ii k\cdot x},\\
	\lbeq{Xspace-Analysis-deltaArgument-2}
	\bigtriangleup \hat G_z(k)&=&-\sum_{x\in\Zd} \|x\|^2_2 G_z(x)\e^{\ii k\cdot x}.
	\end{eqnarray}
Thus, we can bound $\Hcal^{n,l}_z(x)$ using the Fourier representation
	\begin{eqnarray}
	\lbeq{Xspace-Analysis-Tool}
	\label{Xspace-Analysis-Tool-Label}
	\Hcal^{n,l}_z(x)=&\int_{(-\pi,\pi)^d}(-\bigtriangleup \hat G_z(k))
	\hat D^{l}(k)\hat  G^{n}_z(k)\e^{-\ii k\cdot x}\frac {d^dk}{(2\pi)^d}.
	\end{eqnarray}
If we replace $G_z$ in \refeq{Xspace-Analysis-Tool} by $C_{1/2d}$ (recall \refeq{genSRW}),
then we can compute the value directly, see Section \ref{secAnalysisXspaceF3Initial}.
To obtain a bound for $z\in(z_I,z_c)$, in Sections \ref{secAnalysisXspaceF3Decomposition}-\ref{subsecImproF3bar}, we extract a dominant SRW-like contribution from $G_z$, which we compute directly and then we bound the remainder terms separately. The bounds are expressed using several SRW-integrals that can also be computed numerically as we explain in Section \ref{secNumerics}.

\subsubsection{Continuity of $f_3$}

\begin{lemma}[Continuity]
\label{lemmaf3cont-version2}
The function $z\mapsto f_3(z)$ as defined in \refeq{Xspace-Analysis-Object} is continuous in $z\in[z_I,z_c)$.
\end{lemma}

\begin{proof}
We fix an $\varepsilon>0$ and prove that $(\Hcal^{n,l}_z(x))_{x\in\Zd}$ is an equicontinuous family of functions and is uniformly bounded for all $x,n,l$ for all $z\in[0,z_c-\varepsilon)$.
This allows us to obtain the continuity of $z\mapsto \sup_{x\in S} \Hcal^{n,l}_z(x)$ for all sets $S$
directly from the Arzela-Ascoli Theorem. This implies the continuity of $z\mapsto f_3(z)$ as the index set $\mathcal{S}$ over which we take the maximum in \refeq{defFunc3} is finite. By Assumption \ref{assGzBehavoir}, there exists a constant $K(z_c-\varepsilon)<\infty$ such that
	\begin{eqnarray}
	\sum_{x\in\Zd} \|x\|_2^2G_{z_c-\varepsilon}(x)<K(z_c-\varepsilon).
	\end{eqnarray}
Further, $\hat G_{z_c-\varepsilon}(0)=\chi(z_c-\varepsilon)<\infty$, so that, uniformly for $z\in[z_I,z_c-\vep]$,
	\begin{align}
	\Hcal^{n,l}_z(x) \leq& \sup_{y} (G_z^{\star n}\star D^{\star l})(y) K(z_c-\varepsilon)
	\leq \chi(z_c-\varepsilon)^n K(z_c-\varepsilon).
	\end{align}
By Assumption \ref{assGzBehavoir},
	\begin{eqnarray}
	\frac d{dz} \Hcal^{n,l}_z(x)\leq
	& c_{\varepsilon} \sum_{y} \|y\|_2^2(G_z\star D \star G_z) (y)(G_z^{\star n}\star D^{\star l})(x-y)\nnb
	&+ n c_{\varepsilon} \sum_{y} \|y\|_2^2G_z (y)(G_z^{\star (n+1)}\star D^{\star (l+1)})(x-y).
	\end{eqnarray}
We use that $\|w+x+y\|^2_2\leq 3 (\|w\|^2_2+\|x\|^2_2+\|y\|^2_2)$ for all $w,x,y\in\Zd$ to obtain
	\begin{align}
	\sum_{y\in\Zd}\|y\|_2^2(G_z\star D \star G_z) (y)
	&\leq 3\sum_{y,w} \|w\|_2^2 G_z(w)(D \star G_z) (y-w)\nnb
    &\quad +3\sum_{y,w} (G_z \star G_z) (y-w)D(w)\|w\|_2^2\nnb
	&\quad + 3\sum_{y,w}(G_z\star D)(w) G_z (y-w)\|y-w\|_2^2\nnb
	&\leq 6 K(z_c-\varepsilon) \hat G_{z_c-\varepsilon}(0)+3\hat G_{z_c-\varepsilon}(0)^2.
	\end{align}
We conclude that
	\begin{eqnarray}
	\frac d{dz} \Hcal^{n,l}_z(x) &\leq c_{\varepsilon} ( 6 K(z_c-\varepsilon)+3 G_{z_c-\varepsilon}(0) +
	n K(z_c-\varepsilon) ) \hat G_{z_c-\varepsilon}(0)^{n+1}<\infty.
	\end{eqnarray}
By the uniformity of this bound in $x$, we conclude that $(\Hcal^{n,l}_z(x))_{x\in\Zd}$ is equicontinuous.
\end{proof}

\subsubsection{Bound for the initial point $f_3(z_I)$}
\label{secAnalysisXspaceF3Initial}
In this section, we prove that $f_3(z_I)\leq \gamma_3$. By Assumption \ref{assXBoundInitialCondition}, we can bound $G_{z_I}(x)\leq \frac {2d-2}{2d-1} C_{1/(2d)}(x)$.  Thus, we can bound $f_{3}(z_I)$ using only SRW-quantities. We start by computing the derivatives of $\hat C_{1/2d}(k)=\hat C(k)=[1-\hat D(k)]^{-1}$  and $\hat D(k)$, where $\partial_s$ denotes the derivative w.r.t.\ $k_s$:
	\begin{align}
	\sum_{s=1}^d \partial_s \hat C(k)&=\sum_{s=1}^d\frac{\partial_s \hat D(k)}{[1-\hat D(k)]^2}
	=- 	\frac{\frac {1}{d}\sum_{s=1}^d \sin(k_{s})}{[1-\hat D(k)]^2},\\
	\lbeq{DeltaC-Computation}
	\bigtriangleup \hat C(k)&=\sum_{s=1}^d \Bigg(\frac{\partial_s^2 \hat D(k)}{[1-\hat D(k)]^2}
	+ 2 	\frac{(\partial_s \hat D(k))^2}{[1-\hat D(k)]^3}\Bigg),
	\end{align}
and
	\begin{eqnarray}
	\bigtriangleup \hat D(k)&=&\sum_{s=1}^d \partial_s^2\hat D(k)
	=-\frac 1 d \sum_{s=1}^d  \cos(k_{s})=-\hat D(k),\\
	\lbeq{Dsin-def}
	\sum_{s=1}^d (\partial_s \hat D(k))^2&=&\frac 1 {d^2} \sum_{s=1}^d \sin^2(k_{s}):=\hat D^{\sin}(k).
	\end{eqnarray}
We use
	\begin{align}
	\sin^2(k_{s})
	=-\frac 1 4 \left(\e^{\ii k_{s}}-\e^{-\ii k_{s}}\right)^2=\frac 1 2 - \frac 1 4 \e^{2\ii k_{s}}
	- \frac 1 4 \e^{-2\ii k_{s}}=\frac 1 2 \left[1-\cos(2k_{s})\right]
	\end{align}
to compute that
	\begin{align}
	\hat D^{\sin}(k)=\frac d {2d^2} -\frac 1 {4d^2}\sum_{\iota} \e^{-2\ii k_\iota}
	=\frac 1 {2d} [1 -\hat D(2k)].
	\lbeq{Dsin-Split}
	\end{align}
We define $\hat M(k)=\hat D(k)-2 \hat D^{\sin}(k)\hat C(k)$ and conclude from the computations above that
	\begin{align}
	\bigtriangleup \hat C(k)&=-\hat C(k)^2\hat M(k)=-\hat D(k)\hat C(k)^2
	+\frac 1 d \hat C(k)^3-\frac 1{2d^2} \sum_{\iota} \e^{-2 \ii k_\iota}
	\hat C(k)^3.
	\lbeq{XspaceDecomp-tmp8}
	\end{align}
We use this representation to compute the SRW analogue of $\Hcal^{n,l}_z(x)$ as
	\begin{align}\nn
	\sum_{y}&\|y\|_2^2C(y)(D^{\star l}\star C^{\star n})(x-y)\\
	\lbeq{XspaceDecomp-tmp9}
	&=\int_{(-\pi,\pi)^d}\hat D^l(k)\hat  C^{n+2}(k) \hat M(k)\e^{-\ii k\cdot x}\frac {d^dk}{(2\pi)^d} \\
	\lbeq{SRWDispFourierEquation}
	&=\int_{(-\pi,\pi)^d}\hat D^l(k)\hat  C^{n+2}(k) \left(\hat D(k)-\frac 1 d \hat C(k) +
	\frac 1 {2d^2} 	\sum_{\iota} \e^{2\ii k_\iota}\hat C(k)\right)\e^{-\ii k\cdot x}\frac {d^dk}{(2\pi)^d}.
	\end{align}
As we explain in more detail in Section \ref{secNumericsSRW}, we can numerically compute the SRW-integral
	\begin{eqnarray}
	\lbeq{defInlPerview}
	I_{n,l}(x):=(D^{\star l} \star C^{\star n}_{1/(2d)})(x)
	=\int_{(-\pi,\pi)^d} \frac {\hat D^l(k)}{[1-\hat D(k)]^n}\e^{-\ii k\cdot x}\frac {d^dk}{(2\pi)^d}.
	\end{eqnarray}
We use this integral to compute \refeq{SRWDispFourierEquation} numerically as
	\begin{align}
	\lbeq{SRWDispUpperBound}
	\refeq{SRWDispFourierEquation}
	&=I_{n+2,l+1}(x)-\frac  1 d I_{n+3,l}(x)+\frac 1 {2d^2} \sum_{\iota}I_{n+3,l}(x+2\ve[\iota]):=\Isupx_{n,l}(x).
	\end{align}
Thus, $f_3(z_I)$ is bounded by
	\begin{align}
	\lbeq{conditionf3-two-initial-statement}
	f_3(z_I)&\leq \frac {2d-2}{2d-1}
	\max_{\{n,l,S\}\in \mathcal{S}} \frac{ \sup_{x\in S}\Isupx_{n,l}(x)}{c_{n,l,S}},
	\end{align}
By the assumption in $P(\gamma,\Gamma,z)$, this is smaller than $\gamma_3$.

\begin{remark}[Close to the upper critical dimension]
\label{rem-closetodc}
{\rm
The bound \refeq{SRWDispUpperBound} can only be used in dimension $d\geq 2(n+3)+1$ as it uses $I_{n+3,l+1}(x)$, which is only finite in these dimensions. This restricts the analysis shown here to dimensions $d\geq d_c+3$, e.g.\ for percolation we can only use this bound for $d\geq 9$ as we require a bound on $\Hcal^{1,0}_z(x)$ to successfully apply the bootstrap argument. This problem can be avoided using a different bound for the integral. For example, using the bound
	\begin{align}
	\lbeq{Dsinabsolute-bound}
	 \hat D^{\sin}(k)=\frac 1 {d^2} \sum_{s=1}^d [1-\cos^2(k_{s})]
	\leq \frac 2 {d^2} \sum_{s=1}^d [1-\cos(k_{s})]\leq \frac 2 {d} [1-\hat D(k)]
	\end{align}
in \refeq{XspaceDecomp-tmp9}, we obtain that
	\begin{align}
	\sum_{y}&\|y\|_2^2C(y)(D^{\star l}\star C^{\star n})(x-y)\leq I_{n+2,l+1}(x)+\frac  4 d K_{n+2,l}(x),
	\end{align}
where we introduce the SRW-integral $K_{n+2,l}(x)$ in \refeq{Analysis-Def-Integral-K} below.
This and other bounds applicable in $d=d_c+1,d_c+2$, perform \emph{numerically} worse than the bound in \refeq{conditionf3-two-initial-statement}.
As we are not able to prove mean-field behavior in dimension $d_c+1,d_c+2$ anyway, we use the numerically better bound \refeq{conditionf3-two-initial-statement} instead.}
\end{remark}

\subsubsection{Preparations for the improvement of bounds for $f_3$}
\label{secAnalysisXspaceF3Decomposition}
We wish to prove for all $z\in(z_I,z_c)$ that $f_i(z)\leq \Gamma_i$ for all $i\in\{1,2,3\}$ implies $f_3(z)\leq \gamma_3$. This is the most technical part of our proof.
We do this by deriving a bound on
	\begin{align}
	\tag{\ref{Xspace-Analysis-Tool-Label}}
	\Hcal^{n,l}_z(x)=&\int_{(-\pi,\pi)^d}(-\bigtriangleup \hat G_z(k))\hat D^{l}(k)
	\hat  G^{n}_z(k)\e^{-\ii k\cdot x}\frac {d^dk}{(2\pi)^d}.
	\end{align}
For this bound, we extract the SRW-like contributions from $(-\bigtriangleup \hat G_z(k))$ by decomposing it into five terms $H_1,\dots,H_5$.
We compute the SRW-like contributions $H_1$ as in the preceding section and bound the remainder terms using Assumption \ref{assDiagBounds} and certain SRW-integrals that we define next.

\paragraph{SRW-integrals.}
Here we introduce several SRW integrals that we use to bound $\Hcal^{n,l}_z(x)$. Using the terminology introduced in Definition \ref{defSignedPermuations}, we define
	\begin{eqnarray}
	\lbeq{Analysis-dhx}
	\hdx &=&\frac 1 {2^d d!} \sum_{\nu\in \mathcal{P}_d }
	\sum_{\delta\in\{-1,1\}^d} \e^{\ii k \cdot p(x;\nu,\delta)}.
	\end{eqnarray}
Performing the sum over $\delta$ gives cosines, so that $\hdx$ is real.
The following SRW-integrals are adaptations of the integrals used in \cite[Section 1.6]{HarSla92b}: For $x\in\Zd$ and $n,l\in\Nbold$, we let
\begin{eqnarray}
	\lbeq{Analysis-Def-Integral-I}
	I_{n,l}(x)&=&\int_{(-\pi,\pi)^d}\hat D(k)^l
	\hat C(k)^n\hdx \frac {d^dk}{(2\pi)^d},\\
	\lbeq{Analysis-Def-Integral-K}
	K_{n,l}(x)&=&\int_{(-\pi,\pi)^d}|\hat D(k)|^l
	\hat C(k)^n |\hat D^{(x)}(k)|\frac {d^dk}{(2\pi)^d},\\
	\lbeq{Analysis-Def-Integral-T}
	T_{n,l}(x)&=&\int_{(-\pi,\pi)^d}|\hat D^{l} (k)|
	\hat C(k)^n |\hat D^{(x)}(k)||\hat M(k)|\frac {d^dk}{(2\pi)^d},\\
	\lbeq{Analysis-Def-Integral-U}
	U_{n,l}(x)&=&\int_{(-\pi,\pi)^d}|\hat D^{l} (k)|
	\hat C(k)^n|\hat D^{(x)}(k)||\hat D^{\sin} (k)|\frac {d^dk}{(2\pi)^d},
\end{eqnarray}
where $\hat C(k)=\hat C_{1/2d}(k)$ is the critical SRW two-point function and $\hat{M}(k)$ is defined above \refeq{XspaceDecomp-tmp8}.
For any function $f$ such that $f(x)=f(p(x;\nu,\delta))$ for all $\nu,\delta$ (see Definition \ref{defSignedPermuations}), we see that
 	\begin{eqnarray}
	\int_{(-\pi,\pi)^d}\hat f(k)\e^{-\ii k\cdot x}\frac {d^dk}{(2\pi)^d}=\int_{(-\pi,\pi)^d}\hat f(k)
	\hat D^{(x)}(k)\frac {d^dk}{(2\pi)^d}.
	\end{eqnarray}
The functions $G_z$ and $D$ have these symmetries, so that we can replace $\e^{\ii k\cdot x}$ in \refeq{Xspace-Analysis-Tool} by $\hdx$. In Section \ref{secNumericsSRW}, we show how to compute $I_{n,l}(x)$, and in Section \ref{secNumericsAdvanced}, we bound the other integrals in terms of $I_{n,l}(x)$.

\paragraph{Decomposition of the two-point function.}
We decompose $\hat G_z(k)$ and $\bigtriangleup\hat G_z(k)$ into several pieces, which we then bound in the next section.
We start with some preparations for this decomposition. For the SRW-contributions, we define
\begin{align}
\lbeq{Analysis-Def-C-eps}
\hat  C^*(k)&=\frac {1} {1-\hat F_z(0)+\afz[1-\hat D(k)]}.
\end{align}
As $\hat F_z(0)\leq 1$ and $\afz>\lowaf$, we know that $\hat C^*(k)<\frac 1 {\afz} \hat C(k)<\lowaf^{-1} \hat C(k)$.
Further, we conclude from
	\eqan{
	\lbeq{Analysis-Xbound-C-proof}
	\hat C^*(k)&=\frac {1} {1-\hat F_z(0) +\afz}
	\frac 1 {1-\frac {\afz}{1-\hat F_z(0)+\afz}+\frac {\afz} {1-\hat F_z(0) +a}  [1-\hat D(k)]}\\
	&=\frac {1} {1-\hat F_z(0) +\afz } 	\hat C_{\lambda(z)}(k),\nn
	}
with
	\eqn{
	\lambda(z)=\frac 1 {2d}\frac {\afz}{1-\hat F_z(0)+\afz},
	}
and the monotonicity of $C_\lambda(x)$ in the parameter of the generating function $\lambda$ that
\begin{align}
\lbeq{Analysis-Xbound-C-statement}
C^*(x)&\leq \frac {1} {1-\hat F_z(0) +\afz } C(x)\leq \frac {1} {\afz } C(x)\leq \frac 1 {\lowaf} C(x).
\end{align}
Thus, $C^*$ can be bounded by $\afz^{-1}C$ in $x$-space as well as in $k$-space.\\
We abbreviate
	\begin{align}
    \lbeq{FExpandDeltaFNotation}
	\deltaRfz&=\hRfz(0)-\hRfz(k),\qquad
	\deltaRpz=\hRpz(0)-\hRpz(k),\\
	\hat E (k)&=\frac {\deltaRfz\hat C^*(k)}{ 1-\hat F_z(k)},\qquad
	\hat M^*(k)=\hat D(k)-2 \hat D^{\sin}(k)\hat C^*(k).
	\end{align}
\paragraph{Bounds on key quantities.}
To bound the remainder terms, we define
	\begin{eqnarray}
	\lowK=\frac 1 {\lowaf-\betadeltaRfzlow},\quad \text{ so that }\quad \frac 1 {1-\hat F_z(k)}&\leq &\lowK \hat C(k).
	\end{eqnarray}
Further, we assume that $f_2(z)\leq\Gamma_2$, i.e., that
	\begin{eqnarray}
	\lbeq{Bound-on-Abs-Gz}
	|\hat G_z(k)|\leq \frac{2d-2}{2d-1}\Gamma_2\hat C(k),
	\end{eqnarray}
and from now on abbreviate $\Gamma_2'=\frac{2d-2}{2d-1}\Gamma_2$.
We use the bound \refeq{analys-assumed-displacement-1} in Assumption \ref{assDiagBounds} to obtain
	\begin{align}
	|\deltaRfz|&\leq [1-\hat D(k)]\sum_x \| x\|_2^2 |\Rfz(x)| \leq [1-\hat D(k)]\betadeltaRfz.
	\end{align}
Arguing as in \refeq{Xspace-Analysis-deltaArgument}-\refeq{Xspace-Analysis-deltaArgument-2} we can show that
\begin{align}
|\bigtriangleup \deltaRfz|&=\big|-\sum_x \| x\|_2^2 \Rfz(x)\e^{\ii k\cdot x} \big| \stackrel{\refeq{analys-assumed-displacement-1}}\leq \betadeltaRfz.
\end{align}
We conclude the same bounds for $\deltaRpz$, where $\betadeltaRfz$ is replaced by $\betadeltaRpz$.\\
Combining these bounds with $\hat C(k)=1/[1-\hat D(k)]$ we obtain
	\begin{align}
	\lbeq{Bound-on-EK}
	|\hat E(k)|\leq \left|\frac {\deltaRfz\hat C^*(k)}{ 1-\hat F_z(k)}\right|
	&\leq\betadeltaRfz\lowK \hat C^{*}(k)\leq \frac {\betadeltaRfz\lowK}{\lowaf } \hat C(k),
	\end{align}
and
	\begin{align}
	\lbeq{Bound-on-GzRemainder}
	|[\hRpz(k)- \deltaRfz \hat G_z(k)]\hat C^*(k)|
	&\leq \frac 1 {\afz} (\beta_{\sss R,\Phi}+ {\betadeltaRfz\Gamma'_2})\hat C(k),
	\end{align}
\paragraph{Decomposition of $\bigtriangleup\hat G_z(k)$.}
We decompose $\bigtriangleup\hat G_z(k)$ into five contributions $\hat H_i(k)$. The dominant contribution is $\hat H_1(k)$, which is defined to be
	\begin{align}
	\hat H_1(k)=& \left(\afz (\cpz+\apz\hat D(k))\hat C^*(k)+\apz\right)\hat C^*(k)\hat M^{*}(k).
	\lbeq{def-H1}
	\end{align}
The remainder terms $\hat H_2(k), \hat H_3(k), \hat H_4(k)$ and $\hat H_5(k)$ are defined as
	\begin{align}
	\hat H_2(k)=&- \Big(  \afz(\cpz+\apz\hat D(k))
	\Big( \hat C^*(k)+\frac {1}{1-\hat F_z(k)}\Big)+\apz\Big)\hat E(k) \hat M^{*}(k)\nnb
	&+\afz\frac{\hRpz(k)}{ (1-\hat F_z(k))^2}\hat M^{*}(k),
	\lbeq{def-H2} \\
	\hat H_3(k)=&2\frac {\hat D^{\sin}(k)}{1-\hat F_z(k)} \left(\afz\hat G_z(k)+\apz\right)\left(\hat E(k)-\frac {\afz-1} {1-F_z(k)}\right),\lbeq{def-H3}\\
	\hat H_4(k)=&-\frac{\bigtriangleup\hRpz(k)}{1-\hat F_z(k)}- \frac{\bigtriangleup\hRfz(k)}{1-\hat F_z(k)}\hat G_z(k),
	\lbeq{def-H4}\\
	\hat H_5(k)=&
	-2\frac {\sum_{s=1}^d (\partial_{s}\hRfz(k))^2+2\afz\partial_{s}\hat D(k)\partial_{s}\hRfz(k)}
	{(1-\hat F_z(k))^2}\hat G_z(k)\nnb
	&-\frac {2}{(1-\hat F_z(k))^2} \sum_{s=1}^d
	\left( \partial_{s} \hRpz(k)\afz \partial_{s} \hat D(k)+\partial_{s}\hat \Phi_z(k)
	\partial_{s} \hRfz(k)\right).
	\lbeq{def-H5}
	\end{align}
In Appendix \ref{sec-decomposition-f3}, we explicitly show that
	\begin{eqnarray}
	\lbeq{sec-decomposition-f3-statement}
	-\bigtriangleup \hat G_z(k)&=&\sum_{i=1}^5 \hat H_i(k).
	\end{eqnarray}
This computation is quite long and tedious. However, since it is crucial to our analysis, we give the derivation in detail in Appendix \ref{sec-decomposition-f3}. Let us now give some insight into the origin of the different contributions. The first term $\hat H_1(k)$ is a SRW-like contribution that can be bounded similarly as in Section \ref{secAnalysisXspaceF3Initial}. The second term $\hat H_2(k)$ corresponds to everything that has the factor $\hat M^{*}(k)$ and a remainder term. In $\hat H_3(k)$, we collect the remaining $\hat D^{\sin}(k)$ contributions. In $\hat H_4(k)$, we put the contributions of $\bigtriangleup\hRpz(k)$ and $\bigtriangleup\hRfz(k)$, and in $\hat H_5(k)$, we collect all products of single derivatives.

\subsubsection{Improvement of bounds  for $f_3$}
\label{subsecImproF3bar}
In this section we bound $\Hcal^{n,l}_{z}(x)$ by deriving bounds on
	\begin{eqnarray}
	\lbeq{Xspace-Analysis-Tool-repeat}
	\Hcal^{n,l}_{i,z}(x)= \int_{(-\pi,\pi)^d} \hat H_i(k)\hat D^{l}(k)
	\hat  G^{n}_z(k)\hdx \frac {d^dk}{(2\pi)^d},
	\end{eqnarray}
for $i=1,\dots,5$. We do this for $i=1,2,\ldots, 5$ one by one, starting with $\Hcal^{n,l}_{1,z}(x)$. This is the most technical part of the analysis.
We bound each term of $\Hcal^{n,l}_{i,z}(x)$ using the bounds of Assumption \ref{assDiagBounds} and the SRW-integrals
\refeq{Analysis-Def-Integral-I}-\refeq{Analysis-Def-Integral-U}.
\paragraph{Step 1: Bound on $\Hcal^{n,l}_{1,z}(x)$.}
We first recall the rearrangement of \refeq{SRWDispFourierEquation}-\refeq{SRWDispUpperBound} to see that, for $m\geq 0$,
	\begin{align}
	\int_{(-\pi,\pi)^d} \hat D(k)^l \hat C^*(k)^{m+2} \hat M^{*}(k) \e^{\ii k\cdot x}\frac {d^dk}{(2\pi)^d}
      =&\sum_{y}\|y\|_2^2C^{*}(y)(D^{\star l}\star (C^{*})^{\star m})(x-y)\nnb
	\stackrel{\refeq{Analysis-Xbound-C-statement}}\leq & (\afz)^{-(m+1)}\sum_{y}\|y\|_2^2C(y)(D^{\star l}\star C^{\star m})(x-y)\nnb
	=& (\afz)^{-(m+1)} \Isupx_{m,l}(x).
	\lbeq{SRWboundsTmpss}
	\end{align}
We perform the bounds for $n=0$, $n=1$ and $n=2$ separately.
\paragraph{Bound on a weighted line ($n=0$).}
Substituting the definition of $\Hcal^{0,l}_{1,z}(x)$, we see that \refeq{def-H1} leads to three terms.
We bound the first and second term using  \refeq{SRWboundsTmpss}.
For the third term we can not use an $x$-space representation like in \refeq{SRWboundsTmpss}.
To bound this term we repeat \refeq{XspaceDecomp-tmp9}-\refeq{SRWDispUpperBound} to see that
\begin{align}
	\int_{(-\pi,\pi)^d} &\hat D(k)^l \hat C^*(k) \hat M^{*}(k) \e^{\ii k\cdot x}\frac {d^dk}{(2\pi)^d}\nnb
	=&  (D^{\star (l+1)}\star C^*)(x)-\frac 1 {d} (D^{\star l}\star  C^*\star  C^*)(x)
	+\frac {1} {2d^2} \sum_{\iota} (D^{\star l}\star  C^*\star  C^*) (x+2\ve[\iota])\nnb
	\stackrel{\refeq{Analysis-Xbound-C-statement}}\leq &\frac 1 {\afz} I_{1,l+1}(x)+\frac 1 {2d^2} \frac {1} {\afz^2} \sum_{\iota}I_{2,l}(x+2\ve[\iota]).
\end{align}
In this way, we obtain
\begin{align}
	\lbeq{term-Bound-H1-line}
	\Hcal^{0,l}_{1,z}(x)\leq & \upcp \Isupx_{0,l}(x)+\betaap \Isupx_{0,l+1}(x)
	+\frac {\betaap} {\lowaf}I_{1,l+1}(x)+\frac 1 {2d^2} \frac {\betaap} {\lowaf^2}\sum_{\iota}I_{2,l}(x+2\ve[\iota]).
\end{align}

\paragraph{Bound on a weighted bubble ($n=1$).}
To bound $\Hcal^{1,l}_{1,z}(x)$, we expand $\hat G_z(k)$ as follows:
	\begin{align}
	\lbeq{term-toBound-H1-bubble}
	\hat G_z(k)=& (\cpz+\apz\hat D(k))\hat C^*(k)+
	\left(\hRpz(k) - \deltaRfz\hat G_z(k)\right)\hat C^*(k),
	\end{align}
so that
	\begin{align}
	\Hcal^{1,l}_{1,z}(x)
	=&\int_{(-\pi,\pi)^d}(\cpz+\apz\hat D(k)) \hat H_{1}(z)\hat C^{*}(k)
	\hdx\hat D^l(k)\frac {d^dk}{(2\pi)^d}\nnb
	&+\int_{(-\pi,\pi)^d}\left( \hRpz(k)- \deltaRfz \hat G_z(k)\right)\hat C^*(k)\hat H_{1}(z)
	\hdx\hat D^l(k)\frac {d^dk}{(2\pi)^d}.
	\end{align}
We bound the first line using \refeq{SRWboundsTmpss} and the second line using \refeq{Bound-on-GzRemainder} to obtain
	\begin{align}
	|\Hcal^{1,l}_{1,z}(x)|
	\leq & \lowaf^{-1} (\upcp)^2\Isupx_{1,l}(x) +\upcp \lowaf^{-1} \betaap\Isupx_{0,l}(x)+2\upcp\lowaf^{-1}\betaap\Isupx_{1,l+1}(x)\nnb
	&+(\betaap)^2\lowaf^{-1} \Isupx_{0,l+1}(x)+\lowaf^{-1}(\betaap)^2\Isupx_{1,l+2}(x)\nnb
	&+\frac {\beta_{\sss R,\Phi} + \betadeltaRfz\Gamma'_2}{\lowaf^2}
	(\upcp T_{3,l}(x)+\betaap T_{3,l+1}(x)+\betaap T_{2,l}(x)),
	\lbeq{term-Bound-H1-bubble-bound}
	\end{align}
with $T_{n,l}$ as defined in \refeq{Analysis-Def-Integral-T}.

\paragraph{Bound on a weighted triangle ($n=2$).}
We decompose $\hat G_z^2(k)$ into two terms as
	\begin{align}
	\lbeq{term-toBound-H1-triangle-t1}
	\hat G_z(k)^2=& \hat C^*(k)^2(\cpz+\apz\hat D(k))^2 \\
	\lbeq{term-toBound-H1-triangle-t2}
	&+\left[\hRpz(k)- \deltaRfz \hat G_z(k)\right]\hat C^*(k)
	\left[ (\cpz+\apz\hat D(k))\hat C^*(k)+ \hat G_z(k) \right].
	\end{align}
We compute the contribution \refeq{term-toBound-H1-triangle-t1} to be
	\begin{align}
	\hat H_{1}(k)&\hat C^*(k)^2(\cpz+\apz\hat D(k))^2\nnb
	&=(\cpz+\apz\hat D(k))^2 \big[\afz(\cpz+\apz\hat D(k)) \hat C^*(k)+\apz\big]
	\hat M^{*}(k)\hat C^*(k)^3.
	\end{align}
We expand the brackets and then use \refeq{SRWboundsTmpss} to bound this contribution by
	\begin{align}
	  	\lbeq{term-toBound-H1-triangle-t1-bound}
	\int_{(-\pi,\pi)^d}&\hat H_{1}(k)\hat C^*(k)^2(\cpz+\apz\hat D(k))^2
	\hdx\hat D^l(k)\frac {d^dk}{(2\pi)^d}\\
	\leq&\lowaf^{-2} (\upcp)^2  \left[ \upcp \Isupx_{2,l}(x) +\betaap\Isupx_{1,l}(x)
	+3\betaap \Isupx_{2,l+1}(x)\right]\nnb
	&+\lowaf^{-2}(\betaap)^2\upcp \left[ 2\Isupx_{1,l+1}(x)+3\Isupx_{2,l+2}(x)\right]
	+\lowaf^{-2} (\upap)^3\left[ \Isupx_{2,l+3}(x)+\Isupx_{1,l+2}(x)\right],\nn
	\end{align}
where we use that  $\Isupx_{n,l}(x)\geq 0$ for every $x\in \Zd, n,l\geq 0$ by \refeq{SRWDispFourierEquation}--\refeq{SRWDispUpperBound}.

Using \refeq{Bound-on-GzRemainder} we bound the absolute value of the minor contributions given in \refeq{term-toBound-H1-triangle-t2} by
	\begin{align}
	\lbeq{term-toBound-H1-triangle-t2-help}
 	& \frac {\beta_{\sss R,\Phi}+\betadeltaRpz\Gamma_2'} {\afz}
	\left( \frac {(\upcp+\betaap |\hat D(k)|)}{\afz}+\Gamma_2' \right)\hat C(k)^2.
	\end{align}
Thus, we bound the contributions due to \refeq{term-toBound-H1-triangle-t2} by
	\begin{align}
	&\int_{(-\pi,\pi)^d} |\hat H_1(k)|\hat D^{l}(k)\frac {\beta_{\sss R,\Phi}+\betadeltaRpz\Gamma_2'} {\afz}
	\left( \frac {(\upcp+\betaap |\hat D(k)|)}{\afz}+\Gamma_2' \right)\hat C(k)^2 |\hdx| \frac {d^dk}{(2\pi)^d}
	\nnb
	\lbeq{term-toBound-H1-triangle-t2-bound}
	\leq&
\frac {\beta_{\sss R,\Phi}+\betadeltaRpz\Gamma_2'} {\lowaf^2}
\left( \frac {\upcp}{\lowaf}+\Gamma_2' \right)
 \Big[\upcp T_{4,l}(x)+\betaap T_{4,l+1}(x)+\betaap T_{3,l}(x)\Big]\\
	&+\betaap \frac {\beta_{\sss R,\Phi}+\betadeltaRpz\Gamma_2'} {\lowaf^3}
 \Big[\upcp T_{4,l+1}(x)+\betaap T_{4,l+2}(x)+\betaap T_{3,l+1}(x)\Big].\nn
	\end{align}

\paragraph{Conclusion of Step 1.}
We have bounded the contribution due to $\hat H_1(k)$ and have obtained that
	\begin{align}
	\lbeq{term-toBound-H1}
	|\Hcal^{n,l}_{1,z}(x)|\leq
	\begin{cases}
	\refeq{term-Bound-H1-line}&\text{ for }n=0,\\
	\refeq{term-Bound-H1-bubble-bound} &\text{ for }n=1,\\
	\refeq{term-toBound-H1-triangle-t1-bound}+\refeq{term-toBound-H1-triangle-t2-bound}
	&\text{ for }n=2.
	\end{cases}
	\end{align}
By the sum of two equation numbers we mean the sum of the terms given in the right-hand sides of the corresponding equations. As for $z=z_I$, this bound uses $I_{n+2,l}(x)$ and can therefore not be used in $d=d_c+1,d_c+2$. We have chosen to use these bounds, even if other bounds would be available, as they give numerically better bounds.

\paragraph{Step 2: Bound on $\Hcal^{n,l}_{2,z}(x)$.}
For this bound we use
\begin{eqnarray}
  	\lbeq{Analysis-Def-Integral-Tstar}
	T^*_{n,l}(x)&=&\int_{(-\pi,\pi)^d}|\hat D^{l} (k)| \hat C(k)^n |\hat D^{(x)}(k)||\hat M^*(k)|\frac {d^dk}{(2\pi)^d},
\end{eqnarray}
which is an adaptation of $T_{n,l}$ defined in \refeq{Analysis-Def-Integral-T}.
In Section \ref{secNumericsAdvanced} we will bound $T^*_{n,l}$ in the same way as $T_{n,l}$.
We bound the absolute value of $\hat H_2(k)$, defined in \refeq{def-H2}, by
\begin{align}
	|\hat H_2(k)|\leq& \Big( \afz(\ \upcp+\betaap|\hat D(k)|)\Big( \afz^{-1} + \lowK \Big)\hat C(k)+\betaap\Big)
	 \betadeltaRfz\lowK \frac {\hat C(k)} {\afz} | \hat M^{*}(k)|\nnb
	&+\upaf \betadeltaRpz \lowK ^2 \hat C(k)^2 |\hat M^{*}(k)|.
\end{align}
We use that $|\hat  G_z(k)|\leq \Gamma_2'\hat C(k)$ and the integrals $T^*_{n,l}$ to bound $\Hcal^{n,l}_{2,z}(x)$ by
\begin{align}
	\lbeq{term-toBound-H2-bound}
	|\Hcal^{n,l}_{2,z}(x)|\leq\
	&\betadeltaRfz\lowK (\Gamma_2')^n\Big[(\upcp T^*_{n+2,l}(x) +\betaap T^*_{n+2,l+1}(x))\left( \lowaf^{-1} +\lowK \right)\\
	&\qquad\qquad+ \betaap\lowaf^{-1} T^*_{n+1,l}(x) \Big]
	+\upaf \betadeltaRpz \lowK ^2 T^*_{n+2,l}(x).\nn  	
\end{align}
\paragraph{Step 3: Bound on $\Hcal^{n,l}_{3,z}(x)$.}
In \refeq{def-H3}, we have defined $\hat H_3(k)$ to be
	\begin{align}
	\hat H_3(k)=&2\frac {\hat D^{\sin}(k)}{1-\hat F_z(k)} \left(\afz\hat G_z(k)+\apz\right)\left(\hat E(k)-\frac {\afz-1} {1-F_z(k)}\right).
	\end{align}
We bound $|\hat H_3(k)|$ as
	\begin{align}
	|\hat H_3(k)|\leq &\ 2\hat D^{\sin}(k) \lowK \hat C(k)  \left(\afz\Gamma_2'\hat C(k)+\apz\right)\lowK \hat C(k) \left(\frac{\betadeltaRfz}{\afz}+|\afz-1|\right).
	\end{align}
and use this bound and the integral $U_{n,l}$, defined in \refeq{Analysis-Def-Integral-U}, to bound $|\Hcal^{n,l}_{3,z}(x)|$ as follows:
	\begin{align}
	|\Hcal^{n,l}_{3,z}(x)|\leq
	\lbeq{term-toBound-H3-bound}
	&\ 2(\Gamma_2')^{n+1}\lowK^2\left(\betadeltaRfz+\upaf \max\{|\upaf-1|,  |\lowaf-1|\} \right)U_{n+3,l}(x)\\
    &+2(\Gamma_2')^n \lowK^2 \betaap\left(\betadeltaRfz\lowaf^{-1}+\max\{|\upaf-1|,  |\lowaf-1|\} \right)U_{n+2,l}(x).\nn
	\end{align}
\paragraph{Step 4: Bound on $\Hcal^{n,l}_{4,z}(x)$.}
We first bound $\hat H_4(k)$ in Fourier space as
	\begin{align*}
	|\hat H_4(k)|&\leq  \lowK \left(\betadeltaRpz + \betadeltaRfz \Gamma_2' \hat C(k)\right).
	\end{align*}
Then, we use the definition of $K_{n,l}$ in \refeq{Analysis-Def-Integral-K} to bound
	\begin{align}
	|\Hcal^{n,l}_{4,z}(x)|\leq& \lowK
	\left(\betadeltaRpz K_{n,l}(x)+ \betadeltaRfz  \Gamma_2' K_{n+1,l}(x)\right).
	\lbeq{term-toBound-H4-bound}
	\end{align}
\paragraph{Step 5: Bound on $\Hcal^{n,l}_{5,z}(x)$.}
We recall that
	\begin{align}
	\hat H_5(k)=&-2\frac {\sum_{s=1}^d \partial_{s}\hRfz(k)(2\afz\partial_{s}\hat D(k)
	+\partial_{s}\hRfz(k) )}{(1-\hat F_z(k))^2}\hat G_z(k)\nnb
	\lbeq{term-toBound-H5-tmp}
	&-\frac {2}{(1-\hat F_z(k))^2} \sum_{s=1}^d\left( \partial_{s} \hRpz(k)
	\afz \partial_{s} \hat D(k)+
	\partial_{s}\hat \Phi_z(k) \partial_{s} \hRfz(k)\right).
	\end{align}
To bound the single derivatives we note that for a totally rotationally symmetric function $f$, see Definition \ref{defSignedPermuations}, the following holds:
	\begin{eqnarray}
	\partial_s \hat f(k)=\ii\sum_x x_s f(x)\e^{\ii k\cdot x}
	&=& -\sum_{x}f(x) x_s \sin(k_{s} x_s)\prod_{\nu\neq s}\cos(k_\nu x_\nu),
	\end{eqnarray}
for $s\in\{1,\dots,d\}$, so that
	\eqn{
	\lbeq{DifferenceToFirstMoment}
	|\partial_s \hat f(k)|
	\leq \sum_{x}|f(x)| |x_s \sin(k_{s} x_s)|.
	}
Since  $|\sin(n t)|\leq n|\sin(t)|$ for integer $n$, we obtain that
	\begin{eqnarray}
	\lbeq{DifferenceToFirstMoment-single}
	|\partial_s \hat f(k)|\leq |\sin(k_{s})| \sum_{x}|f(x)| x_s^2 .
	\end{eqnarray}
The total rotational symmetry of $f$ also implies that
	\begin{eqnarray}
	\lbeq{DifferenceToFirstMoment-symmetry}
	\sum_{x}|f(x)| x_s^2=\sum_{x}|f(x)| x_t^2=\frac 1 d \sum_{x}|f(x)| \|x\|_2^2
	\end{eqnarray}
for all $s,t\in\{1,\ldots,d\}$. From this we conclude for two totally rotationally symmetric functions $f,g$ that
	\begin{eqnarray}
	\sum_{s=1}^d |\partial_s \hat f(k)\partial_s g(k) |
	\leq&  \sum_{s=1}^d \sin^2(k_{s}) \sum_{x}|f(x)| x_s^2 \sum_{y}|g(y)| y_s^2\nnb
	=& \hat D^{\sin}(k) \sum_{x} \|x\|^2_2 |f(x)|  \sum_{y} \|y\|^2_2|g(y)|,
	\end{eqnarray}
where we recall \refeq{Dsin-def}.
Using this relation we can bound $\hat H_5(k)$ by
	\begin{align}
	|\hat H_5(k)\hat  G^{n}_z(k)|
	\leq &\ 2\lowK^2\Gamma_2'^{n+1} \hat C(k)^{n+3}\hat D^{\sin}(k)
	(2\afz\betadeltaRfz+\betadeltaRfz^2)\nnb
	&+2\lowK^2\Gamma_2'^n \hat C(k)^{n+2}\hat D^{\sin}(k)(\afz\betadeltaRpz
	+\betaap\betadeltaRfz+\betadeltaRfz\betadeltaRpz).
	\end{align}
and obtain the following bound on $\Hcal^{n,l}_{5,z}(x)$:
	\begin{align}
	|\Hcal^{n,l}_{5,z}(x)|
	\leq&\ 2\lowK^2\Gamma_2'^{n+1}(2\upaf\betadeltaRfz+\betadeltaRfz^2)U_{n+3,l}(x)\nnb
	&+2\lowK^2\Gamma_2'^n (\upaf\betadeltaRpz
	+\betaap\betadeltaRfz+\betadeltaRfz\betadeltaRpz)U_{n+2,l}(x).
	\lbeq{term-toBound-H5-bound}
	\end{align}
\paragraph{Final bound on $f_3$.}
In this section, we have bounded $f_3$ by
	\begin{align}
	f_3(z) \leq& \max_{\{n,l,S\}\in \mathcal{S}}
	\frac{ \sup_{x\in S} \Big\{\refeq{term-toBound-H1}+\refeq{term-toBound-H2-bound}
	+\refeq{term-toBound-H3-bound}+\refeq{term-toBound-H4-bound}
	+\refeq{term-toBound-H5-bound}\Big\}}{c_{n,l,S}}.
	\lbeq{bound-on-f3}
	\end{align}
We recall that by the sum of several equation numbers we mean the sum of the terms given in the right-hand sides of the corresponding equations.

In summary and recalling Definition \ref{finalCondition}, when $P(\gamma,\Gamma,z)$ holds, this bound on $f_3(z)$ is smaller than $\gamma_3$. Thus, the improvement of all bounds is successful and we have thus successfully performed the bootstrap. The computation of a numerical value for the bound in \refeq{bound-on-f3} requires the computation of the SRW-integrals $I_{n,l},K_{n,l},T_{n,l},U_{n,l}$. In Section \ref{secNumericsAdvanced}, we show how to bound SRW-integrals and explain for which $x$ the supremum over $S$ is obtained.

The bootstrap function $f_3$ provides various bounds on weighed diagrams. The real size of these diagram depends heavily on the values of $n,l$ and the set $S$ involved. For example, we can expect that $\Hcal^{2,0}_{z}(x)$ is of order $O(1)$ while  $\Hcal^{2,4}_{z}(x)$ is of order $O(d^{-2})$.
Since the form of the bounds on $\Hcal^{n,l}_{z}(x)$ is the same for all $n,l$, we have introduced the constants $c_{n,l,S}$ to merge them into one bootstrap function. Alternatively, we could consider $f_3$ to consist of multiple bootstrap functions that are individually bounded by $\Gamma_3c_{n,l,S}$ within the bootstrap argument.

\section{Rewrite of the NoBLE equation}
\label{secRewrite}
In the preceding part of this paper, we have performed the analysis using the form \refeq{generalForm-simple} for the two-point function. This form is related to the classical lace expansion. We have decided to use this as it considerably simplifies the presentation of the analysis in the preceding section.

In this section, we first derive this characterization from the NoBLE equation, meaning that we identify $\apz,\afz,\Rfz$ and $\Rpz$. Then, we translate the assumptions made on the rewrite \refeq{generalForm-simple} into assumptions on the NoBLE-coefficients $\Xi_z,\Xi^{\iota}_z,\Psi^{\iota}_z,\Pi^{\iota,\kappa}_z$.

The aim of the rewrite is to extract the dominant SRW-like contributions from $\hat \Phi_z$ and
$\hat F_z$, see \refeq{DefPhi}-\refeq{DefFFunction}. These SRW-like contributions will give rise to
$\apz,\afz,\cpz,\cfz$. The remainder is put into $\Rfz$ and $\Rpz$.

Here we show how we extract SRW contributions from $\hat \Phi_z$ and $\hat F_z$ and use them in our analysis. More terms could be extracted from $\hat \Phi_z$ and $\hat F_z$, thereby reducing the value of $\Rfz$ and $\Rpz$ and thus increasing the performance of the perturbative technique. This might allow to prove the infrared bound in even smaller dimensions above the upper critical dimension. We however found the possible gain not in relation with the necessary efforts.

\subsection{Derivation of the rewrite}
\label{secRewriteDerivation}
In this section, we rewrite the functions $\hat \Phi_z(k)$ and $\hat F_z(k)$, as defined in \refeq{DefPhi}-\refeq{DefFFunction}, and identify $\apz,\afz,\Rfz$ and $\Rpz$.
The NoBLE-coefficients are defined as alternating series of non-negative real-valued functions $\Xi^{\ssc[N]}_z,$
$\Xi^{{\ssc[N]},\iota}_z,$ $\Psi^{{\ssc[N]},\iota}_z,$ $\Pi^{{\ssc[N]},\iota,\kappa}_z$:
	\begin{align}
	\lbeq{XiAsBounds}
	\Xi_z(x)&=\sum_{N=0}^\infty (-1)^N\Xi^{\ssc[N]}_z(x),\qquad
	&\Xi^{\iota}_z(x)=\sum_{N=0}^\infty (-1)^N\Xi^{{\ssc[N]},\iota}_z(x),\\
	\lbeq{XiAsBounds-2}
	\Psi^\kappa_z(x)&=\sum_{N=0}^\infty (-1)^N \Psi^{\ssc[N],\kappa}_z(x),\qquad
	&\Pi^{\iota,\kappa}_z(x)=\sum_{N=0}^\infty (-1)^N\Pi^{{\ssc[N]},\iota,\kappa}_z(x).
	\end{align}

\subsubsection{The model-dependent split of the coefficients}
\label{secRewriteSplitSummary}
When rewriting the two-point function, we extract a major SRW-like contribution. We are guided by the intuition that coefficients are of order $O((2d)^{-1})$ and that the main contributions to the NoBLE coefficients are
	\begin{align}
	\Xi^{\ssc[0]}_z(\ve[1])&\approx \Psi^{\ssc[0],\iota}_z(\ve[1]),
	&\Xi^{\ssc[1]}_z(\ve[1])\approx \Psi^{\ssc[1],\kappa}_z(\ve[1]),\\
	\Xi^{\ssc[0],\iota}_z(\ve[\iota])&\approx \aaz\Pi^{\ssc[0],\iota,\kappa}_z(\ve[\iota]).
	\end{align}
Due to the limitation of our bounds it is not beneficial to extract all these contributions. Thus, we create a model-dependent split of the coefficients to improve the performance of the technique.  We define non-negative functions
	\begin{align*}
	\Xi^{\ssc[0]}_{\alpha,z},\quad \Xi^{\ssc[1]}_{\alpha,z},\quad
	\Psi^{\ssc[0],\iota}_{\alpha,{\sss I},z},\quad \Psi^{\ssc[0],\iota}_{\alpha,{\sss II},z},\quad
	\Psi^{\ssc[1],\iota}_{\alpha,{\sss I},z},\quad \Psi^{\ssc[1],\iota}_{\alpha,{\sss II},z},\quad
	\Xi^{\ssc[0],\iota}_{\alpha,{\sss I},z},\quad \Xi^{\ssc[0],\iota}_{\alpha,{\sss II},z},\quad
	\Pi^{\ssc[0],\iota,\kappa}_{\alpha,z},\\
	\Xi^{\ssc[0]}_{{\sss R},z},\quad \Xi^{\ssc[1]}_{{\sss R},z},\quad
	\Psi^{\ssc[0],\iota}_{{\sss R, I},z},\quad \Psi^{\ssc[0],\iota}_{{\sss R, II},z},\quad
	\Psi^{\ssc[1],\iota}_{{\sss R, I},z},\quad \Psi^{\ssc[1],\iota}_{{\sss R, II},z},\quad
	\Xi^{\ssc[0],\iota}_{{\sss R, I},z},\quad \Xi^{\ssc[0],\iota}_{{\sss R, II},z},\quad
	\Pi^{\ssc[0],\iota,\kappa}_{{\sss R},z}.
	\end{align*}
Here these functions satisfy that, for $N=0,1$ and all $x\in \Zd$,
	\begin{align*}
	\Xi^{\ssc[N]}_{z}(x)&=\Xi^{\ssc[N]}_{\alpha,z}(x)+\Xi^{\ssc[N]}_{{\sss R},z}(x),
	\qquad
	\Pi^{\ssc[0],\iota,\kappa}_{\alpha,z}(x)=\Pi^{\ssc[0],\iota,\kappa}_{\alpha,z}(x)
	+\Pi^{\ssc[0],\iota,\kappa}_{{\sss R},z}(x),\\
	\Psi^{\ssc[N],\iota}_{z}(x) &=\Psi^{\ssc[N],\iota}_{\alpha,{\sss I},z}(x)
	+\Psi^{\ssc[N],\iota}_{{\sss R, I},z}(x)
	=\Psi^{\ssc[N],\iota}_{\alpha,{\sss II},z}(x)+\Psi^{\ssc[N],\iota}_{{\sss R, II},z}(x),\\
	\Xi^{\ssc[0],\iota}_{z}(x)&=\Xi^{\ssc[0],\iota}_{\alpha,{\sss I},z} (x)
	+\Xi^{\ssc[0],\iota}_{{\sss R, I},z}(x)
	=\Xi^{\ssc[0],\iota}_{\alpha,{\sss II},z} (x)+\Xi^{\ssc[0],\iota}_{{\sss R, II},z}(x),
	\end{align*}
and, for $x\in \Zd$ with $\|x\|_2>1$,
	\begin{align*}
	\Xi^{\ssc[N]}_{\alpha,z}(x)&=0,
	\qquad\Psi^{\ssc[N],\iota}_{\alpha,{\sss I},z}(x+\ve[\iota])=0,
	\qquad\Psi^{\ssc[N],\iota}_{\alpha,{\sss II},z}(x)=0,\\
	\Xi^{\ssc[0],\iota}_{\alpha,{\sss I},z}(x+\ve[\iota])&=0,
	\qquad\Xi^{\ssc[0],\iota}_{\alpha,{\sss II},z}(x)=0,
 	\end{align*}
and
	\begin{align*}
	\Pi^{\ssc[0],\iota,\kappa}_{\alpha,z}(x)=0,
 	\end{align*}
for $x\nin\{ \ve[\iota],\ve[\iota]+\ve[\kappa]\}$. Further, these functions have the same symmetries as the original coefficients. The idea behind these two different splits (giving rise to the terms with subscripts $I$ and $II$, respectively) is that we split off specific contributions that can be explicitly incorporated in the constant and $\hat{D}(k)$ terms in our expansion. Contributions with subscript $I$ correspond to $x$ for which $\|x-\ve[\iota]\|\leq 1$, while contributions with subscript $II$ correspond to $x$ for which $\|x\|\leq 1$. In Fourier space, this corresponds to contributions with a factor $\e^{\ii k(x-\ve[\iota])}$ and $\e^{\ii k x},$ respectively. See \refeq{Splitf-tmp2} below for how such contributions will arise.

\subsubsection{The Fourier inverse of $\hat F$ and $\hat \Phi$}
Throughout this section, we omit $z$ from our notation and write, e.g., $\aaz=\aa$ and  $\hat F_z(k)=\hat F(k)$. As a first step, we use the Neumann-series to rewrite $\hat F$ and $\hat \Phi$ into a form without matrices. We use that $(\mD[k] + \aa \mJ)^{-1}=(\mD[-k]-\aa\mJ)/(1-\aa^2)$ to rearrange $\hat F(k)$ as
	\begin{eqnarray}
	\hat F(k)&=&\aa\left(\v1+\vPsi[k]\right)\left[ \mD[k] + \aa\mJ +\mPiwoz[k]\right]^{-1} \v1\nnb
	&=&\frac {\aa} {1-\aa^2}  \left(\v1+\vPsi[k]\right)
	\left[ \mI + \frac 1 {1-\aa^2}(\mD[-k] -\aa\mJ) \mPiwoz[k] \right]^{-1} (\mD[-k] -\aa\mJ) \v1\nnb
	&=&\frac {\aa} {1-\aa^2}  \left(\v1+\vPsi[k]\right)
	\sum_{n=0}^{\infty} (-1)^n \left(\frac 1 {1-\aa^2}(\mD[-k] -\aa\mJ) \mPiwoz[k] \right)^{n} (\mD[-k] -\aa\mJ) \v1\nnb
	&=&\frac {\aa} {1-\aa^2}  \sum_{n=0}^{\infty} \sum_{\iota_0,\dots,\iota_n} \left(1+\hat \Psi^{\iota_0}(k)\right)
	\frac {(-1)^n} {(1-\aa^2)^n}\nnb
	\lbeq{PriorTodeff1function}
	&&\qquad\qquad\times\left(\prod_{s=1}^{n} (\e^{-\ii k_{\iota_{s-1}}}\hat \Pi^{\iota_{s-1},\iota_{s}}(k)
	-\aa\hat \Pi^{-\iota_{s-1},\iota_{s}}(k))\right) (\e^{-\ii k_{\iota_n}} -\aa) .
	\end{eqnarray}
We define $\hat F_{n}$ as the $n$th contribution in the sum in \refeq{PriorTodeff1function} and analyze these terms separately.
The Fourier inverse of $\hat F_{n}=\hat F_{n,z}$ is given by
	\begin{align}
	\lbeq{deff0function}
	F_{0}(x)=&\frac {\aa} {1-\aa^2}\sum_{\iota}\left(\delta_{x,-\ve[\iota]}
	-\aa \delta_{0,x}+\Psi^{\iota}(x+e_{\iota})-\aa \Psi^{\iota}(x)\right),\\
	\nonumber
	F_{n}(x)=&\aa \sum_{\iota_0,\dots,\iota_n}\sum_{x_i:\sum_ix_i=x}
	\frac {(-1)^n (\delta_{x_0,0}+\Psi^{\iota_0}(x_0)) } {(1-\aa^2)^{n+1}}
	\left(\prod_{s=1}^{n-1} (\Pi^{\iota_{s-1},\iota_{s}}(x_s+\ve[\iota_{s-1}]) -\aa\Pi^{-\iota_{s-1},\iota_{s}}(x_s))\right)\\
	&\times \big( \Pi^{\iota_{n-1},\iota_{n}}(x_n+\ve[\iota_{n-1}]+\ve[\iota_{n}])
	-\aa\Pi^{\iota_{n-1},\iota_{n}}(x_n+\ve[\iota_{n-1}]) \nnb
	&\qquad\qquad\qquad\qquad-\aa\Pi^{-\iota_{n-1},\iota_{n}}(x_n+\ve[\iota_n])
	+\aa^2\Pi^{-\iota_{n-1},\iota_{n}}(x_n)\big).
 	\lbeq{deffnfunction}
 	\end{align}
In a similar way, we define $\Phi_n$ such that
	\begin{align}
	\hat \Phi(k)&= \sum_{n=0}^\infty \hat \Phi_{n}(k),\qquad\text{ so that also }\qquad \Phi(x)= \sum_{n=0}^\infty \Phi_{n}(x).
	\end{align}
These function are given by
	\begin{align}
	\hat \Phi_{0}(k)=&1+\hat \Xi(k)- \frac {\aa} {1-\aa^2}\sum_{\iota}
	(1+\hat \Psi^{\iota}(k)) (\hat \Xi^{\iota}(k)\e^{-\ii k_\iota}-\aa\hat \Xi^{-\iota}(k)),\\
	\hat \Phi_{n}(k)=& \aa \sum_{\iota_0,\dots,\iota_n} (1+\hat \Psi^{\iota_0}(k))  \frac {(-1)^{n+1}} {(1-\aa^2)^{n+1}} \\
	\nn
	&\times \prod_{s=1}^{n} \left(\hat \Pi^{\iota_{s-1},\iota_{s}}(k) \e^{-\ii k_{\iota_{s-1}}}
	-\aa\Pi^{-\iota_{s-1},\iota_{s}}(x_s)\right) (\hat \Xi^{\iota_n}(k) \e^{-\ii k _{\iota_n}}-\aa\hat \Xi^{-\iota_n}(k)).
	\end{align}
The Fourier inverses of these functions are
	\begin{align}
	\lbeq{defphi0function}
	\Phi_{0}(x)=&\delta_{0,x}+\Xi(x)
	- \frac {\aa} {1-\aa^2}\sum_{\iota,y} (\delta_{0,y}+\Psi^{\iota}(y)) (\Xi^{\iota}(x-y+\ve[\iota])-\aa\Xi^{-\iota}(x-y)),\\
	\lbeq{defphinfunction}
	\Phi_{n}(x)=& \aa \sum_{\iota_0,\dots,\iota_n}\sum_{x_i:\sum_ix_i=x} (\delta_{0,x_0}+\Psi^{\iota_0}(x_0))
	\frac {(-1)^{n+1}} {(1-\aa^2)^{n+1}} \\
	\nn
	&\qquad\times \prod_{s=1}^{n} (\Pi^{\iota_{s-1},\iota_{s}}(x_s+\ve[\iota_{s-1}]) -\aa\Pi^{-\iota_{s-1},\iota_{s}}(x_s))
	(\Xi^{\iota_n}(x_{n+1}+\ve[\iota_n])-\aa\Xi^{-\iota_n}(x_{n+1})).
	\end{align}

\subsubsection{Definition of the rewrite}
\label{secDefinitionRewrite}
In the rewrite, we extract explicit terms that are independent of $k$ and terms that involve $\hat D(k)$ for $\hat F$ and $\hat \Phi$.  Everything else is put into the remainder terms $\hRfz$ and $\hRpz$.  The major contributions that we can extract are part of $\hat F_0$ and $\hat \Phi_0$. Also $\hat F_1$ gives some contributions. We begin with $\hat F_0$ and rewrite it as
	\begin{align}
	\nonumber
	\hat F_0(k)=&\frac {\aa} {1-\aa^2}  \sum_{\iota} \left(1+\hat \Psi^{\iota}(k)\right) (\e^{-\ii k_{\iota}} -\aa) \\
	\lbeq{Splitf-tmp1}
	=& \frac {\aa} {1-\aa^2} \Big(2d \hat D(k)-2d \aa +
	\sum_{N=0}^\infty \sum_{\iota} (-1)^N \hat \Psi^{\ssc[N],\iota}(k) (\e^{-\ii k_{\iota}} -\aa)\Big).
	\end{align}
Recall that the lace-expansion coefficients are defined via an alternating series of non-negative functions, see \refeq{XiAsBounds}, \refeq{XiAsBounds-2}.
For $N=0,1$, we split the sum into
	\begin{align}
	\sum_{\iota}\hat \Psi^{\ssc[N],\iota}(k) (\e^{-\ii k_{\iota}} -\aa)
	&=\sum_{\iota,x}\e^{\ii k\cdot x} \left( \Psi^{\ssc[N],\iota}(x+\ve[\iota]) -\aa \Psi^{\ssc[N],\iota}(x)\right)\nnb
	&=2d(\Psi^{\ssc[N],\iota}_{\alpha,{\sss I}}(\ve[1])-\aa\Psi^{\ssc[N],\iota}_{\alpha,{\sss II}}(0))\nnb
	&\quad+2d \hat D(k)\sum_{\kappa} \Big(\Psi^{\ssc[N],\sss 1}_{\alpha,{\sss I}}(\ve[1]+\ve[\kappa])
	-\aa\Psi^{\ssc[N],\sss 1}_{\alpha,{\sss II}}(\ve[\kappa])\Big)\nnb
	&\quad+\sum_{\iota}\sum_{x\in\Zd}\e^{\ii k\cdot x} \left(\Psi^{\ssc[N],\iota}_{{\sss R, I}} (x+\ve[\iota])
	-\aa \Psi^{\ssc[N],\iota}_{{\sss R, II}}(x)\right),
	\lbeq{Splitf-tmp2}
	\end{align}
where we see how the splits involving the subscripts $I$ and $II$ are used to extract random-walk contributions.
From
\begin{align}
	\hat F_1(k)&=\frac {-\aa} {(1-\aa^2)^2} \sum_{\iota_0,\iota_1} \left(1+\hat \Psi^{\iota_0}(k)\right)
	\Big(\e^{-\ii k_{\iota_{0}}}\hat \Pi^{\iota_{0},\iota_{1}}(k) -\aa\hat \Pi^{-\iota_{0},\iota_{s}}(k))\Big)
	(\e^{-\ii k_{\iota_1}} -\aa),
	\lbeq{Splitf-tmp2.5}
	\end{align}
we extract the contribution of $1 \times \e^{-\ii k_{\iota_{0}}} \Pi^{\ssc[0],\iota_0,\iota_1} \times \e^{-\ii k_{\iota_{1}}} $ and split it as
	\begin{align}
	\sum_{\iota_0,\iota_1}
	\e^{-\ii (k_{\iota_{0}}+k_{\iota_{1}})}\hat \Pi^{\ssc[0],\iota_{0},\iota_{1}}(k)
	&=
	2d\sum_\kappa
	\Big( \Pi^{\ssc[0],1,\kappa}_{\alpha} (\ve[1]+\ve[\kappa])
	+\hat D(k) \Pi^{\ssc[0],1,\kappa}_{\alpha} (\ve[1])\Big)\nnb
	&\quad 
	+\sum_{\iota_0,\iota_1}
	\sum_{x\in\Zd}\e^{\ii k\cdot x}\Pi^{\ssc[0],\iota_{0},\iota_{1}}_{\sss R}(x+\ve[\iota_0]+\ve[\iota_1]).
	\lbeq{Splitf-tmp3}
	\end{align}
We have now collected all the terms of the split in the lines \refeq{Splitf-tmp1}-\refeq{Splitf-tmp3}. The constant terms contribute to $\cfz$. Terms involving $\hat D(k)$ give rise to $\afz$. All other terms contribute to $\hRfz(k)$. Thus, we conclude that
	\begin{align}
	\lbeq{detailed-Def-cfz}
	\cfz=&-\frac {2d\aaz^2} {1-\aaz^2}+
	\frac {2d \aaz}{1-\aaz^2}\sum_{N\in\{0,1\}} (-1)^N(\Psi^{\ssc[N],\iota}_{\alpha,{\sss I},z}(\ve[1])
	-\aa\Psi^{\ssc[N],\iota}_{\alpha,{\sss II},z}(0))\nnb
	& -\frac {2d\aaz} {(1-\aaz^2)^2} \sum_\kappa \Pi^{\ssc[0],1,\kappa}_{\alpha,z} (\ve[1]+\ve[\kappa]),\\
	\lbeq{detailed-Def-afz}
	\afz=&\frac {2d \aaz} {1-\aaz^2} \Big[1+\sum_{N\in\{0,1\}}(-1)^N \sum_{\iota}
	\left(\Psi^{\ssc[N],1}_{\alpha,{\sss I}}(\ve[1]+\ve[\iota])-\aa\Psi^{\ssc[N],1}_{\alpha,{\sss II}}(\ve[\iota])\right)\Big]\nnb
	&-\frac {2d\aaz} {(1-\aaz^2)^2} \sum_\kappa \Pi^{\ssc[0],1,\kappa}_{\alpha} (\ve[1]),
	\end{align}
and
	\begin{align}
	\lbeq{Def-Rfz}
	\hRfz(k)=\hat F_z(k)-\cfz-\afz\hat D(k),
	\end{align}
which is the sum of the final contributions to the right hand side of  \refeq{Splitf-tmp1}, \refeq{Splitf-tmp2},  \refeq{Splitf-tmp3} and the remainder of \refeq{Splitf-tmp2.5}. We rewrite $\hat \Phi(k)$ in the same way. We begin by noting that
	\begin{align}
	\lbeq{Phi0-split}
	\hat \Phi_{0}(k)=&\ 1+ \sum_{N\in\{0,1\}} (-1)^N \hat \Xi^{\ssc[N]}(k) +\sum_{N=2}^\infty (-1)^N \hat \Xi^{\ssc[N]}(k) \\
	&-\frac {\aa} {1-\aa^2}\sum_{\iota} (1+\hat \Psi^{\iota}(0)) (\hat \Xi^{\iota}(k)\e^{-\ii k_\iota}-\aa\hat \Xi^{-\iota}(k)).\nn
	\end{align}
For $N=0,1$, we split $\hat \Xi^{\ssc[N]}(k)$ as
	\begin{align}
	\hat \Xi^{\ssc[N]}(k)=&\ \Xi^{\ssc[N]}_{\alpha}(0)+2d \hat D(k) \Xi^{\ssc[N]}_{\alpha}(\ve[1])+ \hat \Xi^{\ssc[N]}_{{\sss R}}(k).
	\end{align}
Further, we extract the contribution of the factor $1$ and $\Xi^{\ssc[0], \iota}$ in the second line of \refeq{Phi0-split} as
	\begin{align}
	\sum_{x,\iota} \e^{\ii k\cdot x} (\Xi^{\ssc[0],\iota}(x+\ve[\iota])-\aa\Xi^{\ssc[0],-\iota}(0))
	=&2d (\Xi^{\ssc[0],\iota}_{\alpha,{\sss I}}(\ve[\iota])-\aa\Xi^{\ssc[0],\iota}_{\alpha,{\sss II}}(0))\\
	&+2d\hat D(k)\sum_\kappa \left(\Xi^{\ssc[0],\iota}_{\alpha,{\sss I}}(\ve[\kappa]+\ve[\iota])
	-\aa\Xi^{\ssc[0],\iota}_{\alpha,{\sss II}}(\ve[\kappa])\right)\nnb
	&+\sum_{\iota} \sum_{x\in\Zd} \e^{\ii k\cdot x} (\Xi^{\ssc[0],\iota}_{{\sss R, I}}(x+\ve[\iota])
	-\aa\Xi^{\ssc[0],-\iota}_{{\sss R, II}}(x)).\nn
	\end{align}
We define
	\begin{align}
	\lbeq{detailed-Def-cpz}
	\cpz=&\ 1+\sum_{N\in\{0,1\}} (-1)^N \Xi^{\ssc[N]}_{\alpha,z}(0)
	- \frac {2d \aaz} {1-\aaz^2}\left(\Xi^{\ssc[0],\iota}_{\alpha,{\sss I},z}(\ve[\iota])
	-\aaz\Xi^{\ssc[0],\iota}_{\alpha,{\sss II},z}(0)\right),\\
	\lbeq{detailed-Def-apz}
	\apz=&\ 2d\sum_{N\in\{0,1\}} (-1)^N \Xi^{\ssc[N]}_{\alpha,z}(\ve[1])
	- \frac {2d \aaz} {1-\aaz^2}\sum_\kappa \left(\Xi^{\ssc[0],\iota}_{\alpha,{\sss I}}(\ve[\kappa]+\ve[\iota])
	-\aaz\Xi^{\ssc[0],\iota}_{\alpha,{\sss II}}(\ve[\kappa])\right),\\
	\hRpz(k)=&\hat \Phi_z(k)-\cpz-\apz\hat D(k).
	\end{align}
This completes the derivation of the rewrite \refeq{DefPhi-simple} and \refeq{DefFFunction-simple} and identifies $\afz,\apz,\cfz,$ $\cpz,$ $\hRfz$ and $\hRpz$.

At this point, it is worth mentioning that this is not the only possible split. Indeed, we could try to put more terms into $\afz,\apz$, thus reducing $\hRfz,\hRpz$, and thereby improving the efficiency of the analysis. However, numerically we found that the possible gain would not be in relation to the necessary efforts, so we refrain from this.

\subsection{Assumption on the NoBLE coefficients}
\label{assNobleAssumptions}
In this section, we reformulate Assumption \ref{assSym}- \ref{assGzBehavoirCritical} on $\afz,\apz,\hRfz$ and $\hRpz$ in terms of the NoBLE coefficients.  We assume that the NoBLE coefficients have the following properties:
\begin{ass}[Symmetry of the models]
\label{assSymtwo}
Let $\iota,\kappa\in\{\pm 1,\pm 2,\dots,\pm d\}$. The following symmetries hold for all $x\in\Zd$, $z\leq z_c$, $N\in\Nbold$ and $\iota,\kappa$:
	\begin{eqnarray*}
	\Xi^\ssc[N]_z(x)&=& \Xi^\ssc[N]_z(-x), \qquad  \qquad \qquad
	\Xi^{\ssc[N],\iota}_z(x)= \Xi^{\ssc[N],-\iota}_z (-x),\\
	\Psi^{\ssc[N],\iota}_z(x)&=& \Psi^{\ssc[N],-\iota}_z(-x), \qquad \qquad\
	\Pi^{\ssc[N],\iota,\kappa}_z(x)= \Pi^{\ssc[N],-\iota,-\kappa}_z (-x).
	\end{eqnarray*}
For all $N\in\Nbold$, the coefficients
	\begin{align}
	\Xi^\ssc[N](x),\qquad \sum_{\iota}\Psi^{\ssc[N],\iota}_z(x),
	\qquad \sum_{\iota}\Xi^{\ssc[N],\iota}_z(x) \quad \text{and}\quad \sum_{\iota,\kappa}\Pi^{\ssc[N],\iota,\kappa}_z(x),
	\end{align}
as well as the remainder terms of the split
	\begin{align}
	\Xi^{\ssc[N]}_{{\sss R},z}(x),\quad \sum_{\iota}\Psi^{\ssc[N],\iota}_{{\sss R, I},z}(x),
	\quad \sum_{\iota}\Psi^{\ssc[N],\iota}_{{\sss R, II},z}(x),
	\quad \sum_{\iota}\Xi^{\ssc[0],\iota}_{{\sss R, I},z}(x)
	\quad \sum_{\iota}\Xi^{\ssc[0],\iota}_{{\sss R, II},z}(x),
	\quad \sum_{\iota,\kappa}\Pi^{\ssc[0],\iota,\kappa}_{{\sss R},z}(x),
	\end{align}
are totally rotationally symmetric functions of $x\in\Zd$.
Further, the dimensions are exchangeable, i.e., for all $\iota,\kappa$,
	\begin{align}
 	\lbeq{assCoefficentsDimInterchange}
	\hat \Psi^{\ssc[N],\iota}_z(0)=\ \hat \Psi^{\ssc[N],\kappa}_z(0),\qquad
	\hat \Xi^{\ssc[N],\iota}_z(0)=\ \hat \Xi^{\ssc[N],\kappa}_z(0),\qquad
	\sum_{\kappa'}\hat \Pi^{\ssc[N],\iota,\kappa'}_z(0)=\sum_{\iota'}\hat \Pi^{\ssc[N],\iota',\kappa}_z(0).
	\end{align}
\end{ass}
\medskip

The next assumption states a bound on $\Psi^{\ssc[N],\kappa}_z$ and $\Pi^{\ssc[N],\iota,\kappa}_z$ in terms of $\Xi^{\ssc[N],\iota}_z$:

\begin{ass}[Relation between coefficients]
\label{assCoefficentsRelation}
For all $x\in\Zd$, $p\leq p_c$, $N\in\Nbold$ and $\iota,\kappa\in\{\pm 1,\pm 2,\dots,\pm d\}$, the following bounds hold:
	\begin{align}
	\lbeq{XidominatespsiImproved}
	\Psi^{\ssc[N],\kappa}_z(x)\leq&\frac {\aabz}{ \aaz} \Xi^\ssc[N]_z(x),
	\qquad\qquad\Pi^{\ssc[N],\iota,\kappa}_z(x)\leq  \aabz \Xi^{\ssc[N],\iota}_z(x).
	\end{align}
\end{ass}
\medskip

As explained in Section \ref{subsecStructureAna}, to successfully apply the bootstrap argument, we assume that the coefficients obey certain bounds when the bootstrap assumption $f_i(z)\leq \Gamma_i$ holds for all $i\in\{1,2,3\}$ for a given $z\in[0,z_c)$. These bounds do not depend on the value of $z$. However, their form is delicate and depends sensitively on the precise model under consideration. We assume that the same bounds hold for $z_I$ regardless of the values $f_1(z_I),f_2(z_I),f_3(z_I)$. Assumption \ref{assDiagBoundsCoeff} is the most technical assumption of this paper, and is phrased so as to allow maximal flexibility in the application of the NoBLE: 

\begin{ass}[Diagrammatic bounds]
\label{assDiagBoundsCoeff}
Let $\Gamma_1,\Gamma_2,\Gamma_3\geq 0$. Assume that $z\in (z_I,z_c)$ is such that $f_i(z)\leq \Gamma_i$ for $i\in\{1,2,3\}$ holds.
Then $\hat G_z(k)\geq 0$ for all $k\in(-\pi,\pi)^d$. There exists $\betaaa\geq 1,\betaaalow>0$ such that
	\begin{align}
	\lbeq{analys-assumed-rho-Bound-two}
	\frac {\aabz}{ \aaz}\leq \betaaa,\qquad \aaz \geq \betaaalow.
	\end{align}
Further, there exist $\beta_{\sss \Xi}^\ssc[N],\beta_{\sss \Xi^\iota}^\ssc[N],\beta_{{\sss \Delta \Xi}}^\ssc[N],\beta_{{\sss \Delta \Xi^{\iota}},0}^\ssc[N],\beta_{{\sss \Delta \Xi^{\iota}},\iota}^\ssc[N] \geq 0$, such that
	\begin{align}
 	\lbeq{analys-assumed-BoundXi}
	\hat \Xi^\ssc[N]_z(0)&\leq \beta_{\sss \Xi}^\ssc[N],\qquad\qquad \hat \Xi^{\ssc[N],\iota}_z(0)
	\leq \beta_{{\sss \Xi}^{\iota}}^\ssc[N],\\
	\lbeq{analys-assumed-displacement-one-InXDiff}
	\sum_{x}\|x\|_2^2\Xi^{\ssc[N]}_z(x)&\leq \beta_{{\sss \Delta \Xi}}^\ssc[N],
	\qquad\quad\sum_{x} \|x\|_2^2 \Xi^{\ssc[N],\iota}_z(x)\leq \beta_{{\sss \Delta \Xi^{\iota},0}}^\ssc[N],\\
	\lbeq{analys-assumed-displacement-three-InXDiff}
	\sum_{x} \|x-\ve[\iota]\|_2^2 \Xi^{\ssc[N],\iota}_z(x)&\leq \beta_{{\sss \Delta \Xi^{\iota},\iota}}^\ssc[N],
	\end{align}
for all $N\geq 0$ and $k\in(-\pi,\pi)^d$. Moreover, we assume that $\sum_{N=0}^\infty \beta_{\bullet}^\ssc[N] <\infty$ for $\bullet \in\{ \Xi, \Xi^{\iota}, \Delta\Xi, \{\Delta \Xi^\iota,0\},\{\Delta \Xi^\iota,\iota\}\}$ and that
	\begin{align}
	\lbeq{analys-assumed-invertablecondition}
	\frac {(2d-1)\aabz}{1-\aaz}\sum_{N=0}^\infty\beta_{{\sss \Xi}^{\iota}}^\ssc[N]<1.
	\end{align}
Further, there exist $\underline{\beta}_{\sss \Psi}^\ssc[0]$, $\underline{\beta}_{\sss \sum \Pi}^\ssc[1]$  such that
	\begin{align}
	\hat \Psi^{\ssc[0],\iota}_z(0)\geq&\  \underline{\beta}_{\sss \Psi}^\ssc[0],
	&\sum_\kappa \hat \Pi^{\ssc[1],\iota,\kappa}_z(0)\geq\  \underline{\beta}_{\sss
	\sum\Pi}^\ssc[1].
	\end{align}
Additionally, there exist $\beta_{{\sss\Xi_\alpha(0)}}^\ssc[1-0],\beta_{{\sss\Xi_\alpha(0)}}^\ssc[0-1],\beta_{{\sss\Xi_\alpha(\ve[1])}}^\ssc[1-0],\beta_{{\sss\Xi_\alpha(\ve[1])}}^\ssc[0-1]$ with
	\begin{align}
	\lbeq{boundOnXizero}
	-\beta_{{\sss\Xi_\alpha(0)}}^\ssc[1-0]\leq
	& \Xi^\ssc[0]_{\alpha,z}(0)-\Xi^\ssc[1]_{\alpha,z}(0)\leq \beta_{{\sss\Xi_\alpha(0)}}^\ssc[0-1],\\
	-\beta_{{\sss\Xi_\alpha(\ve[1])}}^\ssc[1-0]
	\leq& \Xi^\ssc[0]_{\alpha,z}(\ve[1])-\Xi^\ssc[1]_{\alpha,z}(\ve[1])
	\leq \beta_{{\sss\Xi_\alpha(\ve[1])}}^\ssc[0-1] ,
	\end{align}
and $\beta_{{\sss\Xi^{\iota}_\alpha,I}}^\ssc[0],\beta_{{\sss \sum}{\sss\Xi^{\iota}_\alpha,I}}^\ssc[0],\beta_{{\sss\Xi^{\iota}_\alpha,II}}^\ssc[0], \beta_{ {\sss\sum \Xi^{\iota}_\alpha,II}}^\ssc[0],
\geq 0$ such that
	\begin{align}
	\Xi^{\ssc[0],\iota}_{\alpha,{\sss I},z}(\ve[\iota])
	\leq& \beta_{{\sss\Xi^{\iota}_\alpha,I}}^\ssc[0],\qquad
	\sum_\kappa\Xi^{\ssc[0],\iota}_{\alpha,{\sss I},z}(\ve[\iota]+\ve[\kappa])
	\leq \beta_{{\sss\sum\Xi^{\iota}_\alpha,I}}^\ssc[0],\\
	\Xi^{\ssc[0],\iota}_{\alpha,{\sss II},z}(0)\leq& \beta_{{\sss\Xi^{\iota}_\alpha,II}}^\ssc[0],
	\qquad
	\sum_\kappa\Xi^{\ssc[0],\iota}_{\alpha,{\sss II},z}(\ve[\kappa])\leq
	\beta_{{\sss\sum \Xi^{\iota}_\alpha,II}}^\ssc[0].
	\end{align}
Also, there exist
$\beta_{{\sss\sum \Psi^{\iota}_\alpha,I}}^\ssc[0-1],\beta_{{\sss\sum \Psi^{\iota}_\alpha,II}}^\ssc[0-1],$
$\beta_{{\sss\sum \Psi^{\iota}_\alpha,I}}^\ssc[1-0],$
$\beta_{{\sss\sum \Psi^{\iota}_\alpha,II}}^\ssc[1-0],$
$\underline{\beta}_{ \sss \sum \Pi_\alpha}^\ssc[0],$
$\beta_{\sss\sum \Pi_\alpha}^\ssc[0]$, such that
	\begin{align}
	-\beta_{{\sss\sum \Psi^{\iota}_\alpha,I}}^\ssc[1-0]\leq&
	\sum_\kappa \left(\Psi^{\ssc[0],\iota}_{\alpha,{\sss I},z}(\ve[\iota]+\ve[\kappa])
	-\Psi^{\ssc[1],\iota}_{\alpha,{\sss I},z}(\ve[\iota]+\ve[\kappa])  \right)
	\leq \beta_{{\sss\sum \Psi^{\iota}_\alpha,I}}^\ssc[0-1],\\
	-\beta_{{\sss\sum \Psi^{\iota}_\alpha,II}}^\ssc[1-0]
	\leq&
	\sum_\kappa \left(\Psi^{\ssc[0],\iota}_{\alpha,{\sss II},z}(\ve[\kappa])  -\Psi^{\ssc[1],\iota}_{\alpha,{\sss II},z}(\ve[\kappa])\right)
	\leq \beta_{{\sss\sum \Psi^{\iota}_\alpha,I}}^\ssc[0-1],\\
	\underline{\beta}_{ \sss \sum \Pi_{\alpha}}^\ssc[0]
	\leq &\sum_{\kappa}\Pi^{\ssc[0],\iota,\kappa}_{\alpha,z}(\ve[\iota])\leq \bar{\beta}_{\sss \sum \Pi_\alpha}^\ssc[0].
	\end{align}
For $N=0,1$, there exist
$\beta_{\sss\Xi,R}^\ssc[N]$,
$\beta_{\Delta \sss\Xi,R}^\ssc[N]$,
$\beta_{\sss\Psi,R,I}^\ssc[N]$,
$\beta_{\Delta \sss\Psi,R,I}^\ssc[N]$,
$\beta_{\sss\Psi,R,II}^\ssc[N]$,
$\beta_{\Delta \sss\Psi,R,II}^\ssc[N]\geq 0$,
such that
	\begin{align}
	\sum_{x}\Xi_{{\sss R},z}^\ssc[N](x) \leq &\beta_{\sss\Xi,R}^\ssc[N],  \qquad
	\sum_{x}\|x\|_2^2\Xi_{{\sss R},z}^\ssc[N](x) \leq \beta_{\Delta\sss\Xi,R}^\ssc[N],  \\
	\sum_{x}\Psi_{{\sss R,I},z}^{\ssc[N],\iota} (x) \leq& \beta_{\sss\Psi,R,I}^\ssc[N],  \qquad
	\sum_{x}\|x-\ve[\iota]\|_2^2\Psi_{{\sss R,I},z}^{\ssc[N],\iota} (x)\leq \beta_{\Delta\sss\Psi,R,I}^\ssc[N],
	\\
	\sum_{x}\Psi_{{\sss R,II},z}^{\ssc[N],\iota} (x)\leq& \beta_{\sss\Psi,R,II}^\ssc[N],  \qquad
	\sum_{x}\|x\|_2^2\Psi_{{\sss R,II},z}^{\ssc[N],\iota} (x)\leq \beta_{\Delta\sss\Psi,R,II}^\ssc[N].
	\end{align}
Further, there exist
$\beta_{\sss\Xi^\iota,R,I}^\ssc[0]$,
$\beta_{\Delta \sss\Xi^\iota,R,I}^\ssc[0]$,
$\beta_{\sss\Xi^\iota,R,II}^\ssc[0]$,
$\beta_{\Delta \sss\Xi^\iota,R,II}^\ssc[0]$,
$\beta_{\sss\Pi,R}^\ssc[0]$,
$\beta_{\Delta \sss\Pi,R}^\ssc[0]\geq 0$,
such that
	\begin{align}
	\sum_{x}\Xi_{{\sss R,I},z}^{\ssc[0],1} (x)\leq& \beta_{\sss\Xi^\iota,R,I}^\ssc[0],  \qquad
	\sum_{x}\|x-\ve[\iota]\|_2^2\Xi_{{\sss R,I},z}^{\ssc[0],\iota} (x+\ve[\iota])\leq \beta_{\Delta\sss\Xi^\iota,R,I}^\ssc[0],\\
	\sum_{x}\Xi_{{\sss R,II},z}^{\ssc[0],1} (x)\leq& \beta_{\sss\Xi^\iota,R,II}^\ssc[0],  \qquad
	\sum_{x}\|x\|_2^2\Xi_{{\sss R,II},z}^{\ssc[0],\iota} (x)\leq \beta_{\Delta\sss\Xi^\iota,R,II}^\ssc[0],  \\
	\sum_{x,\iota}\Pi^{\ssc[0],\iota,\kappa}_{{\sss R},z} (x)\leq& \beta_{\sss\Pi,R}^\ssc[0],  \qquad
	\sum_{x,\iota,\kappa}\|x\|_2^2\Pi^{\ssc[0],\iota,\kappa}_{{\sss R},z} (x+\ve[\iota]+\ve[\kappa])\leq
	\beta_{\Delta\sss\Pi,R}^\ssc[0].
	\lbeq{analys-assumed-displacement-xiikD}
	\end{align}
For all $\bullet \in\{ \Xi, \Xi^{\iota}, \Delta\Xi, \{\Delta \Xi^\iota,0\},\{\Delta \Xi^\iota,\iota\}\}$ and $N\in \Nbold$, $\beta_{\bullet}^{\ssc[N]}$ depends only on $\Gamma_1,\Gamma_2,\Gamma_3,d$ and on the model.  If Assumption \ref{assXBoundInitialCondition} holds, then the bounds stated above also holds for $z=z_I$ with the constants $\beta_{\bullet}$ only depending on the dimension $d$ and the model.
\end{ass}
\medskip

Only the bounds \refeq{analys-assumed-rho-Bound-two}-\refeq{analys-assumed-displacement-three-InXDiff}
are essential to perform the analysis for the NoBLE. The bounds stated in \refeq{boundOnXizero}-\refeq{analys-assumed-displacement-xiikD} are used to obtain good bounds on $\cpz,\apz,\afz,\hRfz$ and $\hRpz$ that allow us to increase the performance of the analysis and to show mean-field result in lower dimensions than otherwise possible.

We denote by $\beta_{\bullet}^{\sss \text{abs}}$, $\beta_{\bullet}^{\sss \text{odd}}$ and $\beta_{\bullet}^{\sss \text{even}}$ the sum over all (resp. odd/even) $N$ of  $\beta_{\bullet}^\ssc[N]$, i.e.,
	\begin{align}
	\beta_{\bullet}^{\sss \text{abs}}=\sum_{N=0}^{\infty} \beta_{\bullet }^\ssc[N],
	\qquad\qquad\beta_{\bullet}^{\sss \text{odd}}=\sum_{N=0}^{\infty} \beta_{\bullet }^\ssc[2N+1],
	\qquad\qquad\beta_{\bullet}^{\sss \text{even}}=\sum_{N=0}^{\infty} \beta_{\bullet }^\ssc[2N],
	\end{align}
for $\bullet \in\{ \Xi, \Xi^{\iota}, \Delta\Xi, \{\Delta \Xi^\iota,0\},\{\Delta \Xi^\iota,\iota\}\}$. By \refeq{XiAsBounds}, the values $\beta_{{\sss \Xi}}^\ssss[even],\beta_{{\sss \Xi}^{\iota}}^\ssss[even]$ and $(-\beta_{{\sss \Xi}}^\ssss[odd]),(-\beta_{{\sss \Xi}^{\iota}}^\ssss[odd])$ are explicit upper and lower bounds on $\hat \Xi_z(0)$ and $\hat \Xi^\iota_z(0)$, respectively.
By Assumption \ref{assCoefficentsRelation} they also imply bounds on $\hat\Psi^{\iota}_z(0)$ and $\hat\Pi^{\iota,\kappa}_z(0)$.

We next discuss the left-continuity of the coefficients at $z=z_c$:
\begin{ass} [Growth at the critical point]
\label{assGzBehavoirCriticaltwo}
The functions $z\mapsto \hat \Xi_z(k),z\mapsto \hat  \Xi^{\iota}_z(k),z\mapsto \hat \Psi^{\kappa}_z(k),z\mapsto \hat  \Pi^{\iota,\kappa}_z(k)$ are continuous for $z\in(0,z_c)$. Further, let $\Gamma_1,\Gamma_2,\Gamma_3\geq 0$ be such that $f_i(z)\leq \Gamma_i$ and that Assumption \ref{assDiagBoundsCoeff} holds. Then, the functions stated above are left-continuous in $z_c$ with a finite limit $z\nearrow z_c$ for all $x\in\Zd$. Further, for technical reasons, we assume that $z_c<1/2$.
\end{ass}
\medskip

In the remainder of this section we show that Assumptions \ref{assSymtwo}-\ref{assGzBehavoirCriticaltwo}
imply Assumptions \ref{assSym}-\ref{assGzBehavoirCritical}, as formulated in the following proposition:

\begin{prop}[Translation of the assumptions]
\label{assTranslation}
The assumptions stated in Section \ref{subsecAss} are implied by the assumptions stated in Section \ref{assNobleAssumptions}. More precisely,
\begin{enumerate} [(i)]
\item Assumption \ref{assSymtwo} implies Assumption \ref{assSym},
\item Assumptions \ref{assSymtwo} - \ref{assDiagBoundsCoeff} imply Assumption \ref{assDiagBounds},
\item Assumptions \ref{assCoefficentsRelation} - \ref{assGzBehavoirCriticaltwo} imply Assumption \ref{assGzBehavoirCritical}.
\end{enumerate}
\end{prop}
\medskip

The proof of parts (i),(iii) are relatively straightforward and are performed in Section \ref{secRewriteProperties}.
Part (ii) is proven using a tedious, but also straightforward, application of the bounds stated in Assumption \ref{assDiagBoundsCoeff}.
We add the details of this in Appendix \ref{secRewriteBounds}.
\subsection{Properties of the rewrite}
\label{secRewriteProperties}
In this section, we prove Proposition \ref{assTranslation}(i) and (iii).

\begin{proof}[Proof of Proposition \ref{assTranslation}(i)]
In Assumption \ref{assSymtwo}, we assume that the dimensions are interchangeable (recall \refeq{assCoefficentsDimInterchange}), so that \refeq{assSym-interchange} clearly holds. Further, we also assume in Assumption \ref{assSymtwo} that the two-point function $x\mapsto G_z$ is totally rotationally symmetric.

To see that $\Rfz$ and $\Rpz$ are totally rotationally symmetric, we note that convolutions maintain symmetry and that the NoBLE-coefficients, in the way they are combined in the definition of  $\Rfz$ and $\Rpz$, are thus also totally rotationally symmetric. This completes the proof of Proposition \ref{assTranslation}(i).
\end{proof}

\begin{proof}[Proof of Proposition \ref{assTranslation}(iii)] We prove the statement in three steps:
   	\begin{enumerate}[(a)]
  	\item We prove that $[\mD[k]+\aaz \mJ+\mPi[k]]^{-1}$ is well defined for all $k$ and $z\leq z_c$;
  	\item We conclude that $z\mapsto \hat \Phi_z(k)$,  $z\mapsto \hat F_z(k)$
	are continuous at $z$ and well defined at $z_c$,
  	which implies that the bounds stated in Assumption \ref{assDiagBounds} also hold for $z=z_c$;
  	\item We show that $\hat G_z(k)$ can be continuously extended to $z=z_c$ for $k\neq 0$.
  	\end{enumerate}
This implies the desired statement that $z\mapsto \hat G_z(k)$ is left-continuous at $z=z_c$ for any $k\neq 0$ and  that the bounds stated in Assumption \ref{assDiagBounds} also hold for $z=z_c$. 

(a) We begin by showing that
	\begin{align}
  	\lbeq{technicalline-invertable-tmpa}
	\Big\|\big[\mD[k]+\aaz \mJ\big]^{-1}\mPi[k]\Big\|_\infty=\sup_{\vec v\colon \|v\|_\infty=1}
	\max_{\iota} \Big|\Big( \big[\mD[k]+\aaz \mJ\big]^{-1}\mPi[k] \vec v\Big)_\iota\Big|<1.
	\end{align}
We start by noting that $\big[\mD[k]+\aaz \mJ\big]^{-1}=\frac 1 {1-\aaz^2} \big(\mD[-k]-\aaz \mJ\big),$ so that
	\begin{align}
 	\Big( \big[\mD[k]+\aaz \mJ\big]^{-1}\mPi[k]\vec v\Big)_\iota&
 	=\frac {1}{1-\aaz^2}\sum_{\kappa} (\hat \Pi^{\iota,\kappa}_z(k)\e^{\ii k_\iota}
	-\aaz\hat \Pi^{-\iota,\kappa}_z(k)) v_\kappa.
 	\end{align}
Thus, for $\vec v$ with $\|v\|_\infty=1$,
	\begin{align}
	\Big\|\big[\mD[k]+\aaz \mJ\big]^{-1}\mPi[k]\vec v\Big\|_\infty
	&\leq\frac {1+\aaz}{1-\aaz^2} \|v\|_\infty \sum_{N,\kappa,x}\Pi^{\ssc[N]\iota,\kappa}(x)\nn \\
	&\stackrel{\refeq{XidominatespsiImproved},\refeq{analys-assumed-BoundXi} }
	\leq
	\frac {2d \aaz}{1-\aaz} \beta_{\sss \Xi^\iota}^{\sss \text{abs}}
	\stackrel{\refeq{analys-assumed-invertablecondition} }<1,
	\end{align}
which proves \refeq{technicalline-invertable-tmpa}. From \refeq{technicalline-invertable-tmpa}, it follows that the matrix $\mI+\left[\mD[k]+\aaz \mJ\right]^{-1}\mPi[k]$ is invertible. Then, we use standard linear algebra to compute
	\begin{align}
	&\left[ \mI+\big[\mD[k]+\aaz \mJ\big]^{-1}\mPi[k]\right]^{-1}
	\big[\mD[k]+\aaz \mJ\big]^{-1}	\big[\mD[k]+\aaz \mJ+\mPi[k]\big]\\
	&\qquad=\left[ \mI+\big[\mD[k]+\aaz \mJ\big]^{-1}\mPi[k]\right]^{-1}
	\left[\mI+\big[\mD[k]+\aaz \mJ\big]^{-1}\mPi[k]\right]=\mI,\nn
	\end{align}
which implies that the matrix $\mD[k]+\aaz \mJ+\mPi[k]$ is invertible.

(b) By Assumption \ref{assGzBehavoirCriticaltwo}, we know that the NoBLE coefficients are continuous in $z$.
This also implies that $\mD[k]+\aaz \mJ+\mPi[k]$ is continuous in $z$ and, as it is well defined by (a), its inverse is also continuous.  Further, we note that the bounds on the coefficients $\beta_{\bullet}$ are independent of the value of $z\in(z_I,z_c)$ and the coefficients are left-continuous in $z=z_c$.
Reviewing the definition of $\hat \Phi_z(k)$ and $\hat F_z(k)$ in \refeq{DefPhi}-\refeq{DefFFunction}, we conclude that these functions are continuous in $z\in[z_I,z_c)$ and left-continuous at $z=z_c$.

(c) The dominated convergence theorem implies that
	\begin{align}
	G_{z_c}(x)=& \lim_{z\nearrow z_c}G_z (x)
	= \lim_{z\nearrow z_c} \int_{(-\pi,\pi)^d}\frac {\hat \Phi_z(k)}{1-\hat F_z(k)}
	\e^{\ii k\cdot x}\frac {d^dk}{(2\pi)^d}\nnb
	=&\int_{(-\pi,\pi)^d}\frac {\hat \Phi_{z_c}(k)}{1-\hat F_{z_c}(k)}\e^{\ii k\cdot x}\frac {d^dk}{(2\pi)^d},
	\lbeq{Gzc-continuation}
	\end{align}
where we use the left-continuity of $z\mapsto \hat \Phi_z(k)$ and $z\mapsto \hat F_z(k)$ at $z=z_c$ proved above, and we further
note that $\lowaf-\betadeltaRfzlow>0$ and $\lowcp-\betaap-\beta_{{\sss R,\Phi}}>0$ together with \refeq{generalForm-simple} imply the infrared bound holds uniformly in $[z_I,z_c)$ for
$\hat \Phi_z(k)/[1-\hat F_z(k)]$, so that the integral in \refeq{Gzc-continuation} is well defined.
Thus, we still have the Fourier representation \refeq{generalForm}, with the understanding that $\hat G_{z_c}(k)$ for $k\neq 0$ is defined by
$\hat G_{z_c}(k)=\hat \Phi_{z_c}(k)/(1-\hat F_{z_c}(k))$.
Since $\hat F_{z_c}(0)=1$, this characterization can not be used for $k=0$.
This completes the proof of Proposition \ref{assTranslation}(iii).
\end{proof}

\section{Numerical bounds}
\label{secNumerics}
In this section we discuss the ideas underlying the numerical computation of our bounds on the NoBLE coefficients. These ideas are model independent, while the implementation itself is not.  We first explain how we compute the numerical bounds on the SRW-integrals that we have used for the improvement of bounds in Section \ref{secVeriBootstrap} and to obtain numerical bounds on the coefficients. Then, we explain how the bootstrap functions are used to bound simple diagrams.  At the end of this section, we explain how we compute the $\beta^{\text{abs}}_{\bullet}$ as sums over $\beta^{\ssc[N]}_{\bullet}$,  which is not straightforward as the bounds on the NoBLE coefficients are stated in the form of matrix-products.

\subsection{Simple random walk integrals}
\label{secNumericsSRW}
We bound the SRW-integrals $I_{n,l}, K_{n,l}, T_{n,l}, U_{n,l}$ defined in \refeq{Analysis-Def-Integral-I}-\refeq{Analysis-Def-Integral-U}. We first compute $I_{n,m}(x)$ and then show that the other integrals can be bounded in terms of it.
We compute $I_{n,m}(x)$ using
	\begin{eqnarray}
	\lbeq{Inmrec}
	I_{n,m}(x)=I_{n,m-1}(x) - I_{n-1,m-1}(x),
	\end{eqnarray}
which is obtained by writing $\hat D(k)=1-[1-\hat D(k)]$ in \refeq{defInlPerview}. Using \refeq{Inmrec}, the problem of computing $I_{n,m}$ for general $n,m\in\Nbold$ simplifies to the computation of $I_{n,0}$ and $I_{0,m}$ for all $n,m\in\Nbold$.

\subsubsection{Computation of the Green's function}
We compute $I_{n,0}$ in the same way as Hara and Slade in \cite[Appendix B]{HarSla92b}, as we explain now. Let $b(n,s)$ be the modified Bessel function of the first kind and $F(t,d,n)$ the modified Bessel function, i.e.,
	\eqan{
	b(n,s)&=\sum_{k=0}^\infty (-1)^k \left(\frac {s} 2\right)^{2k+n} \frac {1}{k! \Gamma(n+k+1)!},
	\qquad
	F(t,d,n) = \e^{-t/d} b(n,t/d),
	}
see e.g.\ \cite[(8.401) and (8.406)]{GradKyz94} or \cite[Section 9.6] {AbrSte92}. Using
	\begin{eqnarray}
	\lbeq{NumericOneDrelation}
	\frac {1} {[1-\hat D(k)]^n}=\frac 1 {(n-1)!} \int_{0}^\infty t^{n-1}
	\e^{-t[1-\hat D(k)]} dt,
	\end{eqnarray}
we compute
	\begin{eqnarray}
	I_{n,0}(x)&=& \frac 1 {(n-1)!} \int_{0}^\infty t^{n-1} \prod_{\mu=1}^d F(t,d,|x_\mu|) dt,
	\end{eqnarray}
see \cite[Appendix B]{HarSla92b}. Most mathematical software packages, such as Mathematica, Matlab, and R, come with a method to compute the modified Bessel Integral. We have used Mathematica which allows to control the precision of the computation.  With the built-in function we compute $I_{4,0}(x)$ in $d\geq 15$ and  $I_{5,0}(x)$ in $d\geq 18$ up to a precision of $10^{-20}$. To be able to compute these basic SRW-integrals in lower dimensions, we implement the algorithm given in \cite[Appendix B]{HarSla92b},  where also a rigorous bound on the error is proven. This algorithm is based on a Taylor approximation of the Bessel function.

\subsubsection{Computation of the random walk transition probability}
The computation of $I_{0,m}(x)$ is a purely combinatorial problem as $(2d)^m I_{0,m}(x)=p_m(x)$, where $p_m(x)$ is the number of $m$-step SRWs with $\omega_0=0,\omega_m=x$. The value of $p_n(x)$ can be obtained by simple combinatorial means. As an example, we explain the computation of $p_6(0)$.

When $\omega_0=\omega_6=0,$ the walk uses at most three different dimensions as it needs to undo all its
steps. In the following, we distinguish between the number of dimensions used by the
walker:
\begin{itemize}
\item[$\rhd$] When the walk only uses one dimension, it steps three times to the positive direction (right) and three times to the negative direction (left). As any combination of left and right steps is allowed there are $6!/(3!3!)$ different possibilities for that. As there are $d$ choices for the dimension used, there are $d\frac {6!} {3!3!}$ such walks.

\item[$\rhd$] When the walk uses two dimensions, it makes $4$ steps in one dimension and $2$ in the other. As any combination of moves is allowed, there are $6!/(2!2!1!1!)$ different possibilities for that.
Further, there are $d$ choices for the dimension in which to take $4$ steps, and likewise $d-1$ choices for the dimension where $2$ steps are made. Thus, there are $d(d-1) \frac {6!} {2!2!}$ SRW $6$-step loops using steps in exactly two dimensions.

\item[$\rhd$] When the walk uses three dimensions, then there are $2$ steps in each dimension. There are $6!$ different orders for these $6$ steps (including the back and forth steps in each of the three dimensions). Further, we have to choose $3$ out of the $d$ dimensions (without repetition). This gives a factor
$\frac{d(d-1)(d-2)}{3!}6!$.
\end{itemize}
This means that
	\begin{eqnarray}
	p_6(0)&=& d { 6 \choose 3,3} + d(d-1) { 6 \choose 2,2,1,1}
	+ \frac{d(d-1)(d-2)}{3!} { 6 \choose 1,1,1,1,1,1},
	\end{eqnarray}
where  the multinomial coefficient is defined as
	\begin{eqnarray}
	{ m \choose k_1,k_2,\dots,k_r}=\frac {m!}{k_1!k_2!\dots,k_r!}.
	\end{eqnarray}
For our analysis, we use the values of $p_n(x)$ for $n\in\{0,\dots, 20\}$ for about $24$ different values of
$x$. We have implemented a program for this, and the algorithm can be found in the accompanying Mathematica notebooks (see also Section \ref{sec-Comp}, where the Mathematica notebooks are described in more detail).

\subsection{Bounds on related SRW-integrals}
\label{secNumericsAdvanced}
In this section,  we show how to bound the integrals defined in \refeq{Analysis-Def-Integral-K}-\refeq{Analysis-Def-Integral-U}. This section is an adaption of \cite[Appendix B.1]{HarSla92b} by Hara and Slade, who computed numerical bounds on these integrals to prove mean-field behaviour for nearest-neighbour SAW in $d\geq 5$.

\paragraph{Bound in terms of $I_{n,l}, L_{n}, V_{n}$}
We first show how we bound the integrals defining $K_{n,l}(x),$ $U_{n,l}(x)$ and $T_{n,l}$ in terms of  $I_{n,l}$, as well as the related integrals $L_n(x)$ and $V_{n,l}$ defined by
	\begin{align}
	\lbeq{Analysis-Def-Integral-L}
	L_n(x)&=\int_{(-\pi,\pi)^d}\hat C(k)^n\hat D^{(x)}(k)^2\frac {d^dk}{(2\pi)^d}
	\end{align}
and
	\begin{align}
	\lbeq{Analysis-Def-Integral-V}
	V_{n,l}&=\int_{(-\pi,\pi)^d}\frac { \hat D(k)^l[\hat D^{\sin}(k)]^2}{[1-\hat D(k)]^n}\frac {d^dk}{(2\pi)^d}.
	\end{align}
We use the Cauchy-Schwarz inequality to bound $K_{n,l}(x)$ and $U_{n,l}(x)$ defined in \refeq{Analysis-Def-Integral-K} and  \refeq{Analysis-Def-Integral-U} by
	\begin{align}
	K_{n,l}(x)\leq& [I_{n,2l}(0)L_n(x)]^{1/2},\qquad\qquad
	U_{n,l}(x)\leq [V_{n,2l}L_n(x)]^{1/2}.
	\end{align}
To bound $T_{n,l}$ defined in \refeq{Analysis-Def-Integral-T}, we use \refeq{Dsinabsolute-bound} and $|\hat D^{\sin }(k)|\leq 1/d$, respectively, to compute
	\begin{eqnarray}
	|\hat M(k)|\leq |\hat D(k)|+2|\hat D^{\sin}(k)|\hat C(k)
	\leq  |\hat D(k)|+\min\left\{\frac 4 d,\frac 2 d\hat C(k)\right\}.
	\end{eqnarray}
This leads to
	\begin{eqnarray}
	T_{n,l}(x)&\leq&\int_{(-\pi,\pi)^d}|\hat D^{l} (k)| \hat C(k)^n |\hat M(k)||\hat D^{(x)}(k)|
	\frac {d^dk}{(2\pi)^d}\nnb
	&\leq& K_{n,l+1}(x)+\min\left \{ \frac 4 d K_{n,l}(x),\frac 2 d K_{n+1,l}(x)\right\}.
	\end{eqnarray}
Next, we discuss improvements for the bounds on $K_{n,l}(x)$ and $U_{n,l}$. As $|\hat D^{\sin}(k)|\leq 1/d$ (recall \refeq{Dsinabsolute-bound}), we know that
	\begin{align}
	U_{n,l}(x)\leq&\frac 1 d K_{n,l}(x).
	\end{align}
For $x=0$ and even $l$, we use $\hat D^{\sss (0)}(k)=1$, $\hat D^{\sin}(k)\geq 0$ and \refeq{Dsin-Split} to compute
	\begin{align}
	U_{n,l}(0)= \frac 1 {2d} \left( I_{n,l}(0) - I_{n,l}(2\ve[1])\right).
	\end{align}
For $x=0$ we can use a better bound for $K_{n,l}$ in the form
	\begin{align}
	K_{n,l}(0)&\begin{cases}=I_{n,l}(0)&\text{ if $l$ is even},\\
	\leq I_{n,l-1}(0)^{1/2} I_{n,l+1}(0)^{1/2}&\text{ if $l$ is odd}.\end{cases}
	\end{align}
Moreover, we use a different bound for $l=0$. We note that, for $n\geq 1$,
	\begin{eqnarray}
	\frac 1 {[1-\hat D(k)]^n}=\frac 1 {[1-\hat D(k)]^{n-1}}
	+\frac {\hat D(k)} {[1-\hat D(k)]^{n-1}}+\frac {\hat D(k)^2} {[1-\hat D(k)]^n},
	\end{eqnarray}
which implies that $K_{n,l}(x)\leq K_{n-1,l}(x)+K_{n-1,l+1}(x)+K_{n,l+2}(x)$, and thus
	\begin{align}
	K_{n,0}(x)\leq& K_{n-1,0}(x)+[I_{n-1,2}(0)L_{n-1}(x)]^{1/2}+[I_{n,4}(0)L_n(x)]^{1/2}.
	\end{align}

\paragraph{Computation of $L_{n}$.}
By the definition of $\hdx$ in \refeq{Analysis-dhx}, it is not difficult to see that
	\begin{align}
	L_n(x)=&\int_{(-\pi,\pi)^d}\frac { (\hdx)^2}{[1-\hat D(k)]^n} \frac {d^dk}{(2\pi)^d}
	\lbeq{LN-sumcharacterization-pre}
	=\frac 1 {2^d d!} \sum_{\mu\in \mathcal{P}_d }\sum_{\delta\in\{-1,1\}^d}  I_{n,0}(x+ p(x;\nu,\delta)).
	\end{align}
The set $\mathcal{P}_d$ and the operator $p(x;\nu,\delta)$ are defined in Definition \ref{defSignedPermuations}. As we can compute $I_{n,0}(x)$, we can also compute the sum in \refeq{LN-sumcharacterization-pre} directly.

The value of $I_{n,0}(x)$ only depends on the number of entries that have a given absolute value,
so that we can reduce the domain over which we sum. We explain this in two examples:
\bigskip

\noindent
{\bf Example 1: Computation of $L_n(\ve[1])$.} As the first example, we show that
	\begin{align}
	\lbeq{numericLnNeighbo}
	L_n(\ve[1])=& \frac 1 {2d} I_{n,0}(0)+\frac 1 {2d} I_{n,0}(2\ve[1])+\frac {d-1}{d} I_{n,0}(\ve[1]+\ve[2]).
	\end{align}
By symmetry, $I_{n,0}(\ve[1]+ p(\ve[1];\nu,\delta))=I_{n,0}(\ve[1]+\ve[2])$ for all $\delta\in\{-1,1\}^d$ and $\nu\in \mathcal{P}_d$ with $\nu_1\neq 1$. This explains the third summand of \refeq{numericLnNeighbo}, where we note that there are $(d-1)!(d-1)$ permutations $\nu$ with $\nu_1\neq 1$. That leaves $(d-1)!$ permutations $\nu$ with $\nu_1=1$. As all entries of $p(\ve[1];\nu,\delta)$ except the first one are zero,  the values $\delta_{2},\dots,\delta_{d}$ do not affect the summand. If $\delta_1=1$, then $\ve[1]+ p(\ve[1];\nu,\delta)=2\ve[1]$ and if $\delta_1=-1$, then $\ve[1]+ p(\ve[1];\nu,\delta)=0$. The two cases correspond to the first and second term in \refeq{numericLnNeighbo} and complete the proof of \refeq{numericLnNeighbo}.
\bigskip

\noindent
{\bf Example 2: Computation of $L_n(\ve[1]+\ve[2])$.} As the second example, we derive that
	\begin{align}
	L_n(\ve[1]+\ve[2])=& \frac {(d-2)(d-3)}{d(d-1)} I_{n,0}(\ve[1]+\ve[2]+\ve[3]+\ve[4])\nnb
	&+\frac {d-2}{2d(d-1)}\left(I_{n,0}(\ve[1]+\ve[2])+I_{n,0}(2\ve[1]+\ve[2]+\ve[3])\right)\nnb
	\lbeq{numericLnNeighboNeighbo}
	&+\frac {1}{4d(d-1)}\left(I_{n,0}(0)+I_{n,0}(2\ve[1]+2\ve[2])+2I_{n,0}(2\ve[1])\right).
	\end{align}
There are $2(d-2)!$ permutations $\nu$ with $\{\nu_1,\nu_2\}=\{1,2\}$. Further, there are $2(d-2)!(2d-2)$ permutation $\nu$ that map $1$ to $\{3,\dots, d\}$ and $2$ to $\{1,2\}$.
	That leaves
	\begin{align}
	d!-2(d-2)!-4(d-2)!(2d-2)=(d-2)!(d-2)(d-3)
	\end{align}
permutations $\nu$ that do not map $1$ and $2$ to the first coordinates, i.e.,  $\nu$ for which $\{\nu_1,\nu_2\}\cap\{1,2\}=\varnothing$. For these $\nu$,
	\begin{align}
	I_{n,0}(\ve[1]+\ve[2]+ p(\ve[1];\nu,\delta)) = I_{n,0}(\ve[1]+\ve[2]+\ve[3]+\ve[4]),
	\end{align}
which yields the first summand of \refeq{numericLnNeighboNeighbo}. The second corresponds to the case that either $\ve[1]$ or $\ve[2]$ is mapped to one of the first two coordinates.  For example, let us assume $\nu_1=1$ and $\nu_2=3$, then
	\begin{align}
	\ve[1]+\ve[2]+  p(\ve[1];\nu,\delta)\in\{2\ve[1]\pm\ve[2]\pm \ve[3],\pm \ve[2]\pm \ve[3]\},
	\end{align}
depending on the sign of $\delta_1$. This gives the second summand of \refeq{numericLnNeighboNeighbo}.
If we map $\ve[1]+\ve[2]$ to both of the first two coordinates, then
	\begin{align}
	\ve[1]+\ve[2]+  p(\ve[1];\nu,\delta)\in\{0,2\ve[1],2\ve[2],2\ve[1]+2\ve[2]\},
	\end{align}
which only depends on $\delta_1$ and $\delta_2$. This gives the third summand of \refeq{numericLnNeighboNeighbo} and completes the proof of \refeq{numericLnNeighboNeighbo}.

\paragraph{Computation of $V_{n,l}$.}
The equality \refeq{Dsin-Split} implies that
	\begin{align}
	\hat D^{\sin}(k)^2&=\frac 1 {(2d)^2} [1 -\hat D(2k)]^2
	=\frac 1 {(2d)^2} \Big[1 -\frac {2}{2d}\sum_{\iota} \e^{2\ii  k_\iota }
	+\frac {1}{(2d)^2}\sum_{\iota,\kappa} \e^{2\ii (k_\iota+k_\kappa) }\Big].
	\end{align}
From this, we conclude that
	\begin{align}
	V_{n,l}&=\int_{(-\pi,\pi)^d}\frac { \hat D(k)^l[\hat D^{\sin}(k)]^2}{[1-\hat D(k)]^n}
	\frac {d^dk}{(2\pi)^d}\\
	&=\frac 1 {(2d)^2}\Big(I_{n,l}(0)-2I_{n,l}(2\ve[1])+\frac {d-1}{d}I_{n,l}(2\ve[1]+2\ve[2])
	+\frac 1 {2d} I_{n,l}(0)+\frac 1 {2d} I_{n,l}(4\ve[1]) \Big).\nn
	\end{align}

\paragraph{Bounds on the suprema of $I_{n,l}(x)$, $K_{n,l}(x)$, $T_{n,l}(x)$, $U_{n,l}(x)$.}
In Section \ref{secf3}, we have bounded $\Hcal^{n,l}_z(x)$ in terms of SRW-integrals.
To compute the bound on $f_3$, we need to rely on
	\begin{align}
	\lbeq{numericalSupSubList}
	\sup_{x\in S} I_{n,l}(x),\quad \sup_{x\in S} K_{n,l}(x),\quad \sup_{x\in S} T_{n,l}(x),\quad \sup_{x\in S}
	U_{n,l}(x),
	\end{align}
for different sets of vertices $S$. For finite sets, we simply take the maximum of the elements. To obtain  bounds for infinite $S$, we use monotonicity of the SRW-integrals as formulated in the following lemma:

\begin{lemma}[Monotonicity of $I_{n,l}(x)$ and $L_{n,l}(x)$ in $x$]
\label{NumericLemmaMonotonIL}
Let $n$ be a positive integer and consider $x,y\in\Zd$ with $x_1\geq x_2\geq\dots\geq x_d\geq 0$ and $y_1\geq y_2\geq\dots\geq y_d\geq 0$. Then,
	\begin{align}
	I_{n,l}(x+y)&\leq I_{n,l}(x),\qquad
	\qquad L_{n}(x+y)\leq L_{n}(x).
	\end{align}
\end{lemma}
\noindent

This lemma is a combination of \cite[Lemmas B.3 and B.4]{HarSla92b}.
In the previous section, we have obtained bounds on $K_{n,l}(x), T_{n,l}(x), U_{n,l}(x)$ for a given $x$ in terms of  $I_{n,l}(x), L_{n}(x)$ and $V_{n,l}$. Thus, the supremum in \refeq{numericalSupSubList}  is obtained for an $x$
with the lowest order in the sense of Lemma \ref{NumericLemmaMonotonIL}. Therefore, when we use the following infinite set $Q=\{ x \in \Zd\colon \sum_i |x_i|>2\}$, we bound the corresponding SRW-integrals by
\begin{align}
\sup_{x\in Q} I_{n,l}(x)&=\max\{I_{n,l}(3\ve[1]), I_{n,l}(2\ve[1]+\ve[2]), I_{n,l}(\ve[1]+\ve[2]+\ve[3])\},\\
\sup_{x\in Q} L_{n}(x)&=\max\{L_{n,l}(3\ve[1]), L_{n,l}(2\ve[1]+\ve[2]), L_{n,l}(\ve[1]+\ve[2]+\ve[3])\}.
\end{align}


\subsection{Bounds implied by the bootstrap functions and SRW integrals}
\label{secImpliedBounds}
For the analysis presented in the previous sections, we use that bounds on the
bootstrap functions $f_1$, $f_2$ and $f_3$ imply bounds on the NoBLE coefficients, see e.g., Assumptions \ref{assDiagBounds} \& \ref{assDiagBoundsCoeff}. In the model-dependent papers, we prove that the coefficients can be bounded by combinations of simple diagrams. Simple diagrams are diagrams arising from combinations of two-point functions, like the triangle diagram that we have already seen for percolation:
	\begin{align}
	\bigtriangledown_p(x)=(G_{z}\star G_{z}\star G_{z})(x).
	\end{align}
The simple diagrams can then be bounded using the bootstrap function $f_2$, e.g.,
	\begin{align}
	\bigtriangledown_p(x)
	&=\int_{(-\pi,\pi)^d}\hat G_z(k)^3 \e^{\ii k \cdot x}\frac{d^dk}{(2\pi)^d}\\
	&\leq \left(\frac{2d-2}{2d-1}\Gamma_2 \right)^3 \int_{(-\pi,\pi)^d}\hat C(k)^3
	\frac {d^dk}{(2\pi)^d}
	=\left(\frac{2d-2}{2d-1}\Gamma_2 \right)^3 I_{3,0}(0),\nn
	\end{align}
which can be computed numerically. In this section, we present the bounds that we use in our implementation of the analysis. These bounds are optimized so as to obtain the best numerical bounds on $\beta$ possible for Assumption \ref{assDiagBoundsCoeff}.

\subsubsection{Simple and repulsive diagrams}
In this section, we explain how to bound simple diagrams and {\em repulsive} diagrams, that we define below.

The bounds on the simple diagrams are model dependent.  However, for all models that we consider, the bounds are using the same idea that we present below. We require some notions that have not been introduced yet, namely, the minimal length of an interaction, the modified two-point function and repulsive diagrams. As an example, we give the definitions for percolation. The definitions for LT and LA are straightforward generalizations.

For percolation, we define $\{x \ct {m} y\}$ to be the event that there exists a path consisting of at least $m$ open bonds connecting $x$ and $y$. Similarly, we define $\{x \ct {\underline m} y\}$ as the event that there exists a path consisting of exactly $m$ open bonds connecting $x$ and $y$. [This is not the same as the event that the graph distance in the percolation cluster equals $m$, but it implies that it is at most $m$.]
To characterize these interactions, we define the adapted two-point function,  which are given for percolation by
	\begin{align}
	G_{m,z}(x)=\prob_z \left( x \ct {m} y \right), \qquad
	G_{\underline m,z}(x)=\prob_z \left( x \ct {\underline m} y \right).
	\end{align}
For all models under consideration the following holds, for all $n \in \Nbold$,
	\begin{align}
 	\lbeq{Additional-Bound-Probs-1}
	G_{\underline m,z}(x) &\leq (2d\aabz)^m c_m(x),
	\end{align}
where $c_m(x)$ is the number of $m$-step self-avoiding walks starting at the origin and ending at $x$.
A diagram is {\em repulsive} if the paths involved do not intersect. For percolation, for example, the repulsive bubble and triangle are given by
	\begin{align}
	\diagRepulsiveLetter{B}_{\underline m_1,m_2} (x)
	=&\sum_{y\in\Zd} \prob_z\left( \{ 0\ct{\underline m_1} y\}\circ\{ y\ct {m_2} x\}\right),\\
	\diagRepulsiveLetter{T}_{ m_1,m_2,m_3} (x)
	=&\sum_{v,y\in\Zd} \prob_z\left( \{ 0\ct {m_1}  v\}\circ\{ v\ct {m_2}  y\}\circ\{ y\ct {m_3} x\}\right),
	\end{align}
where the symbol $\circ$ denotes the disjoint occurrence, which is a standard notion in percolation theory, see e.g., \cite{Grim99}. For the examples above, it means that the occupied paths that are required to exist make use of disjoint sets of bonds.

Below, we use the symbol $(f\otimes g)(x)$ to describe that the paths involved for $f$ are disjoint from the path involved in $g$. For example, $(G_{n,z}\otimes G_{m,z})(x)$ represents the diagram where the path used in $G_{n,z}$ is disjoint from that used in $G_{m,z}$. We define $a_n(x)$ to be the number of $n$-step simple random walk from $0$ to $x$ that never use a bond twice, and note that $(a_{n}\otimes a_{m})(x)\leq a_{n+m}(x)$.
Further, $(a_n\otimes G_{m,z})(x)$ represents the combination of an $n$-step SRW path counted in $a_n$ and a percolation path of length at most $m$, where, given the $n$-step SRW path, $G_{m,z}(x)$ is the probability that an occupied percolation path exists that uses different bonds than the $n$-step SRW. The bounds that we explain below rely on the following bounds on the two-point function that hold for all models that we consider:
	\begin{align}
	\lbeq{Additional-Bound-Probs-2}
	G_{m,z}(x)&\leq \diagRepulsiveLetter{B}_{\underline m,0}  \leq (2d\aabz)^m (a_m\otimes
	G_z)(x),\\
 	\lbeq{Additional-Bound-Probs-3}
 	G_{m,z}(x) &\leq \sum_{i=m}^{s-1} G_{\underline i,z}(x)+G_{s,z}(x)\leq
	\sum_{i=m}^{s-1} (2d\aabz)^i a_i(x)+(2d\aabz)^s (a_s\otimes G_z)(x),\\
  	\lbeq{Additional-Bound-Probs-4}
 	\diagRepulsiveLetter{B}_{n,m}(x) &\leq  (2d\aabz)^n   (a_{n}\otimes G_{m,z})(x)
	+ \diagRepulsiveLetter{B}_{n+1,m}(x),
	\end{align}
for $s,n,m\in\Nbold$ with $s>m$.

\subsubsection{Bounds on simple diagrams}
Here we derive efficient numerical bounds on simple diagrams. Throughout this section, we fix $n\in\Nbold$ and $m_i\in\Nbold$ for $i=1,\dots, n$. Further, we use the notation $m_{i,j}=\sum_{s=i}^j m_s$.

Assuming bounds on the bootstrap functions, we obtain the following bound for non-repulsive diagrams from \refeq{Additional-Bound-Probs-1}:
	\begin{align}
	(G_{m_1} \star G_{m_2}\star \cdots \star G_{m_n})&\leq
	(2d\aabz)^{m_{1,n}} ( D^{\star m_{1,n}} \star G^{\star n})(x)\nnb
	&\leq \left(\frac {2d}{2d-1}\Gamma_1\right)^{m_{1,n}}
	\left( \frac {2d-2}{2d-1}\Gamma_2\right)^n
	K_{n,m_{1,n}}(x).
	\end{align}
When $m_{1,n}\geq 10$, we also use this bound also for the repulsive diagrams.  For $m_{1,n}<10$, we instead use the repulsiveness to reduce the numerical bounds in most cases by around $50\%$. We obtain these improved bound by extracting short, explicit contributions. In high dimensions, short connections typically give the leading contribution to diagrams, so treating them more precisely often pays off. This
requires the computation of $a_n(x)$, for all $x\in\Zd$ and $n\in\Nbold$. For $n<10$, we compute these values using a simple Java program that can be downloaded from the website of the first author \cite{FitNoblePage}.

We start by explaining the bound on the repulsive bound for the example of a bubble.
We fix an $M\in\Nbold$ with $M \geq m_1+m_2$ and use \refeq{Additional-Bound-Probs-4} to obtain
	\begin{align}
	\lbeq{Numeric-Bubble-Advanced}
 	\diagRepulsiveLetter{B}_{m_1,m_2}(x)
	\leq &\sum_{i=m_1}^{M-m_2-1} (2d \aabz)^i (a_{i}\otimes G_{z,m_2})(x)
	+ \diagRepulsiveLetter{B}_{M-m_2,m_2}(x) \\
 	\leq& \sum_{s_1=m_1}^{M-m_2-1}\Big(\Big[\sum_{s_2=m_2}^{M-1-s_1}
	a_{s_1+s_2}(x) \aabz^{s_1+s_2}\Big]
 	+ (2d\aabz)^M (D^{\star M}\star G_{z})(x)\Big)\nnb
	&+(2d\aabz)^M (D^{\star M}\star G^{\star 2}_{z})(x)\nnb
	=&\sum_{i=m_{1,2}}^{M-1} (i+1-m_{1,2}) a_{i}(x) \aabz^{i} \nnb
	&+(M-m_{1,2})(2d\aabz)^M (D^{\star M}\star G_{z})(x)
	+(2d\aabz)^M (D^{\star M}\star G^{\star 2}_{z})(x).\nn
	\end{align}
We extend this idea to obtain a bound on the triangle of the form
	\begin{align}
	\lbeq{Numeric-Triangle-Advanced}
 	\diagRepulsiveLetter{T}_{m_1,m_2,m_3}(x)
	\leq&  \sum_{i=m_{1,3}}^{M-1} a_{i}(x) \aabz^{i} \sum_{s_1=m_1}^{i-m_{2,3}}
	\sum_{s_2=m_{2}}^{i-m_3-s_1} 1 \\
	&+ \sum_{s=m_1}^{M-m_{2,3}-1} (M-m_{2,3}-s)(2d\aabz)^M (D^{\star M}\star G_{z})(x)\nnb
	&+(M-m_{1,3})(2d\aabz)^M (D^{\star M}\star G^{\star 2}_{z})(x)
	+(2d\aabz)^M (D^{\star M}\star G^{\star 3}_{z})(x)\nnb
	=& \sum_{i=m_{1,3}}^{M-1} \frac {(i+1-m_{1,3})(i+2-m_{1,3})}{2}a_{i}(x) \aabz^{i}\nnb
	&+\frac {(M-m_{1,3})(M-1-m_{1,3})}{2}(2d\aabz)^M (D^{\star M}\star G_{z})(x) \nnb \nn
	&+(M-m_{1,3})(2d\aabz)^M (D^{\star M}\star G^{\star 2}_{z})(x)
	+(2d\aabz)^M (D^{\star M}\star G^{\star 3}_{z})(x).
	\end{align}
In the same way, we bound the square by
	\begin{align}
	\lbeq{Numeric-Square-Advanced}
 	\diagRepulsiveLetter{S}&_{m_1,m_2,m_3,m_4}(x) \\
	\leq& \sum_{i=m_{1,4}}^{M-1}\frac 1 6 \prod_{s=1}^3 (i-m_{1,4}+s)  a_{i}(x) \aabz^{i}
	+ \frac 1 6 \prod_{s=1}^3 (M-m_{1,4}+s)(2d\aabz)^M (D^{\star M}\star G_{z})(x)\nnb
	&+\frac {(M-m_{1,4})(M-1-m_{1,4})}{2}(2d\aabz)^M (D^{\star M}\star G^{\star 2}_{z})(x) \nnb \nn
	&+(M-m_{1,4})(2d\aabz)^M (D^{\star M}\star G^{\star 3}_{z})(x)
	+(2d\aabz)^M (D^{\star M}\star G^{\star 4}_{z})(x).
	\end{align}

\subsubsection{Bounds on weighted diagrams}
Weighted diagrams, such as $\Hcal^{n,l}_z(x)$ in \refeq{Xspace-Analysis-Object},  are bounded using the bootstrap function $f_3$, defined in \refeq{defFunc3}. It is especially beneficial to extract explicit contributions from the weighted diagrams, as the bound produced by $f_3$ is not very sharp and $\Hcal^{1,l}_z(x)$ decreases quite fast when we increase $l$. We conclude from \refeq{Additional-Bound-Probs-3}  that, for $l\leq M$,
	\begin{align}
	\Hcal^{1,l}_z(x)&=\sum_{y}\|y\|_2^2 G_z(y)(G_z \star D^{\star l})(x-y)
	\leq \sum_{i=l}^{M-1} (2d\aabz)^{i-l} 	\Hcal^{0,i}_z(x)\\
	&\qquad+ (2d\aabz)^{M-l} \Hcal^{1,M}_z(x).\nn
	\lbeq{bound-on-weighted-diags}
	\end{align}
For $x$ that are close to the origin we can bound $\Hcal^{0,i}_z(x)$ quite efficiently.
We abbreviate $H_{z}(x)=\|x\|_2^2G_z(x)$ and compute that
	\begin{align}
	\Hcal^{0,0}_z(0)&=H_{z}(0)=0,\\
	\lbeq{hcal01Bound}
	\Hcal^{0,1}_z(0)&=(D\star H_{z})(0)=\frac 1 {2d} \sum_\iota H_{z} (\ve[\iota]) = G_z(\ve[1]),\\
	\Hcal^{0,2}_z(0)&= (D\star D\star H_{z})(0)
	=\frac 1 {(2d)^2} \sum_{\iota,\kappa}
	\|\ve[\iota]+\ve[\kappa]\|_2^2 G_{z}(\ve[\iota]+\ve[\kappa]),\nnb
	&=\frac 1 {2d} \left( 4(2d-2)G_{z}(\ve[1]+\ve[2])+4 G_{z}(2\ve[1]) \right).
	\end{align}
Most of the weighted diagram that arise in our bounds on the NoBLE coefficient are repulsive. Using this repulsiveness, we can make another substantial improvement. As an example, let us consider $(a_1\otimes H_{z})(0)$. All connections counted in $H_z$ need to make at least three steps as the direct step is used by $a_1$, so that
	\begin{align}
	\frac 1 {2d} (a_1\otimes H_{z})(0)=(D\otimes H_{z})(0) =G_{3,z}(\ve[1]),
	\end{align}
which is numerically a factor $1/2d$ better than the bound in \refeq{hcal01Bound}.  If we consider two explicit steps of $\ve[\iota]+\ve[\kappa]$, then we obtain the bound
	\begin{align}
	(a_2\otimes H_{z})(0)\leq& 8d ( (2d-2)\aabz^4 + G_{6,z}(2\ve[1]))\\
	&\quad +8d(2d-2)(\aabz^2+4(2d-3) \aabz^4 + G_{6,z}(\ve[1]+\ve[2])).\nn
	\end{align}
In our computations we have extracted all paths up to a length of six and manually computed the number of  percolation paths that do not use any bond of the first  path. This leads to excellent bounds on closed weighted repulsive diagrams. For bubbles and triangles, we  extend the idea used in \refeq{Numeric-Bubble-Advanced}-\refeq{Numeric-Triangle-Advanced} for $m_1,m_2\in\Nbold$ with $m_1+m_2<6$, to compute that
	\begin{align}
	(G_{m_1,z}\otimes H_{z})(0)\leq& \sum_{i=m_1}^5 \aabz^i (a_{i}\otimes H_{z})(0)
	+(G_{6,z}\otimes H_{z})(0),\\
	(G_{m_1,z}\otimes G_{m_2,z}\otimes H_{z})(0)
	\leq& \sum_{i=m_1+m_2}^{5} \tfrac{1}{2}(i+1-m_1-m_2)(i+2-m_1-m_2)
	\aabz^i(a_i \otimes H_{z})(0)\nnb
	&+(6-m_1-m_2)\aabz^6 (a_6 \otimes G_{z}\otimes H_{z})(0)\nnb
	&+\aabz^6(a_6 \otimes G_{z}\otimes G_{z}\otimes H_{z})(0).
	\end{align}
The terms involving $a_{i}\otimes H_{z}$ are computed explicitly. All contributions involving a $G_z$ factor also contain a factor $a_6$, numerically making them of order $d^{-3}$. Since $H_z(y)=\|y\|^2G_z(y)$, we can bound these minor contributions using $f_3$ by removing the added restrictions that the $\otimes$ convolution imposes and replacing it by a normal convolution. In this bound, we even bound $a_6(y)\leq (2d)^6 D^{\star 6}(y)$. Following this strategy, we reduce the effect of the bad bound $f_3$ and enhance our numerical bounds.

\subsection{Analysis of matrix-valued diagrammatic bounds}
In Section \ref{secRewrite}, we use the bounds in Assumption \ref{assDiagBoundsCoeff} to bound the terms appearing in the rewrite of the NoBLE equation. These bounds are stated in terms of the functions $\Xi^{\ssc[N]}$ and $\Xi^{\ssc[N],\iota}$. Bounds on these quantities are proved in the model-dependent articles, where they are stated in terms of matrix products, such as
	\begin{align}
 	\lbeq{MatrixBoundExample}
	\sum_x  \Xi^{\ssc[N]}(x) &\leq \vec v^T {\bf B}^{N} \vec w,\\
 	\lbeq{MatrixBoundExample2}
	\sum_x \|x\|_2^2  \Xi^{\ssc[N]}(x)
	&\leq
	(N+2)\Big(\vec h^T {\bf B}^{N} \vec w+\sum_{M=0}^{N-1}\vec v^T
	{\bf B}^M{\bf C}  {\bf B}^{N-M-1} \vec w+\vec v^T {\bf B}^{N} \vec h\Big),
	\end{align}
where $\vec h,\vec v,\vec w\in\Rbold_+^n$ and ${\bf B},{\bf C}\in\Rbold_+^{n\times n}$ for some $n\geq 2$ and $N\geq 0$. We need to sum these bounds over various sets of $N$ to create $\beta_{\bullet}^{\sss \text{abs}},\beta_{\bullet}^{\sss \text{odd}}$ and $\beta_{\bullet}^{\sss \text{even}}$.

In this section, we explain how we compute the sum of these estimates. As a first step, we compute the eigensystem/spectrum of ${\bf B}$, so the left eigenvectors $\vec \eta_i$ and right eigenvectors $\vec \zeta_i$ to the eigenvalue $\lambda_i$. In our applications, there always exists a set of $n$ independent left and right eigenvectors. As these vectors are linearly independent, there exists $r_1,\dots, r_n$ and $b_1,\dots,b_n$ such that
	\begin{eqnarray}
	\lbeq{Numeric-Bound-vector1}
	\vec v=\sum_{i=1}^n r_i\vec \eta_i, \qquad\qquad \vec w=\sum_{i=1}^n b_i\vec \zeta_i.
	\end{eqnarray}
We compute $r_1,\dots,r_n$ using relations of the form
	\begin{align}
	\vec v=\sum_{i=1}^n r_1\vec \eta_i=
    	{\boldsymbol{\eta}}
	\vec{a},
\end{align}
where $\vec{r}^T=(r_1,\ldots, r_n)$, while the $i$th row of the matrix ${\boldsymbol{\eta}}$ equals $\vec \eta_i^T$. As the rows of the matrix ${\mathbf \eta}$ are independent vectors, ${\boldsymbol{\eta}}$ is invertible, so that
	\begin{align}
	{\boldsymbol{\eta}}^{-1}
	\vec v=
	\vec{r},
	\end{align}
which allows us to compute the $a_i$. The $b_i$'s are computed in the same way. We define
	\begin{align}
  	\vec v_i=r_i\vec \eta_i,\qquad\qquad
  	\vec w_i=b_i\vec \zeta_i,
	\end{align}
and note that these are also eigenvectors of ${\bf B}$ to the eigenvalue $\lambda_i$. Thus, for $N\geq 0$,
	\begin{align}
	\sum_x  \Xi^{\ssc[N]}(x) &\leq \vec v {\bf B}^{N} \vec w
	=\vec v \Big(\sum_i \lambda_i^{N} \vec w_i\Big).
	\end{align}
Using a geometric sum, we obtain
	\begin{align}
	\sum _{x,N}\sum_x  \Xi^{\ssc[N]}(x) \leq
	\sum_i \frac {1} {1-\lambda_i} \vec v \vec w_i
	=: \beta_{\sss \Xi}^{\sss \text{abs}}.
    \lbeq{MatrixBoundExample3}
	\end{align}
To create a closed form expression for the bound on the weighted diagrams appearing
in \refeq{MatrixBoundExample2} and use that $\vec v_i$ and $\vec w_j$ are eigenvectors of ${\bf B}$ to obtain
	\begin{align}
	\sum_{N\geq 0}\sum_{x} \|x\|_2^2  \Xi^{\ssc[N]}(x) \leq &
	\vec h^T \Big(\sum_i \sum_{N=0}^\infty(N+2)\lambda_i^N\vec w_i\Big)
	+\Big( \sum_i\sum_{N=0}^\infty (N+2) \lambda_i^N\vec v^T_i\Big) \vec h\nnb
	&+\sum_{N=0}^\infty (N+2) \sum_{M=0}^{N-1}\Big(\sum_i\lambda^M_i \vec v_i^T\Big)
	{\bf C} \Big(\sum_j\lambda_j^{N-M-1}\vec w_j\Big).
	\end{align}
For the second line, we rewrite the sums over $N,M$ for fixed $i,j$ as
	\begin{align}
	\sum_{N=0}^\infty (N+2) \sum_{M=0}^{N-1}\lambda^M_i \lambda_j^{N-M-1} \vec v_i^T{\bf C} \vec w_j
	&=\sum_{N=0}^\infty\sum_{M=0}^\infty (N+M+2) \lambda^M_i \lambda_j^{N} \vec v_i^T
	{\bf C} \vec w_j.
	\end{align}
We use the geometric sum identity
	\begin{align}
	\sum_{n=0}^\infty (n+1)\lambda^n=\frac {1} {(1-\lambda)^2},
	\end{align}
to bound \refeq{MatrixBoundExample2} as
	\begin{align}
	\sum_{N,x}& \|x\|_2^2  \Xi^{\ssc[N]}(x) \nnb
	\leq &
	\vec h^T \Big(\sum_i \sum_{N=0}^\infty(N+2)\lambda_i^{N}\vec w_i\Big)
	+\Big( \sum_i\sum_{N=0}^\infty (N+2) \lambda_i^{N}\vec v^T_i\Big) \vec h\nnb
	&+\sum_{i,j}\sum_{N=0}^\infty\sum_{M=0}^\infty (N+1) \lambda^M_i \lambda_j^{N}
	\vec v_i^T{\bf C} \vec w_j
	+\sum_{i,j}\sum_{N=0}^\infty\sum_{M=0}^\infty (M+1) \lambda^M_i
	\lambda_j^{N} \vec v_i^T{\bf C} \vec w_j\nnb
	=&\sum_i  \vec h^T\lambda_i\vec w_i
	\Big(\frac 1 {(1-\lambda_i)^2}+\frac {1}{1-\lambda_i}\Big)
	+ \sum_i  \lambda_i\vec v^T_i\vec h\Big(\frac 1 {(1-\lambda_i)^2}+\frac {1}{1-\lambda_i}\Big) \nnb
	&+\sum_{i,j} \frac {\vec v_i^T{\bf C} \vec w_j} {(1-\lambda_i)^2(1-\lambda_j)}+
	\sum_{i,j}\frac {\vec v_i^T{\bf C} \vec w_j} {(1-\lambda_i)} \frac 1 {(1-\lambda_j)^2}
	=: \beta_{\sss \Delta \Xi}^{\sss \text{abs}}.
    \lbeq{MatrixBoundExample4}
	\end{align}
This highlight how we bound $\Xi^{\ssc[N]}$. The bounds on $\Xi^{\ssc[N],\iota}$ are obtained in a similar way.
\medskip


\begin{remark}[The numerics of the matrix powers and the eigensystem]
\label{rem-eigen}
In the classical lace expansion, similar exponential bounds appear for the $N$th lace-expansion coefficient as a function of $N$, in which the base of the exponential is roughly bounded by the sum of the matrix elements $\sum_{i,j} {\bf B}_{i,j}$ (in fact, it is worse, since there loops do not have length at least 4). We use the matrix-valued bound to optimally use the fact that the number of steps in shared lines in loops appearing in our lace-expansion bounds provide information about the number of steps in other lines that are part of the same loop. Such bounds are most easily expressed in terms of matrix products. In these matrix bounds, the largest eigenvalue of ${\bf B}$ decides the magnitude of the bounds (see \refeq{MatrixBoundExample3} and \refeq{MatrixBoundExample4}). For example, for percolation and $d=11$, the matrix ${\bf B}$ equals
	\begin{align*}
	{\bf B}=\begin{pmatrix}
  	0.0134202&	0.0112907&	0.0257405 \\
	0.0127527&	0.0108018&	0.0338533\\
	0.028009&	0.0260537&	0.0401418
	\end{pmatrix}.
  	\end{align*}
For this matrix, the largest eigenvalue equals $\lambda_1\approx 0.073$, which is quite small, certainly when compared to $\sum_{i,j} {\bf B}_{i,j}\approx 0.2$.
\end{remark}

\section{Completion of the bootstrap argument and conclusions}
\label{sec-Comp}
In this section, we complete the bootstrap argument, and explain how the conditions are verified in  Mathematica notebooks. We start by summarizing where we stand.

\paragraph{Verification of the bootstrap conditions.}
In Section \ref{secVeriBootstrap}, we have verified the bootstrap conditions. Part of this verification was the improvement of the bootstrap bounds on the functions $f_1$, $f_2$ and $f_3$ defined in \refeq{defFunc1}--\refeq{defFunc3}, the latter being technically the most demanding.
A sufficient condition on the bounds on the NoBLE coefficients that allows us to improve the bounds on $f_1,f_2$ and $f_3$ is formulated in Definition \ref{finalCondition} in terms of bounds on a simplified rewrite of the NoBLE formulated in \refeq{generalForm-simple}, that is formulated in Assumption \ref{assDiagBounds}. These bounds on the rewrite are reformulated in terms of the original NoBLE coefficients $\Xi_z,\Xi_z^\iota,\Psi^\kappa_z$ and $\Pi^{\iota,\kappa}_z$ in Appendix \ref{secRewriteBounds}. The assumptions that we need to verify in the model-dependent papers \cite{FitHof13g} and \cite{FitHof13d}
are Assumptions \ref{assXBoundInitialCondition}-\ref{assContinuous} and Assumptions \ref{assSymtwo}--\ref{assDiagBoundsCoeff}. The initial Assumptions \ref{assSym}--\ref{assGzBehavoirCritical} are replaced using Proposition \ref{assTranslation} in Section \ref{secRewrite}.
\paragraph{Implications of the completed bootstrap and proof main result.}
The bound on $f_1$ implies a bound on $z_c$, while that on $f_2$ implies that the infrared bound holds with an explicit estimate on the constant given by $\Gamma_2$. Thus, $f_1$ and $f_2$ imply our main results for the model-dependent analysis for percolation in  \cite{FitHof13d} and for LT and LA in \cite{FitHof13g}. The bound on $f_3$ implies bounds on simple weighted diagrams, such as weighted lines, bubbles and triangles. These bounds are crucial in improving the bound on $f_2$, which implies our main result.

We now discuss the improvement of $f_3$ in more detail, splitting between the model-independent and the model-dependent parts, and their relations to SRW integrals, the bounds on the NoBLE coefficients and the Mathematica notebooks that finally complete the analysis and thus complete our proofs. We begin by discussing the model-independent improvement of $f_3$ and SRW integrals.
\paragraph{Model-independent improvement of $f_3$ and SRW integrals.}
In Section \ref{secf3}, we have proven that the bootstrap conditions on $f_3$ hold.
Both for the verification of the conditions on the initial point in Section \ref{secAnalysisXspaceF3Initial}, as well as for the improvement of the bounds on $f_3$ in Sections \ref{secAnalysisXspaceF3Decomposition}--\ref{subsecImproF3bar}, we have formulated our conditions in terms of $x$-space SRW integrals such as
$I_{n,l}(x), K_{n,l}(x),
T_{n,l}(x), U_{n,l}(x),  L_{n}(x), V_{n}(x)$ and $\Isupx_{n,l}(x)$. The numerical values of these functions are crucial to allow us to verify that the required bounds on the initial point hold, and to improve the bound on $f_3$. Thus, in order to perform a successful bootstrap analysis, we need to obtain rigorous numerical bounds on such SRW integrals. These numerical bounds are explained in detail in Section \ref{secNumerics}, using the ideas by Hara and Slade in \cite{HarSla92b}. These values are formulated in terms of integrals of Bessel-functions and are computed analytically, up to a specified precision in the Mathematica notebook available at \cite{FitNoblePage}. This notebook needs to be compiled before the model-dependent parts can be performed, as the model-dependent analyses use them.
We will explain the Mathematica notebooks in the next paragraphs. We first explain how the bounds on $f_1,f_2$ and $f_3$ can be used to obtain sharp numerical bounds on the NoBLE coefficients.
\paragraph{Model-dependent improvement of $f_3$ and bounds on NoBLE coefficients.}
Having accurate numerical bounds on all the SRW-integrals involved at hand, we can obtain all bounds on $f_3$ formulated in Section \ref{secf3}, see in particular \refeq{bound-on-f3}. The remaining main ingredient of the bounds on $f_3$ is formed by the bounds on the NoBLE coefficients. The initial bounds on $f_1,f_2,f_3$ imply bounds on $z_c$, $\hat{G}_z(k)$, as well as on several simple weighted diagrams implied by $f_3$, in terms of $\Gamma_1,\Gamma_2$ and $\Gamma_3$. These, in turn, allow us to prove bounds on the NoBLE coefficients and to identify $\beta_{\bullet}^\ssc[N]$ for every $N\geq 0$ and $\bullet \in \{ \Xi, \Xi^{\iota}, \Delta\Xi, \{\Delta \Xi^\iota,0\},\{\Delta \Xi^\iota,1\}\}$, as formulated in Assumption \ref{assDiagBoundsCoeff}. This is, next to the derivation of the NoBLE, the second main result in the model-dependent papers \cite{FitHof13d} and \cite{FitHof13g}, where \cite{FitHof13d} treats percolation, and \cite{FitHof13g} LT and LA. As a result, we then have all necessary bounds needed to verify the improvement of the bootstrap bounds. The numerical verification is performed in several Mathematica notebooks as we explain next.
\paragraph{Model-dependent and computer-assisted verifications using Mathematica.}
In order to complete the model-dependent and computer-assisted proof, we need to run the model-dependent Mathematica notebooks that can be found at the first author's web page \cite{FitNoblePage}.
We first need to choose the dimension in \verb|SRW_basic.nb| and run the file.
This creates the numerical input needed for the model-dependent NoBLE computations, which is done in the files \verb|Percolation.nb|, \verb|LT.nb|, and \verb|LA.nb|.
\footnote{At the time of publication the LT and LA Mathematica files are not made publicly available, as the corresponding article is not finalised.}
When compiled, these model-dependent files compute whether $P(\gamma,\Gamma,\cdot)$ holds for the given input, $\Gamma_1,\Gamma_2,\Gamma_3$ and $c_\mu,c_{n,l,S}$, by going through the loop described in Figure \ref{fig-Heuristic-bootstrap}: We first compute bounds on simple diagrams for the initial point $z_I$, for which we do not need to rely on the bootstrap bounds in terms of $\Gamma_1,\Gamma_2,\Gamma_3$, but we can rely directly on the link to SRW as formulated in Assumption \ref{assXBoundInitialCondition}.
For $z\in (z_I,z_c)$, we do rely on the bootstrap bounds to conclude bounds on the bootstrap functions.
When both the bounds for $z_I$ and for $z\in (z_I,z_c)$, are bounded by the chosen $\Gamma_i$ for $i=1,2,3$,
we conclude the existence of appropriate $\gamma_i$, e.g.\ $\gamma_i=(\Gamma_i+\text{(computed bound)})/2$,
we have proven that $P(\gamma,\Gamma,\cdot)$ holds.
This numerical verification completes the argument for the given model in the given dimension. In \cite{FitHof13d} and \cite{FitHof13g},
it is also explained how we can then use monotonicity in the dimension $d$ to obtain the result for {\em all} dimensions larger than the specified dimension.
Further, the model-dependent notebooks contains an algorithm that helps to choose the optimal value for the constants $\Gamma_1,\Gamma_2,\Gamma_3$ and $c_\mu,c_{n,l,S}$. This is explained in the implementation. See e.g., \cite{FitHof13d} for a more extensive discussion for percolation.
\paragraph{Complete version of the Mathematica SRW files.}
Next to the basic version of the SRW notebooks, we also provide a complete version, which should be used when dealing with the models in relatively low dimensions. The notebook \verb|SRW.nb| serves two purposes. Firstly, it allows us to compute SRW integrals with the desired precision in dimensions close to the upper critical dimension. The file \verb|SRW_basic.nb| uses built-in Mathematica functions, and can only be used for percolation in $d\geq 15$, as otherwise the desired numerical precision cannot be guaranteed. In the file \verb|SRW.nb|, we use a Taylor approximation of the Bessel function, as explained in detail in \cite[Appendix B]{HarSla92b}, instead, that gives reliable results for $d\geq 9$. These computations make compiling \verb|SRW.nb| take around an hour, while \verb|SRW_basic.nb| is compiled in less than a minute. Second, \verb|SRW.nb| allows the use of arbitrary index sets $\mathcal{S}$ for $f_3$, see \refeq{defFunc3}. Using the basic version only the vertex sets $S=0$ and $S=\Zd\setminus\{0\}$ can be considered. These extensions are crucial to reduce the dimension above which our results apply.
\paragraph{Acknowledgements.}
This work was supported in part by the Netherlands Organisation for Scientific Research (NWO) through VICI grant 639.033.806 and the Gravitation {\sc Networks} grant 024.002.003. We thank
David Brydges, Takashi Hara and Gordon Slade for their constant encouragement, as well as for several stimulating discussions. This work builds upon the work by Takashi Hara and Gordon Slade. We particularly thank Takashi Hara for sharing his handwritten notes on the proof of mean-field behavior for $d\geq 19$, and explaining how this can be extended to $d\geq 15$. We have thoroughly enjoyed our animated discussions with Takashi in July 2013, which allowed us to compare notes and estimates on triangles, two-point functions, etc. Without these discussions, it would have been much harder to compare our results to those of Hara and Slade. The work of RF was performed in part at Stockholm University in the period September 2013 until September 2015. The authors gratefully acknowledge the advice of the referee, who found many typos and gave many suggestions that significantly improved the presentation of this paper.

\appendix
\section{Notation}
\begin{table}[h]
\begin{tabular}{|c|c|c|}
  \hline
  Notation & brief description & defined in\\
  \hline
  SRW & simple random walk  												                & Section \ref{subsecintroSRW}\\
  NBW & non-backtracking random walk  										                & Section \ref{subsecintroNRW}\\
  \hline
  & & \\[-2mm]
  $D$ & SRW step distribution                                                               & \refeq{DefDK} \\
  $\iota,\kappa$ & direction of a bond $\iota,\kappa\in\{\pm 1,\dots,\pm d \}$              & Section \ref{subsecintroSRW} \\
  $u,v,w,x,y$ & points on the lattice: $\Zd$                                                & \\
  $k$ & Fourier argument, so $k\in(-\pi,\pi)^d$ 								       	    & Section \ref{subsecintroSRW} \\
  $z$ & parameter of a generating function,  $z\in [0, z_c]$         						& \\
  $f\star g, f^{\star n}$ & convolution of functions $f,g\mapsto \Zd$                       &  \refeq{definition-convolution}\\
  $\partial_i f, \Delta f$ & directional derivative and Laplace operator         &  \refeq{Xspace-Analysis-deltaArgument}\\
  \hline
  & & \\[-2mm]
  $C_z, B_z$ &SRW and NBW two-point functions 	             				    			&  \refeq{genSRW}, \refeq{NBWRecScheme3} \\
  $\mJ$ & permutation matrix with entries $(\mJ)_{\iota,\kappa}=\delta_{\iota,-\kappa}$ 	& \refeq{NBW-MatrixDefinition}\\
  $\mD[k]$ & diagonal matrix with entries $(\mD[k])_{\iota,\kappa}=\delta_{\iota,\kappa} \e^{\ii k_\iota}.$ & \refeq{NBW-MatrixDefinition}\\
  $\chi(z)$ & susceptibility of a given model 				              					&  \refeq{susceptibility-SRW}\\
  $G_z$ & two-point function of general model 									            &  \refeq{basicgeneral1},\refeq{generalForm}, \refeq{generalForm-simple} \\
  $\Xi_z, \Xi^{\iota}_z$ & coefficient of the NoBLE expansion 						        &  \refeq{basicgeneral1}, \refeq{XiAsBounds}\\
  $\Psi^{\kappa}_z, \Pi^{\iota,\kappa}_z$ & coefficient of the NoBLE expansion 			    &  \refeq{basicgeneral2}, \refeq{XiAsBounds-2}\\
  $\aaz, \aabz$ & NoBLE parameters directly connected to $z$ 			        &  \refeq{basicgeneral1},\refeq{basicgeneral2} \\
  $\hat\Phi_z,\hat F_z$ & numerator and denominator of $\hat G_z$ 										            &  \refeq{DefPhi},\refeq{DefFFunction} \\
   $\apz, \cpz$ & SRW contributions in $\hat \Phi_z(k)$ 							        &  \refeq{DefPhi-simple},\refeq{detailed-Def-cpz}, \refeq{detailed-Def-apz}  \\
   $\afz, \cfz$ & SRW contributions in $\hat F_z(k)$ 								        &  \refeq{DefFFunction-simple}, \refeq{detailed-Def-afz}, \refeq{detailed-Def-cfz} \\
   $\hRpz(k),\hRfz(k)$ & remainder of the split of $\hat \Phi_z(k)$ and $\hat F_z(k)$ 		&  \refeq{detailed-Def-Rfz}, \refeq{DefFFunction-simple}\\[2mm]
  \hline
  & & \\[-2mm]
   $f_1,f_2,f_3$ & bootstrap functions                                                      &  \refeq{defFunc1}-\refeq{defFunc3}\\
   $\gamma_i,\Gamma_i$ & assumed/concluded bounds on the $f_i$                              &  \refeq{defFunc1}-\refeq{defFunc3}\\
   $\beta_\bullet$ & assumed bound on the coefficient                                       & Assumptions \ref{assDiagBounds},\ref{assDiagBoundsCoeff}\\
  \hline
  & & \\[-2mm]
   $\hat D^{\sin}(k)$ & $\sum_{s=1}^d (\partial_s \hat D(k))^2$                             & \refeq{Dsin-def} \\
   $I_{n,l},K_{n,l},T_{n,l},U_{n,l}$ & SRW two-point function and integrals                 & \refeq{Analysis-Def-Integral-I}-\refeq{Analysis-Def-Integral-U}\\
   $\Xi^{\ssc[N]}_{\alpha,z},\Psi^{\ssc[N],\iota}_{\alpha,{\sss II},z},\Pi^{\ssc[0],\iota,\kappa}_{{\sss R},z}$ & model-dependent split of the NoBLE coefficients    & Section \ref{secRewriteSplitSummary}\\
  \hline
\end{tabular}
\caption{List of notation, that is used in at least two different sections.}
\end{table}
\clearpage
\section{Proof of Lemmas \ref{lemmaFdifferenceToX}-\ref{lemmacosdecomp} }
\label{sec-proof-lemmas}

\begin{proof}[Proof of Lemma \ref{lemmaFdifferenceToX}]
Using a telescoping sum identity, we write
	\begin{eqnarray}
	1-\e^{\ii k\cdot x}=1-\e^{\ii k_1x_1}
	+\e^{\ii k_1x_1}(1-\e^{\ii k_2x_2})+\dots+\e^{\ii\sum_{j=1}^{d-1}k_jx_j}(1-\e^{\ii k_dx_d}).
	\end{eqnarray}
We reorder the sum over $x$ using the symmetry of $x\mapsto g(x)$, to obtain
	\begin{align}
	\sum_{x}g(x)[1-\cos(k\cdot x)]
	=&\sum_{x}g(x)\sum_{\mu=1}^d\cos\Big(\sum_{j=1}^{\mu-1}k_jx_j\Big)
	[1-\cos(k_\mu\cdot x_\mu)]\nnb
	\leq& \sum_{x}g(x)\sum_{\mu=1}^d [1-\cos(k_\mu\cdot x_\mu)].
	\end{align}
We use $1-\cos(nt)\leq n^2 [1-\cos(t)]$, see e.g.\ Lemma \ref{lemmacosdecomp} below, and the symmetry of $g$ to see that
	\begin{eqnarray}
	\sum_{x}g(x)[1-\cos(k\cdot x)]&\leq& \sum_{x}g(x)\sum_{\mu=1}^d x^2_\mu [1-\cos(k_\mu)]\nnb
	&=&[1-\hat D(k)] \sum_{x}g(x) \|x\|_2^2.
	\end{eqnarray}
	\end{proof}

\begin{proof}[Proof of Lemma \ref{lemmacosdecomp}]
	We obtain \refeq{lemmacosdecompFirstStep} by taking the real part of the telescoping sum identity
		\begin{eqnarray}
		1-\e^{\ii t}=  \sum_{i=1}^J [1-\e^{\ii t_i}] \prod_{j=1}^{i-1} \e^{\ii t_j}.
		\end{eqnarray}
In the following, we use that $|\sin(x+y)|\leq|\sin(x)|+|\sin(y)|$, $|ab|\leq (a^2+b^2)/2$ and $1-\cos^2(a)\leq 2[1-\cos(a)]$ to conclude from \refeq{lemmacosdecompFirstStep} that
		\begin{eqnarray}
		\nn
		1-\cos(t)&\leq& \sum_{i=1}^J [1-\cos(t_i)]+ \sum_{i=2}^J \sum_{j=1}^{i-1} |\sin(t_i)||\sin(t_j)|\\
		\nn
		&\leq& \sum_{i=1}^J [1-\cos(t_i)]+ \frac 1 2 \sum_{i=2}^J
		\sum_{j=1}^{i-1} 	\left[\sin^2(t_i)+\sin^2(t_j)\right]\\
		&=& \sum_{i=1}^J [1-\cos(t_i)]+ \frac {J-1} 2 \sum_{i=1}^J  \sin^2(t_i)\leq J\sum_{i=1}^J [1-\cos(t_i)].
		\end{eqnarray}
	\vskip-0.4cm
	\end{proof}
	
\section{Decomposition of $\bigtriangleup\hat G_z(k)$ for the bootstrap function $f_3$}
\label{sec-decomposition-f3}

We improve the bound on the bootstrap function $f_3$ in Section \ref{subsecImproF3bar} using the decomposition of
$\bigtriangleup\hat G_z(k)$ into the five pieces $\hat{H}_1(k)-\hat{H}_5(k)$ defined in \refeq{def-H1}-\refeq{def-H5}, as stated in
\refeq{sec-decomposition-f3-statement}. We create this decomposition by expanding $\bigtriangleup\hat G_z(k)$ in
the form of the simplified rewrite, defined in \refeq{DefPhi-simple}-\refeq{generalForm-simple}, and group the terms
with common factors.  This decomposition is a key step in our analysis.
The decomposition is long and tedious and, due to the many terms, prone to errors.
For this reason, we explicitly perform it here in the appendix.
\paragraph{Expanding the terms.}
Our strategy is to first expand
	\begin{align}
	\bigtriangleup\hat G_z(k)=&\frac {\bigtriangleup\hat \Phi_z(k)}{1-\hat F_z(k)}+
	\hat \Phi_z(k) \Big(\bigtriangleup\frac {1}{1-\hat F_z(k)} \Big)
	+2\sum_{s=1}^d\partial_{s}\hat \Phi_z(k)
	\partial_{s} \Big(\frac {1}{1-\hat F_z(k)} \Big),
	\lbeq{XspaceDecomp-tmp1}
	\end{align}
and then group them in terms of their common factors.
To not loose track of terms, we label them using the line number in the formula in which they appear, e.g., (\ref{e:XspaceDecomp-tmp2}.a) and (\ref{e:XspaceDecomp-tmp4}.d).\\

Before stating this, we remark how we extract SRW-contributions $\hat C^*(k)$ from inverse powers of $(1-\hat F_z(k))$ using the notation introduced in \refeq{Analysis-Def-C-eps}-\refeq{FExpandDeltaFNotation}:
	\begin{align}
	\lbeq{FExpandForDeltaExpand}
	\frac 1 {1-\hat F_z(k)}&= \frac 1 {1-\hat F_z(0)+\afz[1-\hat D(k)] +\deltaRfz}
	=\hat C^*(k) -\hat E (k),\\
	\lbeq{FExpandForDeltaExpand-2}
	\frac 1 {(1-\hat F_z(k))^2}&=\hat C^*(k)^2 -\hat E (k)\hat C^*(k)-\frac{\hat E (k)}{1-\hat F_z(k)},\\
	\lbeq{FExpandForDeltaExpand-3}
	\frac 1 {(1-\hat F_z(k))^3}&=\hat C^*(k)^3 -\hat E (k)\hat C^*(k)^2-\frac{\hat E (k)\hat C^*(k)}
	{1-\hat F_z(k)}-\frac{\hat E (k)}{(1-\hat F_z(k))^2}.
	\end{align}
Further, we recall the form of $\hat \Phi_z(k)$ and $\hat F_z(k)$, see \refeq{DefPhi-simple}-\refeq{DefFFunction-simple}, and note that
	\begin{align}
	\bigtriangleup\hat \Phi_z(k)=&-\apz\hat D(k)+\bigtriangleup\hRpz(k), \quad
	\bigtriangleup\hat F_z(k)=-\afz\hat D(k)+\bigtriangleup\hRfz(k).
      	\lbeq{deltaFcomp}
	\end{align}
Then, we proceed by expanding the terms in \refeq{XspaceDecomp-tmp1}.
\paragraph{First term of \refeq{XspaceDecomp-tmp1}.}
Using \refeq{FExpandForDeltaExpand} we compute
	\begin{align}
	\lbeq{XspaceDecomp-tmp2}
	\frac{\bigtriangleup\hat \Phi_z(k)}{1-\hat F_z(k)}&=
	\underbrace{-\apz\hat D(k)\hat C^*(k)}_{\text{:=(\ref{e:XspaceDecomp-tmp2}.a)}}
	+\underbrace{\apz\hat D(k)\hat E(k)}_{\text{:=(\ref{e:XspaceDecomp-tmp2}.b)}}
	+\underbrace{\frac{\bigtriangleup\hRpz(k)}{1-\hat F_z(k)}}_{\text{:=(\ref{e:XspaceDecomp-tmp2}.c)}}.
	\end{align}
\paragraph{Second term of \refeq{XspaceDecomp-tmp1}.}
We first note that
	\begin{align*}
 	\bigtriangleup\frac {1}{1-\hat F_z(k)}=&\frac {\bigtriangleup \hat F_z(k)}{(1-\hat F_z(k))^2}
	+2\frac {\sum_{s=1}^d (\partial_{s}\hat F_z(k))^2}{(1-\hat F_z(k))^3}.
	\end{align*}
We use \refeq{FExpandForDeltaExpand-2}, \refeq{deltaFcomp} and $\hat\Phi_z(k)=\cpz+\apz\hat D(k)+\hRpz(k)$ for the term containing $\bigtriangleup \hat F_z(k)$ and obtain
	\begin{align}
	\hat\Phi_z(k)\frac {\bigtriangleup \hat F_z(k)}{(1-\hat F_z(k))^2}=&\hat\Phi_z(k)\frac {-\afz\hat D(k)}{(1-\hat F_z(k))^2}+\hat\Phi_z(k)\frac {\bigtriangleup\hRpz(k)}{(1-\hat F_z(k))^2}
\lbeq{XspaceDecomp-tmp3}\\
	=&\underbrace{-\afz\hat D(k)(\cpz+\apz\hat D(k))\hat C^*(k)^2}_{\text{:=(\ref{e:XspaceDecomp-tmp3}.a)}}\nnb
	&+\underbrace{ \afz\hat D(k)(\cpz+\apz\hat D(k))\left(\hat E(k) \hat C^*(k)+\frac {\hat E(k)}{ (1-\hat F_z(k))}\right)}_{\text{:=(\ref{e:XspaceDecomp-tmp3}.b)}}\nnb
	&\underbrace{- \frac{\hRpz(k)\afz\hat D(k)}{ (1-\hat F_z(k))^2}}_{\text{:=(\ref{e:XspaceDecomp-tmp3}.c)}}
	+ \underbrace{\frac{\bigtriangleup\hRfz(k)}{ (1-\hat F_z(k))}\hat G_z(k)}_{\text{:=(\ref{e:XspaceDecomp-tmp3}.d)}}.\nn
	\end{align}
For the term including the mixed derivatives we compute that
	\begin{align}
	2\hat\Phi_z(k) \frac {\sum_{s=1}^d (\partial_{s}\hat F_z(k))^2}{(1-\hat F_z(k))^3}=&2\hat\Phi_z(k) \frac {\sum_{s=1}^d (\afz\partial_{s}\hat D(k)+\partial_{s}\hRfz(k))^2}{(1-\hat F_z(k))^3}
	\lbeq{XspaceDecomp-tmp4}\\
   =&2\afz^2(\cpz+\apz\hat D(k)+\hRpz(k)) \frac {\sum_{s=1}^d \partial_{s}\hat D(k)^2}{(1-\hat F_z(k))^3}\nnb
    &+2\hat\Phi_z(k)\frac {\sum_{s=1}^d \partial_{s}\hRfz(k)^2+2\afz\partial_{s}\hat D(k)\partial_{s}\hRfz(k)}{(1-\hat F_z(k))^3}\nnb
    =&\underbrace{2\afz(\cpz+\apz\hat D(k))\hat D^{\sin}(k)\hat C^*(k)^3}_{\text{:=(\ref{e:XspaceDecomp-tmp4}.a)}} \nnb
    &\underbrace{-2\afz(\cpz+\apz\hat D(k))\hat D^{\sin}(k) \left(\hat E (k)\hat C^*(k)^2+\frac{\hat E (k)\hat C^*(k)}{1-\hat F_z(k)}\right)}_{\text{:=(\ref{e:XspaceDecomp-tmp4}.b)}}\nnb
    &+\underbrace{2\afz \hRpz(k) \frac {\hat D^{\sin}(k)}{(1-\hat F_z(k))^2} \hat C^*(k)}_{\text{:=(\ref{e:XspaceDecomp-tmp4}.c)}} \underbrace{- 2\afz \hat \Phi_z(k) \frac {\hat D^{\sin}(k)}{(1-\hat F_z(k))^2}\hat E (k)}_{\text{:=(\ref{e:XspaceDecomp-tmp4}.d)}}\nnb
    &+\underbrace{2\afz(\afz-1)\hat \Phi_z(k) \frac {\hat D^{\sin}(k)}{(1-\hat F_z(k))^3}}_{\text{:=(\ref{e:XspaceDecomp-tmp4}.e)}}\nnb
    &+\underbrace{2\hat\Phi_z(k)\frac {\sum_{s=1}^d \partial_{s}\hRfz(k)^2+2\afz\partial_{s}\hat D(k)\partial_{s}\hRfz(k)}{(1-\hat F_z(k))^3}}_{\text{:=(\ref{e:XspaceDecomp-tmp4}.f)}}\nn
	\end{align}


\paragraph{Third term of \refeq{XspaceDecomp-tmp1}.}
The expand the last term of \refeq{XspaceDecomp-tmp1} as follows
	\begin{align}
	2\sum_{s=1}^d\partial_{s}\left(\hat \Phi_z(k)\right)
	& \partial_{s} \left(\frac {1}{1-\hat F_z(k)} \right)=2\sum_{s=1}^d (\apz\partial_{s}\hat D(k)+\partial_{s} \hRpz(k))\frac {(\afz\partial_{s}\hat D(k)+\partial_{s} \hRfz(k))} {(1-\hat F_z(k))^2}
	\nnb
	=& \underbrace{2\apz \hat D^{\sin}(k) \hat C^*(k)^2}_{\text{:=(\ref{e:XspaceDecomp-tmp5}.a)}}
    \underbrace{-2\apz \hat D^{\sin}(k) \hat E(k)\hat C^*(k)}_{\text{:=(\ref{e:XspaceDecomp-tmp5}.b)}}
    \underbrace{-2\apz \hat D^{\sin}(k) \frac {\hat E(k)}{ 1-\hat F_z(k)}}_{\text{:=(\ref{e:XspaceDecomp-tmp5}.c)}}\nnb
    &+\underbrace{2(\afz-1)\apz \hat D^{\sin}(k) \frac 1 {(1-\hat F_z(k))^2}}_{\text{:=(\ref{e:XspaceDecomp-tmp5}.d)}}\nnb
	&+\underbrace{2\frac {1}{(1-\hat F_z(k))^2} \sum_{s=1}^d\left( \partial_{s} \hRpz(k)\afz \partial_{s} \hat D(k)+
	\partial_{s}\hat \Phi_z(k) \partial_{s} \hRfz(k)\right)}_{\text{:=(\ref{e:XspaceDecomp-tmp5}.e)}}.
	\lbeq{XspaceDecomp-tmp5}
	\end{align}
\paragraph{Regrouping the terms.}
In this section we group the terms (\ref{e:XspaceDecomp-tmp2}.a)-(\ref{e:XspaceDecomp-tmp5}.e) to derive $\hat H_1(k)-\hat H_5(k)$.\\
Let us start by recalling that $\hat M^*(k)=\hat D(k)-2 \hat D^{\sin}(k)\hat C^*(k)$ and see that
	\begin{align}
	\text{(\ref{e:XspaceDecomp-tmp2}.a)}+\text{(\ref{e:XspaceDecomp-tmp5}.a)}
	&=-\apz\hat C^*(k)\hat M^*(k),\nnb
	\text{(\ref{e:XspaceDecomp-tmp3}.a)}+\text{(\ref{e:XspaceDecomp-tmp4}.a)}
	&=-\afz(\cpz+\apz\hat D(k))\hat C^*(k)^2\hat M^*(k).\nn
    \end{align}
Reviewing the definition of $\hat H_1(k)$ we see that
    \begin{align}
    \text{(\ref{e:XspaceDecomp-tmp2}.a)}+
    \text{(\ref{e:XspaceDecomp-tmp3}.a)}+
    \text{(\ref{e:XspaceDecomp-tmp4}.a)}+
    \text{(\ref{e:XspaceDecomp-tmp5}.a)}
    =-\hat H_1(k).
    \end{align}
Then, we proceed with the term $\hat H_2(k)$:
	\begin{align}
	\text{(\ref{e:XspaceDecomp-tmp2}.b)}
	+\text{(\ref{e:XspaceDecomp-tmp5}.b)}
	&=\apz\hat E(k)\hat M^*(k),\nnb
    	\text{(\ref{e:XspaceDecomp-tmp3}.b)}
	+\text{(\ref{e:XspaceDecomp-tmp4}.b)}
	&=\afz(\cpz+\apz\hat D(k))\hat M^*(k)\left(\hat E(k) \hat C^*(k)+\frac {\hat E(k)}{ (1-\hat F_z(k))}\right),\nnb
    	\text{(\ref{e:XspaceDecomp-tmp3}.c)}+
	\text{(\ref{e:XspaceDecomp-tmp4}.c)}
	&=-\frac{\afz \hRpz(k)}{ (1-\hat F_z(k))^2}\hat M^*(k),\nnb
    	\text{(\ref{e:XspaceDecomp-tmp2}.b)}+
	\text{(\ref{e:XspaceDecomp-tmp5}.b)}&+
	\text{(\ref{e:XspaceDecomp-tmp3}.b)}+
	\text{(\ref{e:XspaceDecomp-tmp4}.b)}+
	\text{(\ref{e:XspaceDecomp-tmp3}.c)}+
	\text{(\ref{e:XspaceDecomp-tmp4}.c)}
	=-\hat H_2(k).
    	\end{align}
$\hat H_3(k)$ captures the following terms:
	\begin{align}
	\text{(\ref{e:XspaceDecomp-tmp4}.d)}
	+\text{(\ref{e:XspaceDecomp-tmp4}.e)}
	=&2\afz\hat G_z(k)\frac {\hat D^{\sin}(k)}{1-\hat F_z(k)}\left(\frac {\afz-1}{(1-\hat F_z(k))}-  \hat E (k)\right),
	\nnb
	\text{(\ref{e:XspaceDecomp-tmp5}.c)}+
	\text{(\ref{e:XspaceDecomp-tmp5}.e)}
	=&2\apz \frac {\hat D^{\sin}(k)} {1-F_z(k)}  \left(\frac {\afz-1} {1-F_z(k)}- \hat E(k)\right),\nnb
 	\text{(\ref{e:XspaceDecomp-tmp4}.d)}+
	\text{(\ref{e:XspaceDecomp-tmp4}.e)}+&
	\text{(\ref{e:XspaceDecomp-tmp5}.c)}+
	\text{(\ref{e:XspaceDecomp-tmp5}.d)}=-\hat H_3(k).
	\end{align}
In $\hat H_4(k)$, we collect the double derivatives of the remainder terms.
	\begin{align}
	\text{(\ref{e:XspaceDecomp-tmp2}.c)}+
	\text{(\ref{e:XspaceDecomp-tmp3}.d)}
	=\frac{\bigtriangleup\hRpz(k)}{1-\hat F_z(k)}+\frac{\bigtriangleup\hRfz(k)}{ (1-\hat F_z(k))}\hat G_z(k)
	=-\hat H_4(k).
	\end{align}
The two remaining terms consist of mixed derivatives and are put into $\hat H_5(k)$:
	\begin{align}
	\text{(\ref{e:XspaceDecomp-tmp4}.f)}+
	\text{(\ref{e:XspaceDecomp-tmp5}.e)}&=-\hat H_5(k).
	\end{align}
This completes the derivation of the split into $\hat H_1(k)-\hat H_5(k)$.

\section{Bound on the rewrite}
\label{secRewriteBounds}
Here we prove Proposition \ref{assTranslation}(ii). Namely, we prove that the bounds stated in Assumption \ref{assDiagBounds} are implied by Assumptions  \ref{assSymtwo}-\ref{assDiagBoundsCoeff}, by computing expressions for the bounds required in Assumption \ref{assDiagBounds} one after the other.

For better readability, we divide this into five steps: Step 1 contains the six relatively simple bounds stated
in \refeq{analys-assumed-rho-Bound}-\refeq{analys-assumed-Bound-PiPsi}, while Steps 2-5 show
the more elaborate bounds stated in \refeq{analys-assumed-Remainder}-\refeq{analys-assumed-displacement-3}. The bounds involved are not difficult, but rather somewhat tedious.
Throughout this section, we omit the subscripts $z$ and write, e.g., $\aa=\aaz$ and $\aab=\aabz$.

\paragraph{Step 1: bounds stated in \refeq{analys-assumed-rho-Bound}-\refeq{analys-assumed-Bound-PiPsi}.}

\begin{enumerate}[a)]
\item The bound $\betaaa$ in \refeq{analys-assumed-rho-Bound} is also assumed in \refeq{analys-assumed-rho-Bound-two} of Assumption \ref{assDiagBoundsCoeff}.

\item We recall the definition of $\cpz$ in \refeq{detailed-Def-cpz}, use the bounds of Assumption \ref{assDiagBoundsCoeff} and the fact that the coefficients $\Xi_z^\ssc[N],\Xi_z^{\ssc[N],\iota}$ are non-negative to conclude
	\begin{align}
	1-\beta_{{\sss\Xi_\alpha(0)}}^\ssc[1-0]-
	\frac {2d \aa} 	{1-\aa^2}\beta_{{\sss\Xi^{\iota}_\alpha,I}}^\ssc[0]\leq \cpz
	\leq 	1+\beta_{{\sss\Xi_\alpha(0)}}^\ssc[0-1]
	+\frac {2d \aa^2} {1-\aa^2}\beta_{{\sss\Xi^{\iota}_\alpha,II}}^\ssc[0].
	\end{align}
In this and the following lines we bound $\aa$ by $\tfrac {\Gamma_1c_\mu}{2d-1}$.
\item For $\afz$, defined in \refeq {detailed-Def-afz}, we use the bounds of Assumption \ref{assDiagBoundsCoeff} to obtain
	\begin{align}
	\frac {2d \aa} {1-\aa^2}
	&\left[1-\beta_{{\sss\sum \Psi^{\iota}_\alpha,I}}^\ssc[1-0]
	-\aa\beta_{{\sss\sum \Psi^{\iota}_\alpha,II}}^\ssc[0-1]
	-\frac {1} {1-\aa^2}  \bar{\beta}_{\sum \sss \Pi_{\alpha}}^\ssc[0]\right]\nnb
	&\leq \afz\leq\frac {2d \aaz} {1-\aaz^2}
	\left[1+\beta_{{\sss\sum \Psi^{\iota}_\alpha,I}}^\ssc[0-1]
	+\aa\beta_{{\sss\sum \Psi^{\iota}_\alpha,II}}^\ssc[1-0]
	-\frac {1} {1-\aa^2} \underline{\beta}_{ \sss  \sum \Pi_{\alpha}}^\ssc[0]\right].
	\end{align}
We bound $\aa$ on the left-hand side of the inequality from below by $\betaaalow$.
\item The absolute value of $\apz$ defined in \refeq{detailed-Def-apz} is bounded by
	\begin{align}
	|\apz|=\max\Big \{
	&2d\beta_{{\sss\Xi_\alpha(\ve[1])}}^\ssc[1-0]
	+ \frac {2d \aa} {1-\aa^2} \beta_{{\sss\sum\Xi^{\iota}_\alpha,I}}^\ssc[0],
	2d\beta_{{\sss\Xi_\alpha(\ve[1])}}^\ssc[0-1]+ \frac {2d \aa^2} {1-\aa^2}
	\beta_{{\sss\sum\Xi^{\iota}_\alpha,II}}^\ssc[0]\Big\}.
	\end{align}
\item To bound $\sum_{x,\kappa}\Pi_z^{\kappa,\iota}(x)$ we use that the coefficients are defined by an alternating series, as well as Assumption \ref{assCoefficentsRelation}-\ref{assDiagBoundsCoeff}, to obtain
	\begin{align}
	\sum_{x,\kappa}\Pi_z^{\kappa,\iota}(x)\leq&
	\sum_{N=0}^\infty \sum_{x,\kappa}
	\Pi_z^{\ssc[2N],\kappa,\iota}(x) -\Pi_z^{\ssc[1],\kappa,\iota}(x)
	\leq \aab\sum_{x,N,\kappa}\Xi_z^{\ssc[2N],\iota}(x)
	-\sum_{\kappa}\hat \Pi_z^{\ssc[1],\kappa,\iota}(0)
	\nnb
	\leq& 2d \aab\Big(\sum_{N=0}^\infty \beta^\ssc[2N]_{\sss \Xi^\iota}\Big)
	-\underline{\beta}_{\sss \sum \Pi}^\ssc[1].
	\end{align}

\item In the same way, we obtain that
	\begin{align}
	\sum_{x}\Psi_z^{\kappa}(x)\geq  -\betaaa
	\Big(\sum_{N=0}^\infty \beta^\ssc[2N+1]_{\sss \Xi}\Big)+\underline{\beta}_{\sss \Psi}^\ssc[0].
	\end{align}
\end{enumerate}

\paragraph{Step 2: bounds stated in \refeq{analys-assumed-Remainder}.}
In the steps 2-5 we create bounds for the remainder terms, as defined in Section \ref{secDefinitionRewrite}.
For the bounding it is useful to note the explicit formulae for the remainder $\hRfz(k)$ and $\hRpz(k)$:
\begin{align}
	\lbeq{detailed-Def-Rfz}
	\hRfz(k)=& \sum_{n=2}^\infty \hat F_n(k)
	+\frac {\aa} {1-\aa^2}  \sum_{N=2}^\infty \sum_{\iota} (-1)^N\hat \Psi^{\ssc[N],\iota}_z(k) (\e^{-\ii k_{\iota}} -\aa)\\
	&+\frac {\aa} {1-\aa^2}  \sum_{N\in\{0,1\}}\sum_{\iota}	 (-1)^N 	 (\e^{-\ii k_{\iota}}\hat \Psi^{\ssc[N],\iota}_{{\sss R, I},z} (k) -\aa \hat \Psi^{\ssc[N],\iota}_{{\sss R, II},z} (k) ) \nnb
	&-\frac {\aa} {(1-\aa^2)^2} \sum_{\iota_0,\iota_1} \e^{-\ii (k_{\iota_{0}}+k_{\iota_{1}})}
	\Big(\hat \Pi^{\ssc[0],\iota_{0},\iota_{1}}_{{\sss R},z}(k)
	+\sum_{N=1}^\infty  (-1)^N \hat \Pi^{\ssc[N],\iota_{0},\iota_{1}}_z(k)\Big)\nnb
	&+\frac {\aa^2} {(1-\aa^2)^2} \sum_{\iota_0,\iota_1} \left(\e^{-\ii k_{\iota_1}}
	\hat \Pi^{-\iota_{0},\iota_{1}}(k)+\e^{-\ii k_{\iota_{0}}}\hat \Pi^{\iota_{0},\iota_{1}}(k)
	-\aa\hat \Pi^{-\iota_{0},\iota_{1}}(k)\right)\nnb
	&-\frac {\aa} {(1-\aa^2)^2} \sum_{\iota_0,\iota_1} \hat \Psi^{\iota_0}(k)
	\left(\e^{-\ii k_{\iota_{0}}}\hat \Pi^{\iota_{0},\iota_{1}}(k) -\aa\hat \Pi^{-\iota_{0},\iota_{1}}(k))\right)
	(\e^{-\ii k_{\iota_1}} -\aa).\nn
\end{align}
and
\begin{align}
	\lbeq{detailed-Def-Rpz}
	\hRpz(k)=&\sum_{N\in\{0,1\}}(-1)^N\hat \Xi^{\ssc[N]}_{{\sss R}}(k)
	+\sum_{N=2}^\infty (-1)^N \hat \Xi^{\ssc[N]}(k)+\sum_{n=1}^\infty \hat \Phi_n(k)\nnb
	&-\frac {\aaz} {1-\aaz^2}\sum_{\iota} \hat \Psi^{\iota}(k) (\hat \Xi^{\iota}(k)\e^{-\ii k_\iota}-\aaz\hat \Xi^{-\iota}(k))\nnb
	&- \frac {\aaz} {1-\aaz^2}\sum_{N=1}^\infty (-1)^{N}
	\sum_{\iota} (\hat \Xi^{\ssc[N],\iota}(k)\e^{-\ii k_\iota}-\aaz\hat \Xi^{\ssc[N],-\iota}(k))\nnb
	&- \frac {\aaz} {1-\aaz^2}  \sum_{\iota} (\hat \Xi^{\ssc[0],\iota}_{{\sss R, I}}(k)\e^{-\ii k_\iota}
	-\aaz\hat \Xi^{\ssc[0],-\iota}_{{\sss R, II}}(k)).
\end{align}
Now, we start by bounding $\hat F_n$ and $\hat \Phi_n$, for $n\geq 1$.
Starting from \refeq{deffnfunction}, we take the absolute value into the sums and use Assumption \ref{assCoefficentsRelation} for all $\Psi_z^\iota$ and $\Pi^{\iota,\kappa}_z$ to obtain
	\begin{align}
	\sum_{x\in\Zd} |F_{n}(x)|\leq&\aa \sum_x\sum_{\iota_0,\dots,\iota_n}\sum_{x_i\colon \sum_i x_i=x}
	\frac { (\delta_{x_0,0}+\tfrac {\aab}{\aa}|\Xi(x_0)|)}{1-\aa^2}\left(\frac { \aab} {1-\aa^2}\right)^{n} \nnb
	&\times\Big(\prod_{s=1}^{n-1} (|\Xi^{\iota_{s-1}}(x_s+\ve[\iota_{s-1}])|
	+\aa|\Xi^{-\iota_{s-1}}(x_s)|)\Big)\nnb
	&\times \big( |\Xi^{\iota_{n-1}}(x_n+\ve[\iota_{n-1}]+\ve[\iota_{n}])|
	+\aa|\Xi^{\iota_{n-1}}(x_n+\ve[\iota_{n-1}])| \nnb
	&\qquad\qquad\qquad\qquad+\aa|\Xi^{-\iota_{n-1}}(x_n+\ve[\iota_n])|
	+\aa^2|\Xi^{-\iota_{n-1}}(x_n)|\big).
	 \end{align}
As we sum over all $x$, the sums factorize and we obtain
	\begin{align}
	\sum_{x\in\Zd} |F_{n}(x)|\leq &\left( \frac {\aa} {1-\aa}\right)
	\Big(1+\tfrac {\aab}{\aa}\sum_{x}|\Xi(x)|\Big)
	\Big( \frac {\aab} {1-\aa}\Big)^{n}
	\Big(\prod_{s=1}^{n} \sum_{x_s,\iota_{s-1}}|\Xi^{\iota_{s-1}}(x_s)| \Big)
	\sum_{\iota_n}1\nnb
	\leq& \frac{2d \aa}{1-\aa} (1+\tfrac {\aab}{\aa} \beta_{{\sss \Xi}}^\ssss[abs])
	\left( \frac { 2d\aa\beta_{\sss \Xi^{\iota}}^\ssss[abs]} {1-\aa}\right)^{n}.
 	\end{align}
In the same way, we obtain
	\begin{align}
	\sum_{x\in\Zd} |\Phi_{n}(x)|\leq& \frac{2d\aa}{1-\aa} (1+\tfrac {\aab}{\aa}\beta_{{\sss \Xi}}^\ssss[abs])
	\left( \frac {2d\aab\beta_{\sss \Xi^{\iota}}^\ssss[abs]}
	{1-\aa}\right)^{n}\beta_{\sss \Xi^{\iota}}^\ssss[abs].
	\end{align}
Using the geometric sum, we compute that
	\begin{align}
	\sum_{n=2}^\infty\sum_{x\in\Zd} |F_{n}(x)|\leq &\frac{2d\aa}{1-\aa}
	\frac {1+\tfrac {\aab}{\aa}\beta_{{\sss \Xi}}^\ssss[abs]}
	{1- \frac {2d\aab\beta_{\sss \Xi^{\iota}}^\ssss[abs]} {1-\aa}}
	\Big( \frac { 2d\aab\beta_{\sss \Xi^{\iota}}^\ssss[abs]} {1-\aa}\Big)^{2},\\
	\sum_{n=1}^\infty\sum_{x\in\Zd} |\Phi_{n}(x)|
	\leq &\frac{2d\aa \beta_{\sss \Xi^{\iota}}^\ssss[abs]}{1-\aa}
	\frac {1+\tfrac {\aab}{\aa}\beta_{{\sss \Xi}}^\ssss[abs]}
	{1- \frac { 2d\aab\beta_{\sss \Xi^{\iota}}^\ssss[abs]} {1-\aa}}
 	\left(\frac { 2d\aab\beta_{\sss \Xi^{\iota}}^\ssss[abs]} {1-\aa}\right).
	\end{align}
Using these bounds, we obtain that $\Rfz$ (defined in \refeq{detailed-Def-Rfz}) is bounded by
	\begin{align}
	\sum_{x\in\Zd}|\Rfz (x)|\leq&\frac{2d\aa}{1-\aa} \frac {1+\tfrac {\aab}{\aa}\beta_{{\sss \Xi}}^\ssss[abs]}
	{1- \frac {2d\aab\beta_{\sss \Xi^{\iota}}^\ssss[abs]} {1-\aa}}
 	\Big( \frac { 2d\aab\beta_{\sss \Xi^{\iota}}^\ssss[abs]} {1-\aa}\Big)^{2}\nnb
	&+ \frac {2d\aa} {1-\aa^2} \Big(\Big[\sum_{N\in\{0,1\}} \beta_{\sss\Psi,R,I}^\ssc[N] 	+\aa\beta_{\sss\Psi,R,II}^\ssc[N]\Big] + \tfrac {\aab}{\aa}(1+\aa)\sum_{N=2}^\infty \beta_{\sss  \Xi}^\ssc[N]
	\Big)\nnb
	&+\frac {2d\aa} {(1-\aa^2)^2} \Big( \beta_{\sss\Pi,R}^\ssc[0]+
	2d\aab \sum_{N=1}^\infty \beta^{\ssc[N]}_{\sss \Xi^\iota}\Big)\nnb
	&+\frac {(2d\aa)^2\aab} {(1-\aa^2)^2}(2+\mu)\beta^{\ssss[abs]}_{\sss \Xi^\iota}
	+\frac {(2d)^2\aab^2} {\left(1-\aa\right)^2} \beta^{\ssss[abs]}_{\sss \Xi}
	\beta^{\ssss[abs]}_{\sss \Xi^\iota},
	\end{align}
and $\Rpz$ (given in \refeq{detailed-Def-Rpz}) by
	\begin{align}
	\sum_{x\in\Zd} |\Rpz(x)|	
	\leq&\beta_{\sss\Xi,R}^\ssc[0]+\beta_{\sss\Xi,R}^\ssc[1]
	+\sum_{N=2}^\infty\beta^{\ssc[N]}_{\sss \Xi}
	+\frac{2d\aa \beta_{\sss \Xi^{\iota}}^\ssss[abs]}{1-\aa}
	\frac {1+\tfrac {\aab}{\aa}\beta_{{\sss \Xi}}^\ssss[abs]}
	{1- \frac { 2d\aab\beta_{\sss \Xi^{\iota}}^\ssss[abs]} {1-\aa}}
 	\left(\frac { 2d\aab\beta_{\sss \Xi^{\iota}}^\ssss[abs]} {1-\aa}\right)\nnb
	&+ \frac {2d\aab} {1-\aa}\beta_{{\sss \Xi}}^\ssss[abs] \beta_{\sss \Xi^{\iota}}^\ssss[abs]
	+ \frac {2d\aa} 	
	{1-\aa^2}\Big(\beta_{\sss\Xi^\iota,R,I}^\ssc[0]+\aa\beta_{\sss\Xi^\iota,R,II}^\ssc[0]
	+(1+\aa)\sum_{N=1}^\infty\beta_{\sss \Xi^{\iota}}^\ssc[N]\Big).
	\end{align}
This proves the bounds stated in \refeq{analys-assumed-Remainder}.

\paragraph{Step 3: The first bound in \refeq{analys-assumed-displacement-1}.}
We call a sum as appearing on the left hand side of \refeq{analys-assumed-displacement-1} a \emph{weighted} sum, where the factor $\|x\|_2^2$ is called the \emph{weight} (see also Section \ref{sec-disc}). To  obtain a bound on weighted sums, we split the weight using   \refeq{Split-weight-ineq}.
We explain this now in detail for the following contribution to
$\sum_x\|x\|_2^2|\Rpz(x)|$:
	\begin{align}
	\lbeq{example-term-to-bound-rpz}
	\frac {\aa} {1-\aa^2}\sum_{x,y,\iota} \|x+y\|_2^2\left|\Psi^{\iota}(x)
	(\Xi^{\iota}(y-\ve[\iota])-\aa \Xi^{-\iota}(y))\right|.
	\end{align}
We use Assumption \ref{assCoefficentsRelation} to bound $\Psi^{\iota}_z$ by $\Xi_z$, and obtain
	\begin{align}
	\text{\refeq{example-term-to-bound-rpz}}
	\leq \frac {\aab} {1-\aa^2}\sum_{x,y,\iota}
	\|x+y\|_2^2|\Xi(x)| (|\Xi^{\iota}(y-\ve[\iota])|+\aa |\Xi^{-\iota}(y)|).
	\end{align}
Then, we rewrite $\|x\|_2^2$ using the equality in \refeq{Split-weight-ineq} as
	\begin{align}
	\text{\refeq{example-term-to-bound-rpz}} \leq& \frac {\aab} {1-\aa^2}
	\sum_{x} \|x\|_2^2|\Xi(x)| \sum_{y,\iota} (1+\aa)|\Xi^{\iota}(y)|\nnb
	&   + \frac {\aab} {1-\aa^2}\sum_{x} |\Xi(x)| \sum_{y,\iota} \|y\|_2^2(|\Xi^{\iota}(y-\ve[\iota])|
	+\aa |\Xi^{-\iota}(y)|)\nnb
	&   + \frac {2\aab} {1-\aa^2}\sum_{i=1}^d\sum_{x}x_i|\Xi(x)| \sum_{y,\iota}y_i(|\Xi^{\iota}(y-\ve[\iota])|
	+\aa |\Xi^{-\iota}(y)|).
	\end{align}
Since $\Xi$ is totally rotationally symmetric, we know that $\sum_{x}x_i|\Xi(x)|=0$.
Thus, the last term cancels. We apply the bounds of Assumption  \ref{assDiagBoundsCoeff} to obtain
	\begin{align}
	\refeq{example-term-to-bound-rpz}
	\leq \frac {\aab} {1-\aa^2} \left((1+\aa)\beta_{{\sss \Delta \Xi}}^\ssss[abs]
	\beta_{\sss \Xi^{\iota}}^\ssss[abs]
	+\beta_{{\sss \Xi}}^\ssss[abs] (\beta_{\sss \Delta \Xi^{\iota},\iota}^\ssss[abs]
	+\aa\beta_{{\sss \Delta \Xi^{\iota},0}}^\ssss[abs])\right).
	\end{align}
Now we extend this idea. First we use Assumption \ref{assCoefficentsRelation} to obtain
	\begin{align}
	\sum_x \|x\|_2^2|\Phi_{n}(x)|\leq& \frac {\aa}{\aab}
	\left(\frac {\aab}{1-\aa^2} \right)^{n+1}
	\sum_{x_0,\dots,x_{n+1}} \|\sum_{i=0}^Nx_i\|_2^2 (\delta_{0,x_0}+\tfrac {\aab}{\aa}|\Xi(x_0)|) \nnb
	&\times \prod_{s=1}^{n+1} \Big[\sum_{\iota_{s-1}}\left(|\Xi^{\iota_{s-1}}(x_s+\ve[\iota_{s-1}])|
	+ \aa|\Xi^{-\iota_{s-1}}(x_s)\right)\Big].
	\end{align}
Then, we split the weight using the equality in \refeq{Split-weight-ineq}. We note that  $x\mapsto \sum_{\iota} \Xi^{\iota}(x)$ is totally rotationally symmetric so that any contribution due to the second sum in the equality in  \refeq{Split-weight-ineq} cancels. After the split of the weight, the sums factor and we obtain
	\begin{align}
	\sum_x \|x\|_2^2|\Phi_{n}(x)|\leq&
	\left(\frac {\aab(1+\aa)}{1-\aa^2} \right)^{n+1}
	\sum_{x} \|x\|_2^2 |\Xi(x)|  \Big(\sum_{\iota,x}|\Xi^{\iota}(x)|\Big)^{n+1}\\
	&+(n+1)
	\left(\frac {\aab(1+\aa)}{1-\aa^2} \right)^{n+1}
	\sum_{x} (\frac {\aa}{\aab}\delta_{0,x}+ |\Xi(x)|)
	\Big(\sum_{\iota,x}|\Xi^{\iota}(x)|\Big)^{n}\nnb
	&\quad\qquad\times \Big[\sum_{\iota,x} \|x\|_2^2 \left(|\Xi^{\iota}(x+\ve[\iota])|
	+ \aa|\Xi^{-\iota}(x)\right)\Big]\nn\\
	\leq &\left(\frac {2d\aab}{1-\aa} \right)^{n+1}
	(\beta_{\sss \Xi^{\iota}}^\ssss[abs])^{n}
	\Big( \beta_{{\sss \Delta \Xi}}^\ssss[abs] \beta_{\sss \Xi^{\iota}}^\ssss[abs]
	+\frac{n+1}{1+\aa}(\tfrac {\aa}{\aab}+\beta_{{\sss \Xi}}^\ssss[abs])
	\left[\beta_{\sss \Delta \Xi^{\iota},\iota}^\ssss[abs]+\aa\beta_{{\sss \Delta
	\Xi^{\iota},0}}^\ssss[abs]\right]\Big).\nn
	\end{align}
Other contributions to $\Rpz$ (recall \refeq{detailed-Def-Rpz}) are bounded in a similar straightforward manner, leading to the final bound
	\begin{align}
	\lbeq{detailed-Bound-Rpz}
	\sum_{x}\|x\|_2^2|\Rpz(x)|\leq& \beta_{\Delta \sss\Xi,R}^\ssc[0]
	+\beta_{\Delta \sss\Xi,R}^\ssc[1]+
	\sum_{N=2}^\infty \beta_{{\sss \Delta\Xi}}^\ssc[N]\\
	&+\sum_{n=1}^\infty\left(\frac {2d\aab}{1-\aa} \right)^{n+1}
	(\beta_{\sss \Xi^{\iota}}^\ssss[abs])^{n}
	\Big( \beta_{{\sss \Delta \Xi}}^\ssss[abs] \beta_{\sss \Xi^{\iota}}^\ssss[abs]
	+\frac{n+1}{1+\aa}(\tfrac {\aa}{\aab}+\beta_{{\sss \Xi}}^\ssss[abs])
	\left[\beta_{\sss \Delta \Xi^{\iota},\iota}^\ssss[abs]
	+\aa\beta_{{\sss \Delta \Xi^{\iota},0}}^\ssss[abs]\right]\Big)\nnb
	&+\frac {2d\aab} {1-\aa^2} \left((1+\aa)\beta_{{\sss \Delta \Xi}}^\ssss[abs]
	\beta_{\sss \Xi^{\iota}}^\ssss[abs]
	+\beta_{{\sss \Xi}}^\ssss[abs] (\beta_{\sss \Delta \Xi^{\iota},\iota}^\ssss[abs]
	+\aa\beta_{{\sss \Delta \Xi^{\iota},0}}^\ssss[abs])\right)\nnb
	&+\frac {2d\aa} {1-\aa^2}
	\Big( \beta_{\Delta \sss\Xi^\iota,R,I}^\ssc[0]
	+\aa \beta_{\Delta \sss\Xi^\iota,R,II}^\ssc[0]+
	\sum_{N=1}^\infty
	(\beta_{\sss \Delta \Xi^{\iota},\iota}^\ssss[N]+\aa \beta_{{\sss \Delta \Xi^{\iota},0}}^\ssss[N])\Big).\nn
	\end{align}
This proves \refeq{analys-assumed-displacement-1} with $\beta_{\sss \Delta
R,\Phi}$ equal to the right-hand side of \refeq{detailed-Bound-Rpz}.

\paragraph{Step 4: The second bound in \refeq{analys-assumed-displacement-1}.}
We bound the weighted sum of $\Rfz$ in the same way as for $\Rpz$. We require the following three additional bounds:
	\begin{align}
	\lbeq{weightedPsiKBound}
 	\sum_{x\in\Zd}\sum_{\kappa}\|x\|_2^2 \Psi^{\ssc[N],\kappa}(x+\ve[\kappa])
 	\leq& 2d \tfrac {\aab}{\aa} \left( \beta_{{\sss \Delta \Xi}}^\ssc[N]
	+\beta_{{\sss \Xi^{\iota}}}^\ssc[N]\right),\\
 	\lbeq{weightedPiKBound}
	\sum_{x,\iota,\kappa}\|x\|_2^2 \Pi^{\ssc[N],\iota,\kappa}(x+\ve[\kappa])
	\leq &(2d)^2\aab \left( \beta_{{\sss \Delta \Xi^{\iota},0}}^\ssc[N]
	+ \beta_{{\sss \Xi^{\iota}}}^\ssc[N]\right),\\
	\lbeq{weightedPiIKBound}
	\sum_{x,\iota,\kappa}\|x\|_2^2 \Pi^{\ssc[N],\iota,\kappa}(x+\ve[\iota]+\ve[\kappa])
	\leq & (2d)^2\aab \left( \beta_{\sss \Delta \Xi^{\iota},\iota}^\ssc[N]
	+ \beta_{{\sss \Xi^{\iota}}}^\ssc[N]\right).
	\end{align}
Next we explain how to derive these bounds for the example of \refeq{weightedPsiKBound}.

First, we use \refeq{XidominatespsiImproved} and then \refeq{Split-weight-ineq} to obtain
	\begin{align}
	\sum_{x\in\Zd}\sum_{\kappa}\|x\|_2^2 \Psi^{\ssc[N],\kappa}(x+\ve[\kappa])
	\leq& \tfrac {\aab}{\aa}
	\sum_{x\in\Zd}\sum_{\kappa} (\|x+\ve[\kappa]\|_2^2+\|\ve[\kappa]\|_2^2+x_\kappa)
	\Xi^{\ssc[N]}(x+\ve[\kappa]).
	\end{align}
Then, we note that all terms containing a single $x_\kappa$ cancel by the total rotational symmetry of $\Xi$ and $\sum_{\iota}\Xi^\iota$, i.e.,
	\begin{align}
	\sum_{x,\kappa}x_\kappa\Xi^{\ssc[N]}(x+\ve[\kappa])
	=\sum_{x,\kappa,\iota}x_\kappa\Xi^{\ssc[N],\iota}(x+\ve[\kappa])
	=\sum_{x,\kappa,\iota}x_\kappa\Xi^{\ssc[N],\iota}(x+\ve[\iota]+\ve[\kappa])=0.
	\end{align}
Applying the bounds of Assumption \ref{assDiagBoundsCoeff},  we obtain \refeq{weightedPsiKBound}.

For $n\geq 2$, we bound $F_n$ as we have bounded $\Phi_n$ in Step 3. Indeed, first, we use that
	\begin{align}
	\sum_{x\in\Zd} \|x\|_2^2 |F_{n}(x)|\leq &
 	\sum_{\iota_0,\dots,\iota_n}\sum_{x_0,\dots,x_n} \big\|
	\sum_{i=0}^n x_i\big\|_2^2(\frac {\aa}{\aab}+|\Xi(x_0)|) \left( \frac {\aab } {1-\aa^2}\right)^{n+1}\nnb
	&\times \Big(\prod_{s=1}^{n-1} (|\Xi^{\iota_{s-1}}(x_s+\ve[\iota_{s-1}])|
	+\aa|\Xi^{-\iota_{s-1}}(x_s))|\Big)\nnb
	&\times \big( |\Xi^{\iota_{n-1}}(x_n+\ve[\iota_{n-1}]+\ve[\iota_{n}])|
	+\aa|\Xi^{\iota_{n-1}}(x_n+\ve[\iota_{n-1}])| \nnb
	&\qquad\qquad\qquad\qquad
	+\aa|\Xi^{\iota_{n-1}}(x_n+\ve[\iota_n])|+\aa^2|\Xi^{-\iota_{n-1}}(x_n)|\big).
 	\end{align}
We split the weight using the equality in \refeq{Split-weight-ineq}, use that the sums factor and then use the bounds stated above to obtain
	\begin{align}
	\sum_{x\in\Zd} \|x\|_2^2 |F_{n}(x)|\leq &
	\left(\frac {2d\aab}{1-\aa} \right)^{n+1}
	\beta_{{\sss \Delta \Xi}}^\ssss[abs]\left(\beta_{{\sss \Xi^{\iota}}}^\ssss[abs]\right)^n\nnb
	&+ (n-1)\left(\frac {2d\aab}{1-\aa} \right)^{n+1}
	\frac {\left(\beta_{{\sss \Xi^{\iota}}}^\ssss[abs]\right)^{n-1}}{1+\aa}
	(\tfrac {\aa}{\aab}+\beta_{{\sss \Xi}}^\ssss[abs])
	(\beta_{\sss \Delta \Xi^{\iota},\iota}^\ssss[abs]
	+\aa\beta_{{\sss \Delta \Xi^{\iota}},0}^\ssss[abs])\nnb
	& +\left(\frac {2d\aab}{1-\aa} \right)^{n+1}
	\frac {\left(\beta_{{\sss \Xi^{\iota}}}^\ssss[abs]\right)^{n-1} }{1+\aa}
	(\tfrac {\aa}{\aab}+\beta_{{\sss \Xi}}^\ssss[abs])
	\left(\beta_{\sss \Delta \Xi^{\iota},\iota}^\ssss[abs]
	+\aa \beta_{{\sss \Delta \Xi^{\iota},0}}^\ssss[abs]+\beta_{{\sss \Xi^{\iota}}}^\ssss[abs]\right).
 	\end{align}
The additional contribution $\Rfz$ (recall \refeq{detailed-Def-Rfz}) are bounded in a straightforward manner and we obtain the bound
	\begin{align}
	\sum_{x\in\Zd} \|x\|_2^2 |\Rfz(x)|\leq
	&\sum_{n=2}^\infty\left(\frac {2d\aab}{1-\aa} \right)^{n+1}
	\beta_{{\sss \Delta \Xi}}^\ssss[abs]
	\left(\beta_{{\sss \Xi^{\iota}}}^\ssss[abs]\right)^n
	\lbeq{analys-assumed-displacement-2-computed}\\
 	&+\sum_{n=2}^\infty
	(n-1)\left(\frac {2d\aab}{1-\aa} \right)^{n+1}
	\frac {\left(\beta_{{\sss \Xi^{\iota}}}^\ssss[abs]\right)^{n-1}}{1+\aa}
	(\tfrac {\aa}{\aab}+\beta_{{\sss \Xi}}^\ssss[abs])
	(\beta_{\sss \Delta \Xi^{\iota},\iota}^\ssss[abs]+\aa\beta_{{\sss \Delta\Xi^{\iota}},0}^\ssss[abs])
	\nnb
	&+\sum_{n=2}^\infty
	\left(\frac {2d\aab}{1-\aa} \right)^{n+1}\frac {\left(\beta_{{\sss \Xi^{\iota}}}^\ssss[abs]\right)^{n-1}}
	{1+\aa} (\tfrac {\aa}{\aab}+\beta_{{\sss \Xi}}^\ssss[abs])
	\left(\beta_{\sss \Delta \Xi^{\iota},\iota}^\ssss[abs]
	+\aa \beta_{{\sss \Delta \Xi^{\iota},0}}^\ssss[abs]+\beta_{{\sss \Xi^{\iota}}}^\ssss[abs]\right)
	\nnb
	&+\frac {2d\aa} {1-\aa^2}  \Big[
 	\sum_{N\in \{0,1\}}
 	\left(\beta_{\Delta \sss\Psi,R,I}^\ssc[N]+\beta_{\Delta \sss\Psi,R,II}^\ssc[N]\right)+
	\tfrac {\aab}{\aa} \sum_{N=2}^\infty
	\left(\beta_{{\sss \Delta \Xi}}^\ssc[N]+\beta_{\sss \Xi}^\ssc[N]
	+\aa\beta_{{\sss \Delta \Xi}}^\ssc[N]\right)\Big]\nnb
	&+\frac {\aa} {(1-\aa^2)^2} \Big(\beta_{\Delta\sss\Pi,R}^\ssc[0]
	+(2d)^2\aab \sum_{N=1}^\infty\left( \beta_{\sss \Delta \Xi^{\iota},\iota}^\ssc[N]
	+ \beta_{{\sss \Xi^{\iota}}}^\ssc[N]\right)\Big)\nnb
	&+\frac {(2d)^2\aa^2\aab} {(1-\aa^2)^2}
	\left( \beta_{\sss \Delta \Xi^{\iota},\iota}^\ssss[abs]
	+\beta_{{\sss \Delta \Xi^{\iota},0}}^\ssss[abs]
	+ \beta_{{\sss \Xi^{\iota}}}^\ssss[abs]+\aa\beta_{{\sss \Delta \Xi^{\iota}},0}^\ssss[abs]\right)\nnb
	&+\frac {(2d)^2\aab^2} {(1-\aa)^2} \beta_{{\sss \Delta \Xi}}^\ssss[abs]
	\beta_{{\sss \Xi^{\iota}}}^\ssss[abs]
	+\frac{(2d)^2\aab^2} {(1-\aa^2)(1-\aa)} \beta_{{\sss \Xi}}^\ssss[abs]
	\left(\beta_{\sss \Delta \Xi^{\iota},\iota}^\ssss[abs]
	+\aa \beta_{{\sss \Delta \Xi^{\iota},0}}^\ssss[abs]
	+\beta_{{\sss \Xi^{\iota}}}^\ssss[abs]\right).\nn
	\end{align}
We complete Step 4 by defining $\betadeltaRfz$ as the right-hand of \refeq{analys-assumed-displacement-2-computed}.

\paragraph{Step 5: The bounds in \refeq{analys-assumed-displacement-3}.}
Now we compute a lower bound on $\hRfz(0)-\hRfz(k)$:
\begin{align}
\lbeq{RFlower-function}
  \hRfz(0)-\hRfz(k)=\sum_{x} \Rfz(x)[1-\cos(k\cdot x)]
\end{align}
which could also be negative.  We could apply Lemma \ref{lemmaFdifferenceToX} and bound \refeq{RFlower-function} from below by $-\betadeltaRfz[1-\hat D(k)]$.
Since this bound plays a central role, we prefer to use a better bounds. For this, we note that $[1-\cos(k\cdot x)]\geq 0$ for all $k$ and $x$.
So we can create a lower bound by only consider the negative part of the function $\Rfz$.\\
We are not able to identify the points $x$ at which $\Rfz(x)$ is negative. However, we can decompose all summands defining $\hRfz$ in \refeq{detailed-Def-Rfz}
into a negative and a positive part as the coefficients were defined via alternating sequences.
In this ways, we create the lower bound as follows:
\begin{align}
	-\frac {\aa} {(1-\aa^2)^2} \sum_{N=1}^\infty &\sum_{\iota_0,\iota_1} (-1)^N \left[ \hat \Pi^{\ssc[N],\iota_{0},\iota_{1}}(0)-
	\e^{-\ii (k_{\iota_{0}}+k_{\iota_{1}})}\hat \Pi^{\ssc[N],\iota_{0},\iota_{1}}(k)\right]
\nnb
	&\geq -\frac {\aa} {(1-\aa^2)^2} \sum_{N=1}^\infty \sum_{\iota_0,\iota_1} \left[ \hat \Pi^{\ssc[2N],\iota_{0},\iota_{1}}(0)-\e^{-\ii (k_{\iota_{0}}+k_{\iota_{1}})}\hat \Pi^{\ssc[2N],\iota_{0},\iota_{1}}(k)\right]\nnb
	&\geq -\frac {(2d)^2\aa\aab} {(1-\aa^2)^2} \sum_{N=1}^\infty
	\left(\beta_{\sss \Delta \Xi^{\iota},\iota}^\ssc[2N]+2d \beta_{{\sss \Xi^{\iota}}}^\ssc[2N]\right)[1-\hat D(k)],
\end{align}
where we applied Lemma \ref{lemmaFdifferenceToX} and \refeq{weightedPiIKBound} only in the very last step. We use this idea to bound all terms in \refeq{detailed-Def-Rfz}, except the minor term $\sum_{n=2}^\infty \hat F_n(k)$ and in this way obtain
\begin{align}
	-  \frac{\hRfz(0)-\hRfz(k)}{1-\hat D(k)}
	\leq&\sum_{n=2}^\infty\left(\frac {2d\aab}{1-\aa} \right)^{n+1}  \beta_{{\sss \Delta \Xi}}^\ssss[abs]
	\left(\beta_{{\sss \Xi^{\iota}}}^\ssss[abs]\right)^n
	\lbeq{analys-assumed-displacement-3-computed}\\
	 &+\sum_{n=2}^\infty (n-1)\left(\frac {2d\aab}{1-\aa} \right)^{n+1}
	\frac {\left(\beta_{{\sss \Xi^{\iota}}}^\ssss[abs]\right)^{n-1}}{1+\aa}
	(\tfrac {\aa}{\aab}+\beta_{{\sss \Xi}}^\ssss[abs])
	(\beta_{\sss \Delta \Xi^{\iota},\iota}^\ssss[abs]+\aa\beta_{{\sss \Delta\Xi^{\iota}},0}^\ssss[abs])
 	\nnb
	& +\sum_{n=2}^\infty
	\left(\frac {2d\aab}{1-\aa} \right)^{n+1}
	\frac {\left(\beta_{{\sss \Xi^{\iota}}}^\ssss[abs]\right)^{n-1} }{1+\aa}
	(\tfrac {\aa}{\aab}+\beta_{{\sss \Xi}}^\ssss[abs])
	\left(
	\beta_{\sss \Delta \Xi^{\iota},\iota}^\ssss[abs]
	+\aa \beta_{{\sss \Delta\Xi^{\iota},0}}^\ssss[abs]+\beta_{{\sss \Xi^{\iota}}}^\ssss[abs]\right)
	\nnb
	&+\frac {2d\aa} {1-\aa^2}
	\left[\beta_{\Delta\sss\Psi,R,I}^\ssc[1]+\aa\beta_{\Delta\sss\Psi,R,II}^\ssc[0]+
	\tfrac {\aab}{\aa} \sum_{N=1}^\infty \left(\beta_{{\sss \Delta \Xi}}^\ssc[2N+1]
	+ \beta_{{\sss \Xi}}^\ssc[2N+1]+\aa\beta_{{\sss \Delta \Xi}}^\ssc[2N]\right)\right]\nnb
 	&+\frac {\aa} {(1-\aa^2)^2} \left(
 	\beta_{\Delta\sss\Pi,R}^\ssc[0]
	+(2d)^2\aab\sum_{N=1}^\infty\left( \beta_{\sss \Delta \Xi^{\iota},\iota}^\ssc[2N]
	+ \beta_{{\sss \Xi^{\iota}}}^\ssc[2N]\right)\right)\nnb
	&+\frac {(2d\aa)^2\aab} {(1-\aa^2)^2}
	\left( \beta_{\sss \Delta \Xi^{\iota},\iota}^\ssss[odd]
	+\beta_{{\sss \Delta \Xi^{\iota},0}}^\ssss[odd]
	+\beta_{{\sss \Xi^{\iota}}}^\ssss[odd]+\aa\beta_{{\sss \Delta
	\Xi^{\iota}},0}^\ssss[even]\right)\nnb
	&+\frac {(2d\aab)^2} {(1-\aa^2)^2}
	\left(\beta_{{\sss \Delta \Xi}}^\ssss[odd]\beta_{{\sss \Xi^{\iota}}}^\ssss[odd](1+\aa^2)
	+2\aa\beta_{{\sss \Delta \Xi}}^\ssss[even]\beta_{{\sss \Xi^{\iota}}}^\ssss[even]\right)\nnb
	&+\frac{(2d\aab)^2} {(1-\aa^2)^2} \beta_{{\sss \Xi}}^\ssss[odd]
	\left(\beta_{\sss \Delta \Xi^{\iota},\iota}^\ssss[odd]+\beta_{{\sss \Xi^{\iota}}}^\ssss[odd]
	+\aa\beta_{{\sss \Delta \Xi^{\iota},0}}^\ssss[even]+\aa\beta_{{\sss \Xi^{\iota}}}^\ssss[even]
	+\aa\beta_{\sss \Delta \Xi^{\iota},\iota}^\ssss[even]+\aa^2\beta_{{\sss \Delta
	\Xi^{\iota}},0}^\ssss[odd]\right)\nnb
	&+\frac{(2d\aab)^2} {(1-\aa^2)^2} \beta_{{\sss \Xi}}^\ssss[even]
	\left(\beta_{\sss \Delta \Xi^{\iota},\iota}^\ssss[even]+\beta_{{\sss \Xi^{\iota}}}^\ssss[even]
	+\aa\beta_{{\sss \Delta \Xi^{\iota},0}}^\ssss[odd]+\aa\beta_{{\sss \Xi^{\iota}}}^\ssss[odd]
	+\aa\beta_{\sss \Delta \Xi^{\iota},\iota}^\ssss[odd]+\aa^2\beta_{{\sss \Delta \Xi^{\iota},0}}^\ssss[even]
	\right).\nn
\end{align}
We modified the desired inequality \refeq{analys-assumed-displacement-3} here slightly for better readability.
This proves the lower bound in  \refeq{analys-assumed-displacement-3} with $-\betadeltaRfzlow$
equal to the right-hand side of \refeq{analys-assumed-displacement-3-computed}.


{\small
\bibliographystyle{plain}
\bibliography{NoBLEBiB}
}
\end{document}